\documentclass[final]{amsart}
\usepackage[usenames,dvipsnames]{xcolor}
\definecolor{cite}{HTML}{0851A6}
\definecolor{url}{HTML}{0851A6}
\definecolor{link}{HTML}{8F0C00}

\usepackage[colorlinks=true, linkcolor=link, citecolor = cite, linktocpage, backref=page]{hyperref}

\usepackage[shortalphabetic]{amsrefs}

\usepackage{fancyhdr, amsmath, amsthm, enumerate, microtype}
\usepackage[all, cmtip]{xy}

\usepackage{todonotes}

\usepackage[bitstream-charter]{mathdesign}
\usepackage[T1]{fontenc}
\usepackage{stmaryrd}

\DeclareMathAlphabet{\eur}{U}{zeus}{m}{n}
\newcommand{\matheur}[1]{\eur{#1}}

\theoremstyle{plain}
\newtheorem{prop}[subsubsection]{Proposition}
\newtheorem{lem}[subsubsection]{Lemma}
\newtheorem{cor}[subsubsection]{Corollary}
\newtheorem{thm}[subsubsection]{Theorem}
\newtheorem*{thm*}{Theorem}

\theoremstyle{definition}
\newtheorem{defn}[subsubsection]{Definition}
\newtheorem{notation}[subsubsection]{Notation}

\theoremstyle{remark}
\newtheorem{rmk}[subsubsection]{Remark}
\newtheorem{expl}[subsubsection]{Example}

\newcommand{\teq}{\addtocounter{subsubsection}{1}\tag{\thesubsubsection}}

\newcommand{\arrdisplacement}{0.36ex}
\newcommand{\arrdisplacementsp}{0.72ex}

\DeclareMathOperator{\add}{add}
\DeclareMathOperator{\addCoFil}{addCoFil}
\DeclareMathOperator{\addFil}{addFil}
\DeclareMathOperator{\addUnit}{addUnit}
\DeclareMathOperator{\assgr}{ass-gr}
\DeclareMathOperator{\aug}{au}
\DeclareMathOperator{\BarO}{Bar}

\DeclareMathOperator{\car}{char}

\DeclareMathOperator{\coChev}{coChev}

\DeclareMathOperator{\coFib}{coFib}
\DeclareMathOperator{\coFil}{coFil}
\DeclareMathOperator{\coFree}{coFree}
\DeclareMathOperator{\Conf}{Conf}
\DeclareMathOperator{\coLie}{coLie}
\DeclareMathOperator{\coLieshriek}{\coLie^!}

\DeclareMathOperator*{\colim}{colim}
\DeclareMathOperator{\const}{const}
\DeclareMathOperator{\coOp}{coOp}
\DeclareMathOperator{\coPrim}{coPrim}
\DeclareMathOperator{\cotriv}{cotriv}

\DeclareMathOperator{\Cinfty}{C_\infty}
\DeclareMathOperator{\coBarP}{coBar}
\DeclareMathOperator{\ComAlg}{ComAlg}
\DeclareMathOperator{\ComAlgshriek}{\ComAlg^!}
\DeclareMathOperator{\ComAlgstar}{\ComAlg^\star}
\DeclareMathOperator{\ComCoAlg}{ComCoAlg}

\newcommand{\cont}{\mathrm{cont}}
\DeclareMathOperator{\Corr}{Corr}
\DeclareMathOperator{\decay}{decay}
\DeclareMathOperator{\DeCat}{DeCat}

\newcommand{\DGCat}{\mathrm{DGCat}}
\newcommand{\DGCatpres}{\DGCat_\pres}

\newcommand{\disj}{\mathrm{disj}}

\newcommand{\enh}{\mathrm{enh}}
\DeclareMathOperator{\ev}{ev}

\newcommand{\etale}{\'etale}
\DeclareMathOperator{\Fact}{Fact}
\DeclareMathOperator{\Factstar}{\Fact^\star}
\DeclareMathOperator{\Fib}{Fib}
\newcommand{\Fil}{\mathrm{Fil}}
\DeclareMathOperator{\filToGr}{\Fil\to\gr}

\newcommand{\Fq}{\mathbb{F}_q}

\DeclareMathOperator{\Free}{Free}
\DeclareMathOperator{\Frob}{Frob}
\newcommand{\fSet}{\mathrm{fSet}}
\DeclareMathOperator{\Fun}{Fun}

\DeclareMathOperator{\Gr}{Gr}
\newcommand{\gr}{\mathrm{gr}}

\DeclareMathOperator{\grToFiltriv}{\gr\overset{0}{\to}\Fil}
\newcommand{\grun}{\underline{\mathrm{un}}}
\newcommand{\graug}{\underline{\mathrm{au}}}
\newcommand{\Ho}{\mathrm{H}}

\newcommand{\id}{\mathrm{id}}

\DeclareMathOperator{\Linfty}{L_\infty}
\DeclareMathOperator{\Lie}{Lie}

\newcommand{\Mod}{\mathrm{Mod}}
\newcommand{\munit}{\mathbb{1}}
\DeclareMathOperator{\oblv}{oblv}
\newcommand{\op}{\mathrm{op}}
\DeclareMathOperator{\Op}{Op}
\newcommand{\open}{\text{open}}

\newcommand{\otimeshat}{\hat{\otimes}}
\newcommand{\otimesshriek}{\overset{!}{\otimes}}
\newcommand{\otimesstar}{\otimes^\star}
\DeclareMathOperator{\Palg}{-alg}
\DeclareMathOperator{\Pcoalg}{-coalg}

\DeclareMathOperator{\Poinc}{Poinc}
\DeclareMathOperator{\PoincVir}{\Poinc^\mathrm{vir}}
\newcommand{\Poincare}{Poincar\'e}
\newcommand{\pres}{\mathrm{pres}}
\newcommand{\PreStk}{\mathrm{PreStk}}

\newcommand{\pt}{\mathrm{pt}}

\newcommand{\Ql}{\mathbb{Q}_\ell}
\newcommand{\Qlbar}{\lbar{\mathbb{Q}}_\ell}
\DeclareMathOperator{\Quot}{Quot}
\DeclareMathOperator{\Ran}{Ran}
\DeclareMathOperator{\RanOpen}{\overset{\circ}{\Ran}}

\DeclareMathOperator{\Rees}{Rees}
\newcommand{\res}{\mathrm{res}}
\DeclareMathOperator{\Sym}{Sym}

\newcommand{\Sch}{\mathrm{Sch}}

\DeclareMathOperator{\Spec}{Spec}

\DeclareMathOperator{\Shv}{Shv}

\newcommand{\Spc}{\mathrm{Spc}}
\DeclareMathOperator{\stab}{stab}

\newcommand{\surj}{\mathrm{surj}}

\DeclareMathOperator{\tr}{tr}
\DeclareMathOperator{\triv}{triv}
\newcommand{\un}{\mathrm{un}}
\newcommand{\union}{\mathrm{union}}
\newcommand{\unit}{\mathbf{1}}
\newcommand{\Vect}{\mathrm{Vect}}

\newcommand{\ardis}{\ar@<\arrdisplacement>}
\newcommand{\ardissp}{\ar@<\arrdisplacementsp>}

\newcommand{\lbar}[1]{\overline{#1}}
\newcommand{\llrrb}[1]{\llbracket#1\rrbracket}

\newcommand{\oversetsupscript}[3]{\overset{#2}{#1}{}^{#3}}

\newcommand{\gConf}[2]{Z^{#1}_{#2}}

\newcommand{\gConfOpen}[1]{\Conf_{#1}}

\newcommand{\alg}[2]{\matheur{A}_{#1, #2}}
\newcommand{\algb}[2]{\lbar{\matheur{A}}_{#1, #2}}
\newcommand{\liealg}[2]{\mathfrak{a}_{#1, #2}}

\newcommand{\algOr}[2]{\mathsf{A}_{#1, #2}}
\newcommand{\freeAlgOr}[1]{\mathsf{A}_{#1, \infty}}
\newcommand{\liealgOr}[2]{\mathsf{a}_{#1, #2}}
\newcommand{\trivLieAlgOr}[1]{\mathsf{a}_{#1, \infty}}

\newcommand{\multgr}[1]{\mathbb{Z}^{#1}_{\geq 0}}
\newcommand{\multgrplus}[1]{\mathbb{Z}^{#1}_{\geq 0, +}}
\newcommand{\multfil}[1]{\mathbb{Z}^{#1, \to}_{\geq 0}}
\newcommand{\multfilplus}[1]{\mathbb{Z}^{#1, \to}_{\geq 0, +}}

\newcommand{\graded}{\mathbb{Z}_{\geq 0}}
\newcommand{\gradedplus}{\mathbb{Z}_{\geq 0, +}}
\newcommand{\filtered}{\mathbb{Z}^{\to}_{\geq 0}}
\newcommand{\filteredplus}{\mathbb{Z}^{\to}_{\geq 0, +}}

\newcommand{\sheerRight}{\overrightarrow{\mathrm{sh}}}
\newcommand{\sheerLeft}{\overleftarrow{\mathrm{sh}}}

\newcommand{\sOmega}{\vec{\omega}}

\title[Homological stability and densities of generalized configuration spaces]{Homological stability and densities of generalized configuration spaces}
\author{Quoc P. Ho}
\address{Institute of Science and Technology Austria, Klosterneuburg, Austria}
\email{qho@ist.ac.at}
\date{\today}


\begin{document}
\begin{abstract}
We prove that the factorization homologies of a scheme with coefficients in truncated polynomial algebras compute the cohomologies of its generalized configuration spaces. Using Koszul duality between commutative algebras and Lie algebras, we obtain new expressions for the cohomologies of the latter. As a consequence, we obtain a uniform and conceptual approach for treating homological stability, homological densities, and arithmetic densities of generalized configuration spaces. Our results categorify, generalize, and in fact provide a conceptual understanding of the coincidences appearing in the work~\cite{farb_coincidences_2019} of Farb-Wolfson-Wood. Our computation of the stable homological densities also yields rational homotopy types which answer a question posed by Vakil-Wood in~\cite{vakil_discriminants_2015}. Our approach hinges on the study of homological stability of cohomological Chevalley complexes, which is of independent interest.
\end{abstract}

\maketitle
\tableofcontents

\section{Introduction}

\subsection{Motivation}
For any connected curve $X$ over $\Fq$, consider the symmetric powers and the (unordered) configuration spaces of $X$
\[
	\Sym^d(X) = X^d/\Sigma_d, \qquad \Conf_d(X) = \oversetsupscript{X}{\circ}{d}/\Sigma_d \overset{\open}{\subset} X^d/\Sigma_d = \Sym^d X,
\]
where $\Sigma_d$ denotes the group of permutations on $d$ letters and where in the latter, $\oversetsupscript{X}{\circ}{n}$ is the open subset of $X^n$ where we require all the points to be distinct. Moreover, let
\[
	(\Sym^{d_1} X \times \Sym^{d_2} X)_{\disj} \overset{\open}{\subset} \Sym^{d_1} X \times \Sym^{d_2} X
\]
be an open subscheme such that the collections of points in the two factors are disjoint. Then, it is well-known, at least in the case where $X=\mathbb{A}^1$, that we have the following coincidence of ``arithmetic densities'' (see~\cite{morrison_probability_nodate})
\[
	\lim_{d\to \infty} \frac{|(\Conf_d X)(\Fq)|}{|(\Sym^d X)(\Fq)|} = \lim_{(d_1, d_2)\to \infty} \frac{|((\Sym^{d_1} X \times \Sym^{d_2} X)_\disj)(\Fq)|}{|(\Sym^{d_1} X \times \Sym^{d_2} X)(\Fq)|} = \zeta_X(2)^{-1}.
\]

\subsubsection{}
\label{subsubsec:intro_define_spaces}
More generally, for any variety $X$, a pair of positive integers $m, n$, and an $m$-tuple of numbers
\[
	\mathbf{d} = (d_1, d_2, \dots, d_m) \in \multgr{m},
\]
let
\[
	\Sym^{\mathbf{d}}(X) = \prod_k \Sym^{d_k} (X),
\]
and $\gConf{\mathbf{d}}{n}(X) \subset \Sym^\mathbf{d}(X)$ consisting of sets $D$ of $|\mathbf{d}|$ (not necessarily distinct) points in $X$ such that:
\begin{enumerate}[(i)]
	\item precisely $d_k$ of the points in $D$ are labeled with the ``color'' $k$, and
	\item no point of $X$ has multiplicity of at least $n$ for every color, or, equivalently, for each point $x$ of $X$, there exists a color whose multiplicity at $x$ is less than $n$.
\end{enumerate}

For example, when $m=1$ and $\mathbf{d} = (d)$, we have
\begin{align*}
	\gConf{\mathbf{d}}{\infty} (X) &= \gConf{d}{\infty} (X) = \Sym^d X, \\
	\gConf{\mathbf{d}}{2}(X) &= \gConf{d}{2}(X) = \Conf_d X.
\end{align*}
Similarly, when $m=2$ and $\mathbf{d} = (d_1, d_2)$, we have
\begin{align*}
	\gConf{\mathbf{d}}{\infty}(X) &= \Sym^{d_1} X \times \Sym^{d_2} X,\\
	\gConf{\mathbf{d}}{1}(X) &= (\Sym^{d_1} X \times \Sym^{d_2} X)_\disj.
\end{align*}

\subsubsection{} We have the following result (see, for example,~\cite{farb_coincidences_2019}*{Thm. 2.3}).

\begin{thm} \label{thm:arithmetic_densities}
We have the following equality
\[
	\lim_{\mathbf{d}\to \infty} \frac{|\gConf{\mathbf{d}}{n}(X)(\Fq)|}{|\Sym^\mathbf{d}(X)(\Fq)|} = \zeta_X(mn\dim X)^{-1}
\]
for any connected variety $X$ and any positive integers $m, n$. Here, $m$ denotes the number of colors and $\lim_{\mathbf{d} \to \infty}$ means ``as all $d_i \to \infty$,'' at any rate.
\end{thm}

\subsubsection{}
Observing the coincidence in the arithmetic densities appearing in Theorem~\ref{thm:arithmetic_densities}, i.e. the written limit only depends on the product $mn$ and $X$ itself, and taking into account similar phenomena observed by Segal in~\cite{segal_topology_1979}, Farb, Wolfson, and Wood~\cite{farb_coincidences_2019} initiated the studies of ``homological densities'' of these spaces in the case where $X$ is a smooth manifold or a complex variety. More specifically, they asked if the following limit of the quotient between the \Poincare{} polynomials
\[
	\lim_{\mathbf{d} \to \infty} \frac{P_{\gConf{\mathbf{d}}{n}(X)}(t)}{P_{\Sym^\mathbf{d} (X)}(t)} \in \mathbb{Z}\llrrb{t}
\]
also has the same kind of coincidence as in the case of $\zeta$-functions. Here, the \Poincare{} polynomials are with respect to Betti cohomology with rational coefficients.

In many cases, the answer is positive, and we have the following result.

\begin{thm}[\cite{farb_coincidences_2019}*{Thm. 1.2}]
\label{thm:FWW_poinc_poly_coincidences}
Let X be a connected orientable smooth manifold with finite dimensional cohomology groups. Then the limit
\[
	\lim_{\mathbf{d} \to \infty} \frac{P_{\gConf{\mathbf{d}}{n}(X)}(t)}{P_{\Sym^\mathbf{d} (X)}(t)} \in \mathbb{Z}\llrrb{t}
\]
exists, and depends only on the product $mn$, the Betti numbers of $X$, and $\dim X$, when the cup-product of any $mn$ compactly supported cohomology classes vanishes.
\end{thm}

\subsubsection{} The proof of this theorem uses Bj\"orner-Wachs theory of lexicographic shellability from algebraic combinatorics as the main computational input. The vanishing of the cup products is there to ensure that certain spectral sequences degenerate already at page 2 so we can have a good handle on the \Poincare{} polynomials, since these are not motivic.

Using essentially the same computational input,~\cite{farb_coincidences_2019} went on to show that similar coincidences hold for quotients of Euler characteristics as well as quotients of Hodge-Deligne polynomials (i.e. Euler characteristic densities and Hodge-Deligne densities). Unlike the case of \Poincare{} polynomials, however, these objects are motivic, which obviates the need for the spectral sequences involved to degenerate. As a result, these coincidences hold true without this extra vanishing condition.

\subsection{The goal of this paper}
Since these results are discovered as combinatorial coincidences, it is natural to ask for a conceptual explanation of these coincidences, and in fact, we learned about this question from Farb and Wolfson. Using factorization homology,\footnote{Strictly speaking, we use factorization cohomology. This is, however, more of a minor technical point, which we will ignore in the introduction.} we show that hidden behind the combinatorial complexities is a picture of striking simplicity: the cohomology of generalized configuration spaces are controlled in a precise sense by truncated polynomial algebras. As a result, we are able to provide a uniform and conceptual framework for understanding \emph{homological stability} as well as \emph{coincidences in arithmetic and homological densities} of generalized configuration spaces. This categorifies and in fact generalizes known results in two directions:
\begin{enumerate}[(i)]
	\item $X$ is only required to be irreducible (rather than smooth), and in place of cohomology, we work with Borel-Moore homology $C^*(X, \omega_X[-2\dim X](-\dim X)) = C^*(X, \sOmega_X)$ (see Notation~\ref{notation:sOmega} for the definition of $\sOmega$), which specializes to usual cohomology when $X$ is smooth.
	\item Enlarge the class of configurations allowed (more on this below). 
\end{enumerate}

To reach this final goal, along the way, we prove a general algebraic statement about homological stability and stable homology of the cohomology of $\coLie$-coalgebras,\footnote{Due to technical convenience, we will work most of the time with $\coLie$-coalgebras rather than $\Lie$-algebras. However, since all the concrete $\coLie$-coalgebras that appear in the paper are finite dimensional, the reader should feel free to instead think about their duals, which are $\Lie$-algebras. Note also that since everything is derived in our convention, what we call $\Lie$-algebras are classically called dg-Lie algebras.} which is of independent interest, and which forms the technical heart of the paper. The geometric statements about stability and coincidences are then obtained by applying these general results to certain $\coLie$-coalgebras produced using techniques from factorization homology.

We remark that most of the results in this paper are of a general nature. Moreover, once the general formalism has been laid down, results about the cohomology of these generalized configuration spaces follow naturally, requiring almost no computation. Indeed, configuration spaces do not appear until the last two sections, \S\ref{sec:fact_hom} and~\S\ref{sec:cohomology_of_Z_m_n}, and all the computations, which are simple and elementary in nature, are done in~\S\ref{sec:cohomology_of_Z_m_n}.

\begin{rmk}
Even though the results are stated in the language of algebraic geometry and $\ell$-adic cohomology in this paper, the techniques and cohomological results appearing in this work apply equally well to other settings where a sheaf theory with a six-functor formalism is available. For example, if one is interested in the topological setting, the formalism of~\cite{schnurer_six_2018} can be used to replace $\ell$-adic cohomology.

In this paper, however, we decide to work in the context of algebraic geometry to fix ideas and moreover, our main goal is to demonstrate the link between homological and arithmetic densities.
\end{rmk}

\subsection{Prerequisites and guides to the literature}
For the reader's convenience, we include a quick review of the necessary background as well as pointers to the existing literature in \S\ref{sec:prelim}. The readers who are unfamiliar with the language and notation used in the introduction are encouraged to take a quick look there before returning to the current section.

\subsection{An outline of our strategy and results} We will now give a more precise outline of our results. Technically, our work can be divided into two main parts, factorization homology and Koszul duality. The first brings generalized configuration spaces into the purview of factorization homology. The second provides a way to compute factorization homology as cohomological Chevalley complexes of $\coLie$-coalgebras and hence serves as the technical tool to process the output produced by the first. The two aspects are technically quite different. The readers can thus read either one first, and refer to the other whenever necessary.

Throughout, we fix an irreducible scheme $X$ of dimension $d$ over an algebraically closed field. When we talk about Frobenius actions, we assume that $X$ is over a field of characteristic $p$, and moreover, it is a pullback of a scheme defined over $\Fq$. When we talk about point counts, we refer to the point counts of the original object over $\Fq$.

\subsubsection{Koszul duality and homological stability of cohomological Chevalley complexes}
We start with the Koszul duality aspect of the approach. Let $\mathfrak{a}$ be a $\multgrplus{m}$-graded $\coLie$-coalgebra. We want to ask when $\coChev^{\un} \mathfrak{a}$, the unital cohomological Chevalley complex, satisfies homological stability. The question does not make sense as is, since to talk about homological stability, we need to first have maps between different graded-degrees so that we can compare them. A priori, $\coChev^{\un} \mathfrak{a}$ is only a $\multgr{m}$-graded object.

Suppose that $\coChev^{\un} \mathfrak{a}$ admits the structure of an algebra over the graded polynomial algebra $\Lambda[\multgr{m}] \simeq \Lambda[x_1 \dots, x_m]$ where $\Lambda$ is a fixed field of characteristic $0$ (see also~\S\ref{subsubsec:convention_categories} for our conventions) and where $x_i$ lives in graded degree $(0, \dots, 0, 1, 0, \dots, 0)$ with $1$ being at the $i$-th place. Then, using the algebra action, we obtain maps between the different gradings, which put us in the setting where questions about homological stability make sense. Using Koszul duality, we can translate this condition to a condition on the $\coLie$-coalgebra $\mathfrak{a}$. This allows us to single out a class of $\coLie$-coalgebras, which we call strongly unital, where we can formulate and establish homological stability. The main result we prove in this setting is the following

\begin{thm}[Theorem~\ref{thm:homological_stability_coChev}]
Let $\mathfrak{a} \in \coLie^\un(\Vect^{\multgrplus{m}})$ be a strongly unital $\coLie$-coalgebra. Fix $c_0 \in \mathbb{N}$, $1\leq k\leq m$ and suppose that there exist $s, s_1, \dots, s_{k-1}, s_{k+1}, \dots, s_m \in \mathbb{N}$ such that for any non-zero $v\in \Ho^*(\Quot_\un(\mathfrak{a}))$\footnote{$\Quot_\un(\mathfrak{a})$ denotes the $\coLie$-coalgebra obtained from removing the ``unit'' of $\mathfrak{a}$. See~\S\ref{subsubsec:quot_un_for_coLie} for the precise definition.} of cohomological degree $c-1\leq c_0$ and graded degree $\mathbf{d} \in \multgrplus{m}$, we have
\[
	\mathbf{d}_k \leq sc + \sum_{i\neq k} s_i \mathbf{d}_i,
\]
where $\mathbf{d}_i$ denotes the $i$-th coordinate of $\mathbf{d}$.

Then, for all $c\leq c_0$ and $\mathbf{d} \in \multgrplus{m}$ such that
\[
	\mathbf{d}_k \geq sc + \sum_{i\neq k} s_i \mathbf{d}_i,
\]
the following map is an equivalence
\[
	\tau_{\leq c}(\coChev^{\un} \mathfrak{a})_{\mathbf{d}} \to \tau_{\leq c} (\coChev^{\un} \mathfrak{a})_{\mathbf{d} + \unit_k},
\]
and the following map is injective
\[
	\Ho^{c+1}(\coChev^{\un} \mathfrak{a})_\mathbf{d} \to \Ho^{c+1}(\coChev^{\un} \mathfrak{a})_{\mathbf{d}+\unit_k}.
\]
Here, $\coChev^{\un}$ denotes the functor of taking the unital cohomological complex, see~\S\ref{subsubsec:defn_coChev_et_al} and~\S\ref{subsec:unital_vs_nonunital}.
\end{thm}

In particular, when $\mathfrak{a}$ is strongly unital and finite dimensional, we have homological stability (Corollary~\ref{cor:qualitative_homological_stab_fin_dim} and Remark~\ref{rmk:can_do_better_than_finite_dim}).

\subsubsection{}
When $\mathfrak{a}$ is abelian, $\coChev^{\un} \mathfrak{a} = \Sym(\mathfrak{a}[-1])$, and the result is immediate. To reduce to this case, in Corollary~\S\ref{cor:coChev_as_limit}, we establish a technical fact about the functor $\coChev$ which says that it can be written as a sequential limit, whose successive fibers form $\Sym (\mathfrak{a}[-1])$. In fact, we prove a similar result for $\coBarP_{\matheur{P}}$ for any co-operad $\matheur{P}$ in Proposition~\ref{prop:coBar_as_limit}, which specializes to the case of $\coChev$ when $\matheur{P} = \coLie$. In~\S\ref{subsec:homological_stability_of_a_limit}, we establish the relation between homological stability of a sequential limit and that of the associated fiber, which allows us to make the reduction.

\subsubsection{}
The use of Koszul duality is crucial not only in the understanding of homological stability, but also in extracting information out of an algebra (or more relevant to us, quotients of algebras) via decategorification. Roughly speaking, the functor $\coChev$ exhibits a multiplicative nature, which, in many cases, manifests itself via products/quotients of the decategorification. This is why it is useful for our purposes.

In~\S\ref{subsec:decategorification_coChev}, we establish several results in this direction which will be used later to extract information out of factorization cohomology. We learned this picture from~\cite{gaitsgory_weils_2014}, of which the results in~\S\ref{subsec:decategorification_coChev} are routine adaptations.

\subsubsection{Factorization cohomology and cohomology of configuration spaces}
We will now discuss the factorization cohomology aspect of the approach. In the topological setting, this is also called the topological chiral (co)homology, which is a vast generalization of Hochschild homology.

\subsubsection{}
\label{subsubsec:intro_define_general_spaces}
Let us start with the definition of the generalized configuration spaces associated to $X$ that we will study. Let $m$ be a positive integer and $\mathbf{n}$ an $m$-tuple of positive integers such that
\[
	|\mathbf{n}| = \sum_{k=1}^m \mathbf{n}_k \geq 2.
\]
Similarly to~\S\ref{subsubsec:intro_define_spaces}, for each $\mathbf{d} \in \multgr{m}$,  we let
\[
	\gConf{\mathbf{d}}{\mathbf{n}}(X) \subset \Sym^{\mathbf{d}}(X)
\]
consist of sets $D$ of $|\mathbf{d}|$ (not necessarily distinct) points of $X$ such that
\begin{enumerate}[(i)]
	\item precisely $\mathbf{d}_k$ of the points in $D$ are labeled with the ``color'' $k$, and
	\item no point of $X$ has multiplicities of at least $\mathbf{n}_k$ for every $k$, or, equivalently, for each point $x$ of $X$, there exists a color $k$ whose multiplicity at $x$ is less than $n_k$.
\end{enumerate}
We will use the following notation to gather the whole family together
\[
	\gConf{m}{\mathbf{n}}(X) = \bigsqcup_{\mathbf{d} \in \multgr{m}} \gConf{\mathbf{d}}{\mathbf{n}}(X).
\]

We will also use the following notation
\[
	\gConf{m}{\infty}(X) = \bigsqcup_{\mathbf{d} \in \multgr{m}} \gConf{\mathbf{d}}{\infty}(X) = \bigsqcup_{\mathbf{d} \in \multgr{m}} \Sym^{\mathbf{d}}(X).
\]

\subsubsection{}
Consider the following $\multgr{m}$-graded commutative algebra\footnote{The variables $x_i$'s are in certain graded and cohomological degrees which we suppress in the introduction to keep the introduction simple.\label{fn:intro_ignore_coh_shifts}}
\[
	\algOr{m}{\mathbf{n}} = \Lambda[x_1, \dots, x_m]/(\sqcap_{k=1}^m x_k^{\mathbf{n}_k}),
\]
and when $\mathbf{n} = \infty$,
\[
	\algOr{m}{\infty} = \Lambda[x_1, \dots, x_m].
\]
The starting point is the following

\begin{prop}[Proposition~\ref{prop:factorization_cohomology_vs_Zmn}] \label{prop:intro:factorization_cohomology_vs_Zmn}
The factorization cohomology of $X$ with coefficients in $\algOr{m}{\mathbf{n}}$, denoted by $\alg{m}{\mathbf{n}}(X)$ is the cohomology of $\gConf{m}{\mathbf{n}}(X)$ (including the case of $\mathbf{n} = \infty$).\footnote{As mentioned earlier, by cohomology of $\gConf{\mathbf{d}}{\mathbf{n}}(X)$, for example, we really mean $C^*(\gConf{\mathbf{d}}{\mathbf{n}}(X), \sOmega_{\gConf{\mathbf{d}}{\mathbf{n}}(X)})$. See Notation~\ref{notation:sOmega}.}
\end{prop}

Note that this result establishes a link between an a priori combinatorially complicated object, the cohomology of $\gConf{m}{\mathbf{n}}(X)$, and an extremely simple commutative algebra. Now, since the algebras $\algOr{m}{\mathbf{n}}$ are so simple, their Koszul duals $\liealgOr{m}{\mathbf{n}}$ can be easily computed and turn out to be also very simple. As a result, we have an expression of $\alg{m}{\mathbf{n}}(X)$ as the $\coChev^{\un}$ of a $\coLie$-coalgebra $\liealg{m}{\mathbf{n}}(X)$, see Proposition~\ref{prop:conf_vs_cochev} and Corollary~\ref{cor:computation_of_liealg}. Homological stability of $\gConf{m}{\mathbf{n}}(X)$, or equivalently, of $\alg{m}{\mathbf{n}}(X)$, is now a direct consequence of Theorem~\ref{thm:homological_stability_coChev} and we obtain

\begin{thm}[Theorem~\ref{thm:homological_stability_Zmn}]
Let $X$ be an irreducible scheme of dimension $d\geq 1$. For each $1\leq k\leq m$ and $c\geq 0$, there exists a natural map (see Notation~\ref{notation:sOmega})
\[
	\Ho^c(\gConf{\mathbf{d}}{\mathbf{n}}(X), \sOmega_{\gConf{\mathbf{d}}{\mathbf{n}}(X)}) \to \Ho^c(\gConf{\mathbf{d} + \unit_k}{\mathbf{n}}(X), \sOmega_{\gConf{\mathbf{d}+\unit_k}{\mathbf{n}}(X)})
\]
that is
\begin{enumerate}[(i)]
	\item an equivalence when $\mathbf{d}_k \geq 2c$, and injective when $\mathbf{d}_k = 2c-1$, when $d=1, m = 1, n = 2$ (the case of configuration spaces of a curve),
	\item an equivalence when $\mathbf{d}_k \geq c$, and injective when $\mathbf{d}_k = c - 1$ otherwise.
\end{enumerate}
\end{thm}

More qualitatively, and more generally, we prove the following result about homological stability of factorization cohomology

\begin{thm}[Theorem~\ref{thm:homological_stab_fact_coh}]
Let $X$ be an irreducible scheme of dimension $d$. Let $\matheur{A} \in \ComAlg(\Vect^{\multgrplus{m}})$ such that $\mathfrak{a} = \coPrim[1](\matheur{A})$ is a finite dimensional $d$-shifted strongly unital $\coLie$-coalgebra (see Definition~\ref{defn:d_shifted_unital}). Then,
\[
	\pi_{?*}^{\un} \pi^! \matheur{A} \simeq \coChev^{\un}(C^*(X, \omega_X) \otimes \mathfrak{a})
\]
satisfies homological stability.
\end{thm}

\subsubsection{Categorified densities}
In what follows, we will use $\algb{m}{\mathbf{n}}(X)$ to denote the stable homology of $\alg{m}{\mathbf{n}}(X)$ (or equivalently, of $\gConf{m}{\mathbf{n}}(X)$), including the case where $\mathbf{n} = \infty$. Via a general functorial construction (mimicking the telescoping construction of~\cite{mcduff_homology_1976}) which works even when there is no homological stability, we show that $\algb{m}{\mathbf{n}}(X)$ naturally inherits a structure of a commutative algebra (see~\S\ref{subsubsec:telescope_construction} and Corollary~\ref{cor:adjunction_for_algebras_colim}).

Having explained how the factorization homology point of view is beneficial for establishing homological stability phenomena of generalized configuration spaces, we will now briefly sketch how densities could be understood under this lens. We have a natural map of commutative algebras
\[
	\algOr{m}{\infty} \simeq \Lambda[x_1, \dots, x_m] \to \Lambda[x_1, \dots, x_m]/(\sqcap_{k=1}^m x_k^{\mathbf{n}_k}) \simeq \algOr{m}{\mathbf{n}},
\]
which, via taking factorization cohomology, induces a map of commutative algebras
\[
	\alg{m}{\infty}(X) \to \alg{m}{\mathbf{n}}(X),
\]
which, in turn, induces a map of commutative algebras
\[
	\algb{m}{\infty}(X) \to \algb{m}{\mathbf{n}}(X)
\]
by taking stable homology.

\subsubsection{}
Now, there are naturally two kinds of ``densities'' we can examine:
\[
	\Lambda \otimes_{\alg{m}{\infty}(X)} \alg{m}{\mathbf{n}}(X) \quad \text{and}\quad \Lambda \otimes_{\algb{m}{\infty}(X)} \algb{m}{\mathbf{n}}(X).
\]
The first captures the density of the whole family $\gConf{m}{\mathbf{n}}(X)$ inside the family $\gConf{m}{\infty}(X)$, whereas the latter captures the stable density, i.e. density between stable homologies. The goal is to show that these densities depend only on $|\mathbf{n}| = \sum_{k=1}^m \mathbf{n}_k$.

Using the fact that factorization cohomology preserves relative tensors of algebras, it is now a direct consequence of the following equivalences\footnote{See footnote~\ref{fn:intro_ignore_coh_shifts}.} 
\[
	\Lambda\otimes_{\Lambda[x_1, \dots, x_m]} (\Lambda[x_1, \dots, x_m]/(\sqcap_{k=1}^m x_k^{\mathbf{n}_k})) \simeq \Lambda \oplus \Lambda[1] \simeq \Lambda \otimes_{\Lambda[x]} (\Lambda[x]/(x^{|\mathbf{n}|}))
\]
as commutative algebras in $\Vect^{\graded}$, where in the first term on the left, we collapse the grading into one via addition (see~\S\ref{subsec:change_of_gradings} for more details). In other words, the a priori complicated densities between the cohomology of $\gConf{m}{\mathbf{n}}(X)$ (either stable or the whole family) are controlled by extremely simple quotients involving truncated polynomial algebras, which explains the source of the coincidences. We prove the following result.

\begin{thm}[Theorems~\ref{thm:dependence_on_mn_of_graded_quotients}, \ref{thm:equivalence_at_limit}, and Proposition~\ref{prop:stable_homological_density}]
\label{thm:intro:categorified_densities_coincidences}
We have the following natural equivalence of algebras\footnote{Here, $\oblv_{\gr}$ denotes the functor of forgetting the gradings, i.e. taking direct sum of all the graded components.}
\[
	\Lambda \otimes_{\algb{m}{\infty}(X)} \algb{m}{\mathbf{n}}(X) \simeq \oblv_{\gr}(\Lambda \otimes_{\alg{m}{\infty}(X)} \alg{m}{\mathbf{n}}(X)).
\]
Moreover, $\Lambda \otimes_{\alg{m}{\infty}(X)} \alg{m}{\mathbf{n}}(X)$ and $\Lambda \otimes_{\algb{m}{\infty}(X)} \algb{m}{\mathbf{n}}(X)$ depend only on $|\mathbf{n}|$.
\end{thm}

We in fact produce canonical maps between these quotients for various $m, |\mathbf{n}|$. Moreover, we explicitly compute these two quotients, in Corollary~\ref{cor:computation_of_graded_quotients} and Proposition~\ref{prop:stable_homological_density} respectively. We note that the equivalence in the Theorem above is a special feature of the current situation, which is due to the fact that a certain $\coLie$-coalgebra happens to be abelian due to degree considerations.

\subsubsection{Decategorifications}
Applying various decategorification procedures (for example, Euler characteristics, \Poincare{} polynomials, or when $X$ is a pullback of a scheme over $\Fq$, virtual \Poincare{} polynomials, and also traces of the Frobenius) to Theorem~\ref{thm:intro:categorified_densities_coincidences}, we obtain generalizations of the results proved in~\cite{farb_coincidences_2019}. This is done in~\S\ref{subsec:decategorifications_Zmn} and~\S\ref{subsec:L_infty_algebras_computation}. In this part of the introduction, let us highlight the two cases of decategorification that are of special interest: Frobenius traces and \Poincare{} polynomials.

\subsubsection{} We show, in Corollary~\ref{cor:quotient_Frob_trace_infty}, that computing the Frobenius trace on $\Lambda \otimes_{\algb{m}{\infty}(X)} \algb{m}{\mathbf{n}}(X)$ recovers the $\zeta$-value $\zeta(\dim X|\mathbf{n}|)^{-1}$ and moreover, it is linked to the limit of the quotient of the point counts on $\gConf{m}{\mathbf{n}}(X)$ and $\gConf{m}{\infty}(X)$ (Remark~\ref{rmk:quotient_Frob_trace_infty}). As mentioned above, the commutative algebra $\Lambda \otimes_{\algb{m}{\infty}(X)} \algb{m}{\mathbf{n}}(X)$ thus provides an answer to the question posed in~\cite{vakil_discriminants_2015}.

\subsubsection{} One feature that makes the case of \Poincare{} polynomials stand out is that it does not behave nicely with respect to cofiber sequences. Because of this reason, quotients of algebras do not translate to quotients of \Poincare{} polynomials. In Proposition~\ref{prop:quotient_decat_Poincare}, we formulate a special condition where \Poincare{} polynomials still behave nicely with respect to relative tensors of commutative algebras. Translated to the geometric setting, it has to do with the $\coLie$-coalgebras $\liealg{m}{\mathbf{n}}(X)$ (the Koszul duals of $\alg{m}{\mathbf{n}}(X)$) being abelian. A simple calculation shows why we need the extra conditions in Theorem~\ref{thm:FWW_poinc_poly_coincidences} of~\cite{farb_coincidences_2019}, which is generalized in Proposition~\ref{prop:densities_Poincare_Zmn}.

\subsection{Other links to the literature}
\subsubsection{} In~\cite{vakil_discriminants_2015}, motivated by the Dold-Thom theorem and its apparent link to homological stability as well as special values of $\zeta$-functions, Vakil-Wood asked if the appearance of special values of $\zeta$-functions is controlled by a certain rational homotopy type. As a result of our work, we see that all the $\zeta$-values that appear as arithmetic densities come from corresponding commutative algebras obtained from taking the quotients, a.k.a relative tensor of commutative algebras, of the stable homology of $\gConf{m}{n}(X)$ by that of $\gConf{m}{\infty}(X)$ (see Corollary~\ref{cor:quotient_Frob_trace_infty} and Remark~\ref{rmk:quotient_Frob_trace_infty}). The resulting commutative algebras, which we call the stable homological densities, could thus be viewed as the rational homotopy types\footnote{Since everything is derived by our convention, by commutative algebra, we mean cdga.} responsible for the appearance of the $\zeta$-values $\zeta_X(mn\dim(X))^{-1}$, and hence, give an answer to a question asked by Vakil-Wood in~\cite{vakil_discriminants_2015}. In fact, these rational homotopy types have a simple description which we compute explicitly, in Proposition~\ref{prop:stable_homological_density}. We expect that the conceptual understanding of these phenomena provided in this paper will pave the way for realizing these coincidences at the level of spaces rather than just algebras.

\subsubsection{}
Another interesting feature of our work is its resemblance to the topological factorization homology picture, making the link between the two worlds more precise.

Firstly, the construction of the limiting rational homotopy type $\algb{m}{n}(X)$ from $\alg{m}{n}(X)$ closely related to the construction of the group completion of an $E_n$-algebra via the telescoping construction in~\cite{mcduff_homology_1976}. This is of special interest since group completion is known to be intimately linked to homological stability (see~\cite{may_geometry_1972,segal_configuration-spaces_1973,kupers_$e_n$-cell_2018,kupers_homological_2016}). Note, however, that our procedure is applied to an $E_\infty$-algebra (i.e. commutative), rather than $E_n$, and the maps between the different pieces in the colimit are induced by forgetting a point rather than adding a point at infinity, as the latter does not have any algebro-geometric meaning.

Secondly, our result on homological stability shares the same local-to-global principle as the one in the work of Kupers-Miller in~\cite{kupers_$e_n$-cell_2018}. More precisely, the finiteness condition on the $\coLie$-coalgebras resembles that which appears in loc. cit (see Remark~\ref{rmk:bounded_generation}). Our result, therefore, effectively provides an algebro-geometric analog of their topological result. 

Finally, note also that the commutative algebras appearing in our work are the $E_n$-Koszul duals to those in~\cite{kupers_$e_n$-cell_2018,knudsen_betti_2017,ho_free_2017}. The fact that these happen to be commutative allows us, via another type of Koszul duality between $\coLie$-coalgebras and commutative algebras, to reduce many computations to the $\coLie$-coalgebra side, which is finite dimensional in nature, greatly simplifying the picture appearing there. Moreover, formulating statements in terms of commutative algebras allows us to define homological densities, which is crucial for our purposes.

\subsubsection{}
Since this paper first appeared, Petersen~\cite{petersen_cohomology_2020} has obtained similar results regarding the cohomology of generalized \emph{ordered} configuration spaces of \emph{one color} by completely different means. Ordered vs. unordered and one color vs. multiple colors aside, our expression of the cohomology of configuration spaces in terms of factorization cohomology (Proposition~\ref{prop:intro:factorization_cohomology_vs_Zmn}) could be viewed as an abstract and model-free construction of the functor $\matheur{C}\matheur{F}$ of~\cite{petersen_cohomology_2020}*{Thm~1.1.6}. Note that even though we focus only on the case of constant sheaves with characteristic $0$ coefficients in this paper, it is easy to see that this step works more generally with arbitrary coefficients as in~\cite{petersen_cohomology_2020}. However, while~\cite{petersen_cohomology_2020} obtains a spectral sequence to compute $\matheur{C}\matheur{F}$ in general, our corresponding results, Proposition~\ref{prop:conf_vs_cochev} and Corollary~\ref{cor:computation_of_liealg}, rely on Koszul duality and as is, only work in characteristic $0$.\footnote{In our case, the spectral sequence is the one attached to a homotopy limit, using Corollary~\ref{cor:coChev_as_limit}.} See also~\S\ref{subsubsec:future_work_general_coefficients}.

\subsection{Future work}\label{subsec:future_work}
Our work points to further questions, to which we hope to return in the future.

\subsubsection{General coefficients.} \label{subsubsec:future_work_general_coefficients}
The techniques developed in this paper are particularly well suited for the studies of generalized configuration spaces. In a recent preprint~\cite{ho_higher_2020}, we introduce twisted commutative factorization algebras and use the techniques developed in this paper to study higher representation stability of generalized ordered configuration spaces in the style of~\cite{miller_higher_2016}. In~\cite{petersen_cohomology_2020}, the author uses twisted commutative algebras to go beyond the case of characteristic $0$ coefficients. It would be very interesting to see if these ideas can be applied to our settings as well.

\subsubsection{Spaces of maps} In a slightly different vein, the spaces of particles have been known to be closely related to the spaces of maps, for example,~\cite{segal_configuration-spaces_1973,mcduff_homology_1976, segal_topology_1979,boyer_topology_1994,guest_topology_1995}. In fact, many of these works serve as the inspiration for~\cite{farb_coincidences_2019} and our current work. A feature shared by these examples is that the target spaces have natural torus actions, whose sets of fixed points are finite. It is thus plausible that the technique of torus localization could be used to produce factorization algebras labeled by these fixed points whose factorization homology
\begin{enumerate}[(i)]
	\item admits a descripion in terms of labeled configuration spaces,
	\item exhibits homological stability, and
	\item whose stable homology is related to the homology of the space of maps under investigation.
\end{enumerate}

More recently, using Fourier--Deligne transform and ideas from the circle method,~\cite{browning_geometric_2020} provides the first evidence pointing to the relation between the spaces of particles and the spaces of maps beyond the toric case mentioned above. In fact, the paper formulates a precise conjecture relating the two~\cite{browning_geometric_2020}*{Conjecture 1.3}. Intriguingly, a certain factorization structure also appears in their work and is used in a crucial way. It is thus very interesting to investigate this question in the framework of factorization homology. 

\subsection{An outline of the paper}
We will now give a brief outline of how the paper is organized.

In~\S\ref{sec:prelim}, besides fixing the notation, we give a brief review of the various mathematical tools used throughout the paper as well as pointers to the exisiting literature. In~\S\ref{sec:coBar_as_limit}, we establish a technical result concerning the functor of taking cohomological Chevalley complex (or more generally, the functor $\coBarP_{\matheur{P}}$ for a co-operad $\matheur{P}$) in a pro-nilpotent category, which will allow us to reduce many arguments to the case where the $\coLie$-coalgebra involved is trivial (a.k.a. abelian). In~\S\ref{sec:generalities_homological_stability} we lay down the framework where homological stability naturally takes place. In~\S\ref{sec:cohomological_chevalley_complex} we prove homological stability for cohomological Chevalley complexes for a large class of $\coLie$-coalgebras (along with the explicit stability bounds) and give an expression for the stable homology in special cases. Moreover, we also introduce various decategorification procedures to extract information out of a cohomological Chevalley complex. Geometry makes its first appearance in~\S\ref{sec:fact_hom}. After a quick review of factorization algebra and homology (in the graded setting), we prove  homological stability of factorization cohomology and end the section with a functorial procedure on commutative factorization algebras which will be used later in~\S\ref{sec:cohomology_of_Z_m_n} to link the cohomology of  $\gConf{m}{\mathbf{n}}(X)$ to factorization cohomology. In~\S\ref{sec:cohomology_of_Z_m_n}, we finally study the spaces $\gConf{m}{\mathbf{n}}(X)$. In particular, we prove agreements between categorified homological densities and show that via decategorifications, we obtain generalizations of previous known results.

\section{Preliminaries} \label{sec:prelim}
In this section, we collect notation, conventions, as well as provide definitions, results, and references to them, that we will use throughout the paper.

\subsection{Notation and conventions}
\subsubsection{Category theory} \label{subsubsec:convention_categories}We will use $\DGCat$ to denote the $(\infty, 1)$-category of stable infinity categories, $\DGCat_\pres$ to denote the full subcategory of $\DGCat$ consisting of presentable categories, and $\DGCat_{\pres, \cont}$ the (non-full) subcategory of $\DGCatpres$ where we only allow continuous functors, i.e. those commuting with colimits. 

We will also fix an algebraically closed field $\Lambda$ of characteristic 0 (see also~\S\ref{subsubsec:convention_sheaves_schemes}). Unless otherwise specified, all our stable infinity categories and functors between them are in $\DGCat_{\pres, \cont}$. In particular, for symmetric monoidal categories, the monoidal/tensor products commute with colimit in each variable. Moreover, we ask that these categories are linear over $\Lambda$.

The main reference for the subject are~\cite{lurie_higher_2017, lurie_higher_2017-1}. For a slightly different point of view, see also~\cite{gaitsgory_study_2017}.

\subsubsection{Algebraic geometry} Throughout this paper, $k$ will be used to denote an algebraically closed ground field. We will denote by $\Sch$ the $\infty$-category obtained from the ordinary category of separated schemes of finite type over $k$. All our schemes will be objects of $\Sch$. In most cases, we will use the calligraphic font to denote prestacks, for eg. $\matheur{X}, \matheur{Y}$ etc., and the usual font to denote schemes, for eg. $X, Y$ etc.

\subsubsection{$t$-structure} Let $\matheur{C}$ be a stable infinity category, equipped with a $t$-structure. Then we have the following diagram of adjoint functors
\[
\xymatrix{
	\matheur{C}^{\leq 0} \ardis[r]^>>>>>{i_{\leq 0}} & \ardis[l]^>>>>>{\tr_{\leq 0}} \matheur{C} \ardis[r]^<<<<<{\tr_{\geq 1}} & \ardis[l]^>>>>>{i_{\geq 1}} \matheur{C}^{\geq 1}
}
\]
Note that we use cohomological indexing.

We use $\tau_{\leq 0}$ and $\tau_{\geq 1}$ to denote
\[
	\tau_{\leq 0} = i_{\leq 0} \circ \tr_{\leq 0}: \matheur{C} \to \matheur{C}
\]
and
\[
	\tau_{\geq 1} = i_{\geq 1} \circ \tr_{\geq 1}: \matheur{C} \to \matheur{C}
\]
respectively.

Shifts of these functors, for e.g. $\tau_{\geq n}$ and $\tau_{\leq n}$, are defined in the obvious ways.

\subsection{Koszul duality} The theory of Koszul duality, initially developed in~\cite{quillen_rational_1969}, illuminated the duality between co-commutative co-algebras and Lie algebras. It was vastly generalized to the operadic setting in~\cite{ginzburg_koszul_1994}. An application of this theory to the factorizable setting was carried out in~\cite{francis_chiral_2011}.

In this paper, we will make extensive use of Koszul duality between commutative algebras and $\coLie$-coalgebras. We will summarize the relevant features of this theory below. The interested readers could find a more detailed development in~\cite{francis_chiral_2011,gaitsgory_study_2017,loday_algebraic_2012}. The presentation here follows closely those of~\cite{francis_chiral_2011, gaitsgory_study_2017}.

\subsubsection{Symmetric sequences} Let $\Vect^\Sigma$ denote the category of symmetric sequences. Namely, it consists, as objects, collections
\[
\matheur{O} = \{\matheur{O}(n), n\geq 1\},
\]
where $\matheur{O}(n) \in \Vect^{B\Sigma_n}$, i.e. chain complexes equipped with actions of $\Sigma_n$, the permutation group of $n$ letters.

This category is equipped with a (non-symmetric) monoidal structure, designed to make the functor
\[
\Vect^\Sigma \to \Fun(\Vect, \Vect)
\]
given by
\[
\matheur{O} \star V = \bigoplus_{n\geq 1} (\matheur{O}(n) \otimes V^{\otimes n})_{\Sigma_n}
\]
monoidal.

\subsubsection{Operads and co-operads} By an operad (resp. co-operad), we will mean an augmented associative algebra (resp. co-algebra) in $\Vect^\Sigma$. We will use $\Op$ (resp. $\coOp$) to denote the categories of operads (resp. co-operads).

\subsubsection{} The Bar and coBar construction gives us a pair of adjoint functors
\[
\BarO: \Op \rightleftarrows \coOp: \coBarP.
\]
For an operad $\matheur{O}$ (resp. co-operad $\matheur{P}$), we will also use $\matheur{O}^\vee$ (resp. $\matheur{P}^\vee$) to denote $\BarO(\matheur{O})$ (resp. $\coBarP(\matheur{P})$). 

In this paper, we will restrict ourselves to the class of operads/co-operads where the augmentation map is an equivalence in degree 1, i.e. $\matheur{O}(1) \simeq \Lambda$ and $\matheur{P}(1) \simeq \Lambda$ (see~\ref{subsubsec:convention_sheaves_schemes}). Under this restriction, one can show that the unit map
\[
\matheur{O} \to \coBarP\circ\BarO(\matheur{O}) = (\matheur{O}^\vee)^\vee
\]
is an equivalence.

\subsubsection{Algebras and co-algebras}
\label{subsubsec:algebras_coalgebras}
Let $\matheur{C}$ be a stable presentable symmetric monoidal $\infty$-category compatibly tensored over $\Vect$. Then, an operad $\matheur{O}$ (resp. co-operad $\matheur{P}$) naturally defines a monad (resp. comonad) on $\matheur{C}$. Thus, we can talk about the category of algebras $\matheur{O}\Palg(\matheur{C})$ (resp. coalgebras $\matheur{P}\Pcoalg(\matheur{C})$) in $\matheur{C}$.

One has the following pairs of adjoint functors
\[
\Free_{\matheur{O}}: \matheur{C} \rightleftarrows \matheur{O}\Palg(\matheur{C}):\oblv_{\matheur{O}} \qquad \text{and}\qquad \BarO_{\matheur{O}}: \matheur{O}\Palg(\matheur{C}) \rightleftarrows \matheur{C} :\triv_{\matheur{O}}
\]
for an operad $\matheur{O}$, and similarly, the following pairs of adjoint functors
\[
\oblv_{\matheur{P}}: \matheur{P}\Pcoalg(\matheur{C}) \rightleftarrows \matheur{C} :\coFree_{\matheur{P}} \qquad\text{and}\qquad \cotriv_{\matheur{P}}: \matheur{C} \rightleftarrows \matheur{P}\Pcoalg(\matheur{C}) :\coBarP_{\matheur{P}}
\]
for a co-operad $\matheur{P}$.

Note that even though the functor $\BarO_{\matheur{O}}$ (resp. $\coBarP_{\matheur{P}}$) is defined abstractly as an adjoint functors, it can also be explicitly written as the homotopy colimit (resp. limit) of an explicit simplicial (resp. cosimplicial) object similar to the usual $\BarO$ (resp. $\coBarP$) construction.

\subsubsection{Koszul duality} The functors $\BarO_{\matheur{O}}$ and $\coBarP_{\matheur{P}}$ mentioned in~\S\ref{subsubsec:algebras_coalgebras} in fact lift to the following pair of adjoint functors
\[
\BarO^\enh_\matheur{O}: \matheur{O}\Palg(\matheur{C}) \rightleftarrows \matheur{P}\Pcoalg(\matheur{C}): \coBarP^\enh_\matheur{P}, \teq\label{eq:koszul_adjunction}
\]
where $\matheur{P} = \matheur{O}^\vee$ and where
\[
\oblv_{\matheur{P}}\circ \BarO_\matheur{O}^\enh \simeq \BarO_\matheur{O} \qquad\text{and}\qquad \oblv_{\matheur{O}}\circ \coBarP_\matheur{P}^\enh \simeq \coBarP_\matheur{P}.
\]

\subsubsection{The pro-nilpotent case} The pair of adjoint functors $\BarO^\enh_\matheur{O} \dashv \coBarP^\enh_\matheur{P}$ are not mutual inverses in general. One of the main achievements of~\cite{francis_chiral_2011} is to formulate a general condition on $\matheur{C}$, namely, the pro-nilpotence condition, under which these functors are inverses. Let us recall the definition.

\begin{defn}[\cite{francis_chiral_2011}*{Defn. 4.1.1}]
	\label{defn:pro_nilpotent}
	Let $\matheur{C}$ be a (not necessarily unital) stable presentable symmetric monoidal category. We say that $\matheur{C}$ is pro-nilpotent if it can be exhibited as a limit
	\[
	\matheur{C} \simeq \lim (\matheur{C}_0 \leftarrow \matheur{C}_1 \leftarrow \matheur{C}_2 \leftarrow \cdots)
	\]
	such that
	\begin{enumerate}[(i)]
		\item $\matheur{C}_0 = 0$;
		\item For every $i\geq j$, the transition functor
		\[
		f_{i, j}: \matheur{C}_i \to \matheur{C}_j
		\]
		commutes with limits;\footnote{They also commute with colimits as per our convention, see~\S\ref{subsubsec:convention_categories}.}
		\item For every $i$, the restriction of the tensor product $\matheur{C}_i \otimes \matheur{C}_i \to \matheur{C}_i$ to $\ker(f_{i, i-1})\otimes \matheur{C}_i$ is null-homotopic.
	\end{enumerate}
\end{defn}

We have the following result.
\begin{prop}[\cite{francis_chiral_2011}*{Prop. 4.1.2., Lem. 3.3.4, Lem. 4.1.6 (b)}]
When $\matheur{C}$ is pro-nilpotent, then the pair of adjoint functors~\eqref{eq:koszul_adjunction} are mutually inverses. Moreover, they exchange trivial and free objects on both sides.
\end{prop}

\begin{rmk}
	In this paper, the pro-nilpotent categories that appear will take the following pattern. Let $\matheur{C}$ be a symmetric monoidal stable infinity category, and $\mathbb{Z}_+$ the (discrete) monoidal category whose objects are the positive integers and monoidal structure given by addition. Then, the category of graded objects in $\matheur{C}$, $\Fun(\mathbb{Z}_+, \matheur{C})$, has a natural monoidal structure, which is pro-nilpotent. Indeed, let $\mathbb{Z}_+^{<d}$ be the full-subcategory of $\mathbb{Z}_+$ consisting of integers smaller than $d$. Then, $\Fun(\mathbb{Z}_+^{<d}, \matheur{C})$ has a natural monoidal structure which makes the restriction $\Fun(\mathbb{Z}_+, \matheur{C}) \to \Fun(\mathbb{Z}_+^{<d}, \matheur{C})$ symmetric monoidal. Clearly,
	\[
		\Fun(\mathbb{Z}_+, \matheur{C}) \simeq \lim_{d} \Fun(\mathbb{Z}_+^{<d}, \matheur{C})
	\]
	and it is easy to see that this limit exhibits $\Fun(\mathbb{Z}_+, \matheur{C})$ as a pro-nilpotent category.
	
	See~\S\ref{sec:generalities_homological_stability} for the precise forms of the categories that appear.
\end{rmk}

\subsubsection{The case of $\coLie$ and $\ComAlg$} We have the following equivalence of co-operads (see, for example,~\cite{francis_chiral_2011}*{\S4.3}):
\[
\Lie^\vee \simeq \ComCoAlg[1], \teq\label{eq:lie_cocom_operad}
\]
where
\[
\ComCoAlg[1](n) \simeq \Lambda[n-1]
\]
is equipped with the sign action of the symmetric group $\Sigma_n$. This gives rise to the Koszul duality between Lie algebras, and cocommutative coalgebras, originally discovered by Quillen.

\subsubsection{} In general, for any co-operad $\matheur{P}$, the shifted co-operad $\matheur{P}[1]$ is defined as follows,
\[
\matheur{P}[1](n) \simeq \matheur{P}(n)[n-1]. 
\]
Thus, the functor
\[
[1]: \matheur{C} \to \matheur{C}
\]
gives rise to an equivalence of categories
\[
[1]: \matheur{P}[1]\Pcoalg(\matheur{C}) \simeq \matheur{P}\Pcoalg(\matheur{C}).
\]

\subsubsection{} Dual to~\eqref{eq:lie_cocom_operad} is the following equivalence of operads
\[
\ComAlg^\vee \simeq \coLie[1],
\]
which we will make use of extensively in this paper.

Note that from what we said above, the functor
\[
[1]: \matheur{C} \to \matheur{C}
\]
gives rise to an equivalence of categories
\[
[1]: \coLie[1](\matheur{C}) \simeq \coLie(\matheur{C}).
\]

\subsubsection{} \label{subsubsec:defn_coChev_et_al} This gives us the following commutative diagram which summarizes the situation for the case of $\ComAlg$ and $\coLie$.
\[
\xymatrix{
	\ComAlg(\matheur{C}) \ardissp[ddrr]^{\coPrim[1]} \ardis[dd]^{[1]} \ardis[rr]^<<<<<<<<<<<<<{\BarO_{\ComAlg}} && \ardis[ll]^>>>>>>>>>>>>>{\coBarP_{\coLie[1]}} \coLie[1](\matheur{C}) \ardis[dd]^{[1]} \\ \\
	\ComAlg[-1](\matheur{C}) \ardis[uu]^{[-1]} \ardis[rr]^>>>>>>>>>>>{\BarO_{\ComAlg[-1]}} && \ardis[ll]^<<<<<<<<<<<{\coBarP_{\coLie}} \coLie(\matheur{C}) \ardis[uu]^{[-1]} \ar[uull]^{\coChev}
}
\]

As indicated in the diagram, we record the following definition.

\begin{defn}
We use $\coChev$ to denote
\[
	[-1] \circ \coBarP_{\coLie} \simeq \coBarP_{\coLie[1]} \circ [-1]
\]
and $\coPrim[1]$ to denote
\[
	[1] \circ \BarO_{\ComAlg} \simeq \BarO_{\ComAlg[-1]} \circ [1].
\]
\end{defn}

\subsubsection{} Since a trivial $\Lie$-algebra (one with trivial brackets) is usually called an abelian $\Lie$-algebra, we will adopt the same convention for $\coLie$-coalgebras. Namely, the terms \emph{trivial} and \emph{abelian} $\coLie$-coalgebras are used interchangeably.

\subsection{Unital vs non-unital}
\label{subsec:unital_vs_nonunital}

Because of the way our operad $\ComAlg$ is defined, $\ComAlg(\matheur{C})$ is the category of \emph{non-unital} commutative algebras. In particular, for any object $c\in \matheur{C}$,
\[
\Free_{\ComAlg} c = \bigoplus_{n > 0} \Sym^n c = (\Sym c)_+.
\]

Compared to the usual (i.e. unital) commutative algebras, non-unital ones are somewhat less frequently used, and moreover, many constructions are more naturally implemented in the unital setting. In this subsection, we will give a quick review of the relations between the two.

\subsubsection{}
When $\matheur{C}$ is (unital) symmetric monoidal, we use $\ComAlg^{\un, \aug}(\matheur{C})$ to denote the category of augmented unital commutative algebra objects in $\matheur{C}$, and we have an equivalence of categories
\begin{align*}
\addUnit: \ComAlg(\matheur{C}) \rightleftarrows \ComAlg^{\un, \aug}(\matheur{C}): (-)_+
\end{align*}
given by formally adjoining a unit in one direction and taking the augmentation ideal the other direction.

\subsubsection{}
\label{subsubsec:fake_un_aug}
In general, when $\matheur{C}$ is not unital, but it is a full monoidal subcategory of a unital symmetric monoidal stable $\infty$-category $\matheur{C}^\un$, we let
\[
\ComAlg^{\un, \aug}(\matheur{C})
\]
be the category of unital augmented commutative algebra objects in $\matheur{C}^\un$ whose augmentation ideal lies in $\matheur{C}$. Then, by the same reasoning as above, we have an equivalence of categories
\[
\addUnit: \ComAlg(\matheur{C}) \rightleftarrows \ComAlg^{\un, \aug}(\matheur{C}): (-)_+
\]
given by formally adjoining the unit and removing the unit.

Note that a priori, $\ComAlg^{\un, \aug}(\matheur{C})$ depends on the embedding $\matheur{C} \subset \matheur{C}^\un$. However, the equivalence above implies that this is not the case. This is why we choose to suppress $\matheur{C}^\un$ from the notation.

\subsubsection{} \label{subsubsec:non_unital_quotient} One of the common procedures we perform with commutative algebras is taking relative tensor products. Let
\[
\xymatrix{
	\matheur{A}_1 \ar[d] \ar[r] & \matheur{A}_2 \\
	\matheur{A}_3
}
\]
be a diagram in $\ComAlg(\matheur{C})$. The pushout of this diagram in $\ComAlg(\matheur{C})$, denoted by $\matheur{A}_2 \sqcup_{\matheur{A}_1} \matheur{A}_3$, could be computed as follows
\[
\matheur{A}_2 \sqcup_{\matheur{A}_1} \matheur{A}_3 \simeq (\addUnit(\matheur{A}_2) \otimes_{\addUnit(\matheur{A}_1)} \addUnit(\matheur{A}_3))_+.
\]
Here, the tensor on the right hand side is the usual relative tensor of unital commutative algebras. Note that for unital commutative algebras, relative tensors implement pushouts.

\subsubsection{}
Let $\matheur{C}$ be a pro-nilpotent category. As mentioned above, we have the following equivalence of categories
\[
\coPrim[1]: \ComAlg(\matheur{C}) \rightleftarrows \coLie(\matheur{C}): \coChev.
\]

Suppose $\matheur{C}$ is a full sub-monoidal-category of $\matheur{C}^\un$ as above. We will use the notation
\[
\coChev^\un = \addUnit \circ \coChev
\]
to denote the unital cohomological Chevalley functor. In particular, when $\mathfrak{a}$ is a trivial coLie-coalgebra (i.e. trivial co-multiplication maps), we have
\[
\coChev^\un \mathfrak{a} \simeq \addUnit \circ \Free_{\ComAlg} (\mathfrak{a}[-1]) = \Sym (\mathfrak{a}[-1]).
\]

\subsection{Prestacks} The theory of sheaves on prestacks has been developed in~\cite{gaitsgory_weils_2014,gaitsgory_atiyah-bott_2015}. In this subsection and the next, we will give a brief review of this theory. We will state them as facts, without any proof, which, unless otherwise specified, could all be found in~\cite{gaitsgory_atiyah-bott_2015}. 

Prestacks are needed to set up the theory of factorization algebra and factorization homology. We will only make use of this theory in~\S\ref{sec:fact_hom} and~\S\ref{sec:cohomology_of_Z_m_n}, where it is needed to bridge cohomological Chevalley complexes and cohomology of the spaces $\gConf{\mathbf{d}}{\mathbf{n}}(X)$ mentioned in the introduction. At the first pass, the reader should feel free to treat it as a blackbox.

\subsubsection{} A prestack is a contravariant functor from $\Sch$ to $\Spc$, the $\infty$-category of spaces. In other words, a prestack $\matheur{Y}$ is a functor
\[
	\matheur{Y}: \Sch^\op \to \Spc.
\]

Let $\PreStk$ be the $\infty$-category of prestacks. By Yoneda's lemma, we have a fully-faithful embedding
\[
	\Sch \hookrightarrow \PreStk.
\]

\subsubsection{} Direct from the definition, any prestack $\matheur{Y}$ can be written as a colimit of schemes
\[
	\matheur{Y} \simeq \colim_{i\in I} Y_i.
\]

\subsubsection{} A morphism
\[
	\matheur{X} \to \matheur{Y}
\]
between prestacks is said to be \emph{schematic} if its pullback along any scheme
\[
	S \to \matheur{Y}
\]
is a scheme.

\subsection{Sheaves on prestacks}
As mentioned above, proofs of all the results in this subsection, unless otherwise specified, could be found in~\cite{gaitsgory_atiyah-bott_2015}.

\subsubsection{Sheaves on schemes} \label{subsubsec:convention_sheaves_schemes} For a scheme $S$,
\begin{enumerate}[(i)]
	\item when the ground field is $\mathbb{C}$, and $\Lambda$ is an arbitrary field of characteristic 0, we take $\Shv(S)$ to be the ind-completion of the category of constructible sheaves on $S$ with $\Lambda$-coefficients.
	
	\item for any algebraically closed ground field $k$ in general, and $\Lambda = \Ql, \Qlbar$ with $\ell \neq \car k$, we take $\Shv(S)$ to be the ind-completion of the category of constructible $\ell$-adic sheaves on $S$ with $\Lambda$-coefficients. See also~\cite[\S4]{gaitsgory_weils_2014}, \cite{liu_enhanced_2012}, and~\cite{liu_enhanced_2014}.
\end{enumerate}

For any morphism between schemes
\[
	f: S_1 \to S_2,
\]
we have the following pairs of adjoint functors
\begin{align*}
	f_!: \Shv(S_1) \rightleftarrows \Shv(S_2): f^! \qquad \text{and}\qquad 
	f^*: \Shv(S_2) \rightleftarrows \Shv(S_1): f_*.
\end{align*}
Moreover, we also have the box-product $\boxtimes$, and hence, $\otimes$ and $\otimesshriek$. Note also that $f^!$ is continuous, i.e. it preserves colimits.

\subsubsection{} When $X = \pt = \Spec k$, we use
\[
	\Vect = \Shv(\pt) \teq\label{eq:Vect_is_Shv(pt)}
\]
to denote the category of chain complexes in vector spaces over $\Lambda$.

\subsubsection{Sheaves on prestacks} The theory of sheaves on scheme provides us with a functor
\[
	\Shv: \Sch^\op \to \DGCat_{\pres, \cont},
\]
where we use the $!$-pullback functor to move between schemes.

We can right Kan extend this functor along the Yoneda embedding
\[
	\Sch^\op \hookrightarrow \PreStk^\op
\]
to obtain a functor
\[
	\Shv: \PreStk^\op \to \DGCat_{\pres, \cont}.
\]

For formal reasons, the functor
\[
	\Shv: \PreStk^\op \to \DGCat_{\pres, \cont}
\]
preserves limits, i.e.
\[
	\Shv(\colim_i \matheur{Y}_i) \simeq \lim_i \Shv(\matheur{Y}_i).
\]

In particular, if a prestack
\[
	\matheur{Y} \simeq \colim_i Y_i
\]
is a colimit of schemes, then
\[
	\Shv(\matheur{Y}) \simeq \lim_i \Shv(Y_i).
\]

\subsubsection{} Unwinding the definition of limits, informally, an object $\matheur{F}\in \Shv(\matheur{Y})$ is the same as the following data
\begin{enumerate}[(i)]
	\item A sheaf $\matheur{F}_{S, y} \in \Shv(S)$ for each $S\in \Sch$ and $y: S\to \matheur{Y}$ (i.e. $y\in \matheur{Y}(S)$).
	\item An equivalence of sheaves $\matheur{F}_{S', f(y)} \to f^! \matheur{F}_{S, y}$ for each morphism of schemes $f: S' \to S$.
\end{enumerate}
Moreover, we require that this assignment satisfies a homotopy-coherent system of compatibilities.

\begin{rmk}
Historically, the theory of sheaves on prestacks was first applied to the theory of $D$-modules on ind-schemes which favors the $!$-functors in multiple ways. Firstly, for $D$-modules, $!$-pullback is always defined while it is not true for $*$-pullback. Secondly, for (possibly infinite dimensional) ind-schemes, homology is better behaved compared to cohomology as it has better finiteness properties, and the $!$-functors are designed to capture homology within sheaf cohomology theory.

Locally, the prestacks that appear in this paper are defined by finite colimits of diagrams of schemes where all the maps are \etale{}. We could thus have used the $*$-functors to set up the theory. However, for ease of reference, we decide to stick to the more well-established theory of sheaves on prestacks using the $!$-functors.
\end{rmk}

\subsubsection{} Directly from the definition, for any morphism of prestacks
\[
	f: \matheur{X} \to \matheur{Y},
\]
we obtain a functor
\[
	f^!: \Shv(\matheur{Y}) \to \Shv(\matheur{X})
\]
which commutes with both limits and colimits. In particular, $f^!$ admits a left adjoint $f_!$.

\subsubsection{} In general, we do not have an explicit formula for the functor $f_!$. However, when the target of $f$ is a scheme, i.e.
\[
	f: \matheur{Y} \to S,
\]
where $S$ is a scheme, and suppose that
\[
	\matheur{Y} \simeq \colim_i Y_i.
\]
Then,
\[
	f_! \matheur{F} \simeq \colim_i f_{i!} \matheur{F}_{Y_i}
\]
where
\[
	f_i: Y_i \to \matheur{Y} \to S
\]
and $\matheur{F}_{Y_i}$ is the $!$-pullback of $\matheur{F}$ to $Y_i$.

\subsubsection{} \label{subsubsec:pseudo_proper_basechange}
More generally, suppose $f: \matheur{X} \to \matheur{Y}$ is a pseudo-proper morphism (briefly reviewed below), then the functor $f_!$ satisfies base-change with respect to $g^!$ for any morphism of prestacks $g: \matheur{Y}'\to \matheur{Y}$. In particular, we can take $S$ to be a scheme, and reduce to the situation above. Here, by pseudo-properness, we mean that the base change of $f$ to any scheme $S$ is a colimit of schemes proper over $S$.

\subsubsection{} For a general morphism between prestacks
\[
	f: \matheur{X} \to \matheur{Y},
\]
the adjoint pair $f^* \dashv f_*$ between the associated categories of sheaves is not defined. However, when $f$ is schematic, then the adjoint pair exists and is computable (see~\cite{ho_free_2017}*{\S2.6}). 

The most important property of $f_*$ is that it satisfies base change. Namely for any pullback square of prestacks
\[
\xymatrix{
	\matheur{X}' \ar[d]_f \ar[r]^g & \matheur{X} \ar[d]_f \\
	\matheur{Y}' \ar[r]^g & \matheur{Y}
}
\]
where $\matheur{X} \to \matheur{Y}$, and hence, also $\matheur{X}'\to \matheur{Y}'$, is schematic, and for any $\matheur{F} \in \Shv(\matheur{X})$, we have a natural equivalence
\[
	g^! f_* \matheur{F} \simeq f_* g^! \matheur{F}
\]

\subsubsection{} \label{subsubsec:base_change_for_*}
In particular, if
\[
	\matheur{Y} = \colim_i Y_i
\]
where $Y_i$'s are schemes, then for each $i$, the pullback diagram above becomes
\[
\xymatrix{
	X_i \ar[d]_f \ar[r]^{g_i} & \matheur{X} \ar[d]_f \\
	Y_i \ar[r]^{g_i} & \matheur{Y}
}
\]
where $X_i$ is also a scheme. The object $f_* \matheur{F} \in \Shv(\matheur{Y})$ is determined by its $!$-pullbacks to the $Y_i$'s, which are computed in terms of the usual $*$-pushforward. Namely:
\[
	g_i^! f_* \matheur{F} \simeq f_* g_i^! \matheur{F}
\]
and the $f_*$ on the RHS is the usual pushforward of sheaves between schemes.

\subsubsection{} The functor $f^*$ is slightly more complicated to describe. However, when
\[
	f: \matheur{Y}_1 \to \matheur{Y}_2
\]
is \etale{}, which is the case where we need, we have a natural equivalence (see~\cite{ho_free_2017}*{Prop. 2.7.3})
\[
	f^! \simeq f^*. \teq\label{eq:etale_!_*_pullback_equiv}
\]

\subsubsection{Monoidal structure} The theory of sheaves on prestacks inherits the monoidal structure (i.e. box tensor) from the theory of sheaves on schemes. Namely, let $\matheur{F}_i \in \Shv(\matheur{Y}_i)$ where $\matheur{Y}_i$'s are prestacks, for $i=1, 2$. Then, the sheaf $\matheur{F}_1 \boxtimes \matheur{F}_2$ is characterized by the condition that for any pair of maps $f_i: S_i \to \matheur{Y}_i$ where the $S_i$'s are schemes, we have
\[
	(f_1 \times f_2)^! (\matheur{F}_1 \boxtimes \matheur{F}_2) \simeq f_1^! \matheur{F}_1 \boxtimes f_2^! \matheur{F}_2.
\]

As in the theory of sheaves for schemes, restricting along the diagonal map lets us define $\otimesshriek$ in general, and $\otimesstar$ when the diagonal map is schematic.

\subsection{The $\Ran$ space/prestack}
\label{subsec:intro_ran_space}
The $\Ran$ space plays the central role in the theory of factorization algebras and homology. The version of the $\Ran$ space used in this paper is that of a colored/graded $\Ran$ space. In this section, however, we will review the usual/non-graded $\Ran$ space to familiarize the readers with the kinds of structure involved. The colored version will be defined in~\S\ref{sec:fact_hom}.

All results in this section could be found in~\cite{francis_chiral_2011,gaitsgory_weils_2014,gaitsgory_atiyah-bott_2015}. The readers are referred to these papers for a full account of the theory. The topologically inclined readers could also find a more intuitive introduction in~\cite{ho_free_2017}.

\subsubsection{} For a scheme $X\in \Sch$, we will use $\Ran X$ to denote the prestack parametrizing non-empty finite subsets of $X$. More precisely, for each scheme $S \in \Sch$,
\[
	(\Ran X)(S) = \{\text{non-empty finite subsets of }X(S)\}.
\]
Equivalently, we have
\[
	\Ran X = \colim_{I\in \fSet^{\surj, \op}} X^I
\]
where $\fSet^\surj$ is the category of non-empty finite sets.

\subsubsection{$\otimesstar$-monoidal structure} The $\Ran$ space has a natural (non-unital) monoidal structure given by the union map
\[
	\union: \Ran X \times \Ran X \to \Ran X.
\]
This allows us to define the $\otimesstar$-monoidal structure on $\Shv(\Ran X)$ given by:
\[
	\matheur{F} \otimesstar \matheur{G} = \union_!(\matheur{F} \boxtimes \matheur{G}).
\]

\subsubsection{Commutative factorizable sheaves}
Using $\otimesstar$-monoidal structure on $\Shv(\Ran X)$, one can make sense of the categories of algebras and coalgebras on it. The one that is relevant to us is
\[
	\ComAlgstar(\Ran X) = \ComAlg(\Shv(\Ran X)^{\otimesstar}),
\]
the category of commutative algebra objects in $\Shv(\Ran X)$ with respect to the $\otimesstar$-monoidal structure.

\subsubsection{} For any (positive) integer $n$, let
\[
	j: (\Ran X)^n_{\disj} \to (\Ran X)^n
\]
where $(\Ran X)^n_\disj$ is the open sub-prestack of $(\Ran X)^n$ such that for any scheme $S$, $(\Ran X)^n_\disj(S)$ consists of $n$ non-empty subsets of $X(S)$ whose graphs are pair-wise disjoint.

Let $\matheur{A} \in \ComAlgstar(\Ran X)$, then for any $n$, we have the multiplication map
\[
	\union_!(\matheur{A} \boxtimes \cdots \boxtimes \matheur{A}) = \matheur{A} \otimesstar \matheur{A} \otimesstar \cdots\otimesstar \matheur{A} \to \matheur{A}
\]
which is equivalent to a map
\[
	\matheur{A} \boxtimes \cdots \boxtimes \matheur{A} \to \union^! \matheur{A}.
\]
Applying $j^!$ to both sides, we get a map of sheaves on $(\Ran X)^n_\disj$
\[
	j^!(\matheur{A} \boxtimes \cdots \boxtimes \matheur{A}) \to (j\circ \union)^! \matheur{A}. \teq\label{eq:com_fact_natural_map}
\]

\begin{defn}
Let $\matheur{A} \in \ComAlgstar(\Ran X)$. We say that $\matheur{A}$ is a commutative factorization algebra if the map~\eqref{eq:com_fact_natural_map} is an equivalence for each $n$. We use $\Factstar(X)$ to denote the full subcategory of $\ComAlgstar(\Ran X)$ consisting of commutative factorization algebras.
\end{defn}

\section{$\coBarP_{\matheur{P}}$ as a sequential limit}
\label{sec:coBar_as_limit}

Since the functor $\coChev$, or more generally, the functor $\coBarP_{\matheur{P}}$ for a co-operad $\matheur{P}$, is defined abstractly using operadic Koszul duality, it is too opaque to study directly conveniently. The main result of this section is Proposition~\ref{prop:coBar_as_limit}, which says that inside a pro-nilpotent category, $\coBarP_{\matheur{P}} A$ could be written as a sequential limit, such that the direct sum of successive fibers is, up to a cohomological shift, naturally identified with the free $\matheur{P}^\vee$ object generated by $\oblv_{\matheur{P}} A$, the underlying object of the coalgebra $A$. Many questions about $\coBarP_{\matheur{P}}$, such as those about homological stability, could thus be reduced to the case where $A$ has a trivial $\matheur{P}$-coalgebra structure, which is usually much easier. For example, in the case where $\matheur{P} = \coLie$, the free object is a symmetric algebra, of which we have complete control (see Corollary~\ref{cor:coChev_as_limit}).

In the dual situation, where one wants to understand the functor $\BarO_{\matheur{O}}$ for an operad $\matheur{O}$, a general strategy was given in~\cite{gaitsgory_study_2017}, which goes by the name ``the $\addFil$ trick.'' Unfortunately, the strategy presented in~\cite{gaitsgory_study_2017}*{Vol. II, Chapter 6, \S1,2} cannot be used directly. The main difficulty in the co-operad case lies mainly in the fact that the functor $\coBarP_{\matheur{P}}$ is defined using limits, and most of our functors do not commute with limits.

A similar difficulty appears in~\cite{ho_atiyah-bott_2020}. There, it was overcome by imposing a certain homological stability condition with respect to a certain $t$-structure to make sure that the limits are well-behaved. In this section, we will give a variation of the strategy there. In our case, the situation is better, since our (co-)algebras live in a pro-nilpotent category, and so, our stability condition does not involve any $t$-structure, but instead, is expressed in terms of the limit defining the pro-nilpotent category. As a result, the final statement works without any restriction on the object.

In~\S\ref{subsec:cofiltered_graded_generalities}, we set up the general framework of co-filtered and graded objects, which helps formalizing the notion of taking limits and successive fibers (which we will call associated-graded below) mentioned above. In~\S\ref{subsec:stabilizing_decaying_sequences}, we will introduce the stability and decay conditions on co-filtered and graded objects respectively. These technical conditions enable us to have control over infinite limits. We finally prove our main results in~\S\ref{subsec:cofil_coBar}.

Since the materials presented in this section are rather technical, while the technique being used is unrelated to the rest of the paper, we recommend that the reader take Corollary~\ref{cor:coChev_as_limit} as a blackbox at the first pass.

\subsection{Co-filtered and graded objects}
\label{subsec:cofiltered_graded_generalities}
In this subsection, we will introduce the notation for co-filtered and graded objects as well as basic facts about them. These are used to formulate the idea of taking limit and associated-graded. The discussion here is parallel to that of~\cite{gaitsgory_study_2017}*{Vol. II, Chapter 6 \S1,2}.

\begin{notation}
Let $\matheur{C}$ be a stable $\infty$-category. Then the category of co-filtered objects in $\matheur{C}$ is defined to be the category of functors
\[
	\matheur{C}^{\coFil} = \Fun(\mathbb{Z}^{\to, \op}, \matheur{C}),
\]
and the category of graded objects
\[
	\matheur{C}^{\gr} = \Fun(\mathbb{Z}, \matheur{C}).
\]
Here, the notation $\mathbb{Z}^{\to}$ is used to denote the category obtained from considering the usual ordering of $\mathbb{Z}$.
\end{notation}

\begin{notation}
We will use $\matheur{C}^{\coFil^{>0}}$ and $\matheur{C}^{\gr^{>0}}$ to denote the full subcategories of $\matheur{C}^{\coFil}$ and $\matheur{C}^{\gr}$ respectively of co-filtered and graded objects concentrated in only positive degrees.
\end{notation} 

\subsubsection{} We define
\[
	\assgr: \matheur{C}^{\coFil} \to \matheur{C}^\gr
\]
to be a functor of taking the associated graded object
\[
	\assgr(V)_n = \Fib(V_n \to V_{n-1}).
\]

\begin{rmk}
Later on, we will consider another functor of taking associated graded object of a filtered (rather than co-filtered) object. That functor will be defined as a $\coFib$ rather than $\Fib$. It should be clear from the context which functor we are using.
\end{rmk}

\subsubsection{}
We define
\[
	\oblv_{\coFil}: \matheur{C}^{\coFil} \to \matheur{C}
\]
to be the functor of taking limit
\[
	\oblv_{\coFil}(V) = \lim_{n\in \mathbb{Z}^{\to, \op}} V_n,
\]
where
\[
	V = (\cdots V_{n+1} \to V_n \to V_{n-1} \to \cdots) \in \matheur{C}^{\coFil}.
\]

Similarly, we define
\[
	\prod: \matheur{C}^{\gr} \to \matheur{C}
\]
to be the functor of taking product
\[
	\prod(V) = \prod_{n\in \mathbb{Z}} V_n
\]
where $V = (V_n)_{n\in \mathbb{Z}} \in \matheur{C}^{\gr}$.

Note that both functors can be implemented as the right Kan extensions along $\mathbb{Z}^{\to, \op} \to \pt$ and $\mathbb{Z} \to \pt$ respectively, where $\pt$ is the trivial category.

\subsubsection{} When $\matheur{C}$ is a symmetric monoidal category, then so are $\matheur{C}^{\coFil}$ and $\matheur{C}^{\gr}$ and the functor $\assgr$ is symmetric monoidal. Note that here, we use the right Kan extension to construct these monoidal structures, i.e. 
\[
	(V \otimeshat W)_n \simeq \lim_{i+j \leq n} V_i \otimes W_j \in \matheur{C}^{\coFil} \teq\label{eq:otimes_hat_for_cofiltered}
\]
when $V, W\in C^{\coFil}$ and
\[
	(V \otimeshat W)_n \simeq \prod_{i+j=n} V_i \otimes W_j. \teq\label{eq:otimes_hat_for_gr}
\]
We use the notation $\otimeshat$ to distinguish the monoidal structures obtained by taking limits from the usual ones.

\subsubsection{}
The limits appearing in~\eqref{eq:otimes_hat_for_cofiltered} and~\eqref{eq:otimes_hat_for_gr} are infinite in general. However, restricted to $\matheur{C}^{\coFil^{>0}}$ and $\matheur{C}^{\gr^{>0}}$, these are finite limits. The monoidal products in these cases, therefore, commute with colimits in each variable. Moreover, in the case of $\matheur{C}^{\gr^{>0}}$, the product is equivalent to a finite direct sum, which is how the usual $\otimes$-monoidal structure on $\matheur{C}^{\gr^{>0}}$ is defined. In other words, $\otimeshat \simeq \otimes$ for $\matheur{C}^{\gr^{>0}}$. 

\subsubsection{} The functor $\oblv_{\coFil}$ is not symmetric monoidal in general, even when we restrict ourselves to the case of $\matheur{C}^{\coFil^{>0}}$, due to the fact that
\[
\otimes: \matheur{C} \times \matheur{C} \to \matheur{C}
\]
does not commute with infinite limits in each variable in general. Indeed,
\[
	\oblv_{\coFil} (V\otimeshat W) \simeq \lim_{n\in \mathbb{Z}^{\to, \op}} (V\otimeshat W)_n \simeq \lim_{n\in \mathbb{Z}^{\to, \op}}\lim_{i+j\leq n} V_i\otimes W_j \simeq \lim_{i, j \in \mathbb{Z}^{\to, \op}} V_i\otimes W_j, \teq\label{eq:tensor_computation_oblvcoFil}
\]
but the last term is not equivalent to
\[
	(\lim_{i \in \mathbb{Z}^{\to, \op}} V_i)\otimes (\lim_{j\in \mathbb{Z}^{\to, \op}} W_j) \simeq (\oblv_{\coFil} V) \otimes (\oblv_{\coFil} W)
\]
in general.

Using the same reasoning, we see that the functor $\prod$ is not symmetric monoidal in general.

\subsubsection{} These are the main reasons why we cannot directly apply the $\addFil$ trick of~\cite{gaitsgory_study_2017} in this setting. The fix is to consider certain full sub-categories of $\matheur{C}^{\coFil}$ and $\matheur{C}^{\gr}$ which satisfy some stability/finiteness condition to ensure that the limits are well-behaved. This is exactly what we are doing next.

\subsection{Stabilizing and decaying sequences}
\label{subsec:stabilizing_decaying_sequences}
We will now introduce the stability conditions required to make the limits involved well-behaved.

\subsubsection{} Let $\matheur{C}$ be a pro-nilpotent category (see Definition~\ref{defn:pro_nilpotent}). In particular, it can be exhibited as a limit
\[
	\matheur{C} \simeq \lim_{p\in \mathbb{N}^\op} \matheur{C}_p, \teq\label{eq:pro_nilp_lim}
\]
where all the functors involved in the limit are symmetric monoidal and commute with both limits and colimits.

For each $p\in \mathbb{N}$, we will use
\[
	\ev_p: \matheur{C} \to \matheur{C}_p
\]
to denote the canonical projection/evaluation functor. Note that this functor is symmetric monoidal, and moreover, it commutes with both limits and colimits.

For an object $V$ in $\matheur{C}^{\coFil}$ or $\matheur{C}^{\gr}$, we use $V_{n, p}$ to denote $\ev_p(V_n)$.

\begin{defn}
Let $\matheur{C}$ be as above. Then, a co-filtered object $V\in \matheur{C}^{\coFil^{>0}}$ is said to stabilize if for all $p$, the map
\[
	(\ev_p(V_{n+1}) \to \ev_p(V_n)) = \ev_p(V_{n+1}\to V_n)
\]
is an equivalence when $n\gg 0$.

A graded object $V\in \matheur{C}^{\gr^{>0}}$ is said to be decaying if for all $p$,
\[
	\ev_p(V_n) \simeq 0
\]
when $n\gg 0$.
\end{defn}

\begin{notation}
We use $\matheur{C}^{\coFil^{>0}, \stab}$ and $\matheur{C}^{\gr^{>0}, \decay}$ to denote the full-subcategories of $\matheur{C}^{\coFil^{>0}}$ and $\matheur{C}^{\gr^{>0}}$ consisting of stabilizing and decaying objects respectively.
\end{notation}

\begin{rmk}
Note that the notions of stabilizing and decaying objects depend on the presentation of $\matheur{C}$ as a limit as in~\eqref{eq:pro_nilp_lim}. We will elide this dependence from the notation unless confusion is likely to happen.
\end{rmk}

The following result is tautological.
\begin{lem} \label{lem:stab=decay}
Let $\matheur{C}$ be as above and $V\in \matheur{C}^{\coFil^{>0}}$. Then
\[
	V \in \matheur{C}^{\coFil^{>0}, \stab}
\]
if and only if
\[
	\assgr(V) \in \matheur{C}^{\gr^{>0}, \decay}.
\]
\end{lem}

As a consequence, we have the following
\begin{cor}
The categories $\matheur{C}^{\gr^{>0}, \decay}$ and $\matheur{C}^{\coFil^{>0}, \stab}$ are closed under tensor products in $\matheur{C}^{\gr^{>0}}$ and $\matheur{C}^{\coFil^{>0}}$ respectively. In other words, they are symmetric monoidal categories.
\end{cor}
\begin{proof}
By Lemma~\ref{lem:stab=decay} above and the fact that $\assgr$ is symmetric monoidal, it suffices to show the statement for $\matheur{C}^{\gr^{>0}, \decay}$. However, this is direct from the way tensors are computed in $\matheur{C}^{\gr^{>0}}$. Indeed, let $V, W \in \matheur{C}^{\gr^{>0}, \decay}$. We want to show that for any fixed $p$,
\[
	 (V\otimeshat W)_{n, p} \simeq \prod_{n_1 + n_2 = n} V_{n_1, p} \otimes W_{n_2, p} \simeq 0
\]
when $n \gg 0$. But this is clear since when $n\gg 0$, either $n_1 \gg 0$ or $n_2 \gg 0$, and we are done.
\end{proof}

Finally, using the computation at~\eqref{eq:tensor_computation_oblvcoFil}, we immediately get the following result promised earlier.
\begin{lem} \label{lem:oblv_coFil_sym_mon}
The functor
\[
	\oblv_{\coFil}: \matheur{C}^{\coFil^{>0}, \stab} \to \matheur{C}
\]
is symmetric monoidal.
\end{lem}
\begin{proof}
We want to show that for each $V, W \in \matheur{C}^{\coFil^{>0}, \stab}$, the natural map
\[
	\oblv_{\coFil} V \otimes \oblv_{\coFil} W \to \oblv_{\coFil}(V \otimeshat W)
\]
is an equivalence. It suffices to show that this is an equivalence after evaluating on $\matheur{C}_p$ for each $p$. But now, the stability condition allows us to turn the limit in the last term of~\eqref{eq:tensor_computation_oblvcoFil} inside the tensor, since we can just take $i$ and $j$ big enough there and remove the limit altogether.
\end{proof}

\subsubsection{} The decay condition on $\matheur{C}^{\gr^{>0}, \decay}$ also makes the functor
\[
\prod: \matheur{C}^{\gr^{>0}, \decay} \to \matheur{C}
\]
well-behaved.

\begin{lem}
	We have the following natural equivalence of functors
	\[
		\bigoplus \simeq \prod: \matheur{C}^{\gr^{>0}, \decay} \to \matheur{C}.
	\]
	In particular, the functor $\prod$ commutes with colimits and is symmetric monoidal
\end{lem}
\begin{proof}
	It suffices to show that we have the following natural equivalence
	\[
		\ev_p \circ \bigoplus \simeq \ev_p \circ \prod
	\]
	for each $p$. However, this is clear since within each $\matheur{C}_p$, the products/direct sums that appear are finite.
\end{proof}

\begin{rmk}
	Note that $\prod$ automatically commutes with limits since it's defined as a limit itself.
\end{rmk}

\subsubsection{} To end this subsection, we will record the following piece of notation: we will use
\[
	\addCoFil: \matheur{C} \to \matheur{C}^{\coFil^{>0}, \stab}
\]
to denote the following functor which sends $V \in \matheur{C}$ to the co-filtered object
\[
	\cdots \leftarrow 0 \leftarrow V \leftarrow V \leftarrow \cdots
\]
where $V$'s first appearance is at degree 1 and all the maps between them are identities.

\subsubsection{} As in~\cite{gaitsgory_study_2017}*{Vol. II, Chapter 6, \S1,2}, for any co-operad $\matheur{P}$, the functor $\addCoFil$ upgrades to a functor
\[
	\addCoFil: \matheur{P}\Pcoalg(\matheur{C}) \to \matheur{P}\Pcoalg(\matheur{C}^{\coFil^{>0}, \stab}),
\]
and moreover, as in~\cite{gaitsgory_study_2017}*{Vol. II, Chapter 6, Prop. 1.4.6} the following diagram commutes
\[
\xymatrix{
	\matheur{P}\Pcoalg(\matheur{C}) \ar[r]^<<<<<{\addCoFil} \ar[d]_{\oblv_{\matheur{P}}} & \matheur{P}\Pcoalg(\matheur{C}^{\coFil^{>0}, \stab}) \ar[r]^{\assgr} & \matheur{P}\Pcoalg(\matheur{C}^{\gr^{>0}, \decay})\ar@{=}[d] \\
	\matheur{C} \ar[r]^<<<<<<<<<<<<{\deg = 1} & \matheur{C}^{\gr^{>0}, \decay} \ar[r]^<<<<<<<<<{\cotriv_{\matheur{P}}} & \matheur{P}\Pcoalg(\matheur{C}^{\gr^{>0}, \decay}) 
} \teq\label{eq:addCoFil_triv}
\]
Here, $(\deg = 1) = \assgr\circ \addCoFil$ denotes the functor of placing an object of $\matheur{C}$ in graded degree $1$.

\subsection{A co-filtration on $\coBarP_\matheur{P}$} \label{subsec:cofil_coBar}
Using the framework setup above, in this subsection, we will show that for any coLie-coalgebra $\mathfrak{a}$ in a pro-nilpotent category $\matheur{C}$, $\coChev \mathfrak{a}$ could be computed as a sequential limit
\[
	\coChev \mathfrak{a} = \lim (\coChev^1 \mathfrak{a} \leftarrow \coChev^2 \mathfrak{a} \leftarrow \cdots) = \oblv_{\coFil} ((\coChev^n \mathfrak{a})_{n\in \mathbb{Z}_{>0}})
\]
for some
\[
	(\coChev^n \mathfrak{a})_{n\in \mathbb{Z}_{>0}} \in \matheur{C}^{\coFil^{>0}, \stab}
\]
such that
\[
	\assgr((\coChev^n \mathfrak{a})_{n \in \mathbb{Z}_{>0}}) \simeq (\Sym^n \mathfrak{a}[-1])_{n\in \mathbb{Z}_{>0}} \in \matheur{C}^{\gr^{>0}, \decay}.
\]

In fact, we will show a general statement, Proposition~\ref{prop:coBar_as_limit}, about $\coBarP_{\matheur{P}}$ for any co-operad $\matheur{P}$, which specializes to the case of $\coChev$ described above, Corollary~\ref{cor:coChev_as_limit}. The following Proposition is the technical statement from which our desired result follows.

\begin{prop} \label{prop:addCoFil_fund}
	Let $\matheur{C}$ be a pro-nilpotent category, $\matheur{P}$ a co-operad with Koszul dual $\matheur{O} = \matheur{P}^\vee$. Then, the following diagram commutes
	\[
	\xymatrix{
		\matheur{P}\Pcoalg(\matheur{C}) \ar[d]_{\addCoFil} \ar[rr]^{\coBarP_{\matheur{P}}} && \matheur{O}\Palg(\matheur{C})  \\
		\matheur{P}\Pcoalg(\matheur{C}^{\coFil^{>0}, \stab}) \ar[d]_{\assgr} \ar[rr]^{\coBarP_{\matheur{P}}} && \ar[u]_{\oblv_{\coFil}} \matheur{O}\Palg(\matheur{C}^{\coFil^{>0}, \stab}) \ar[d]^{\assgr} \\
		\matheur{P}\Pcoalg(\matheur{C}^{\gr^{>0}, \decay}) \ar[rr]^{\coBarP_{\matheur{P}}} \ar[d]_{\bigoplus \simeq \prod} && \matheur{O}\Palg(\matheur{C}^{\gr^{>0}, \decay}) \ar[d]^{\bigoplus \simeq \prod} \\
		\matheur{P}\Pcoalg(\matheur{C}) \ar[rr]^{\coBarP_{\matheur{P}}} && \matheur{O}\Palg(\matheur{C})
	} \teq \label{eq:add_cofil_trick_fund}
	\]
\end{prop}
\begin{proof}
The first thing we need to check is that this diagram actually makes sense. Namely, the functors involved land in the correct categories. By inspection, we see that the only cases where it is not immediate are the middle two instances of $\coBarP_{\matheur{P}}$. By Lemma~\ref{lem:stab=decay}, it suffices to check the statement for the third instance of $\coBarP_{\matheur{P}}$ (from top down). For that, it suffices to check that for each
\[
	A \in \matheur{P}\Pcoalg(\matheur{C}^{\gr^{>0}, \decay})
\]
and each $p \in \mathbb{Z}_+$,
\[
	\ev_p(\coBarP_\matheur{P}(A)_n) =  \coBarP_\matheur{P}(A)_{n, p} \simeq 0
\]
when $n \gg 0$.

Let $\coBarP^\bullet_{\matheur{P}}(A)$ be the co-simplicial object defining $\coBarP_\matheur{P}(A)$. Then since the monoidal structure on $\matheur{C}_p$ is nilpotent by assumption, only finite numbers of tensor powers of $A$ appear in
\[
	\ev_p(\coBarP_\matheur{P}^\bullet(A)).
\]
The decay condition on $A$ then implies that
\[
	\coBarP^\bullet_\matheur{P}(A)_{n, p} \simeq 0
\]
when $n \gg 0$, and we are done.

Now, observe that the bottom two squares of~\eqref{eq:add_cofil_trick_fund} commute since the vertical arrows are all symmetric monoidal and co-continuous (i.e. commute with limits). It remains to show that the top square commutes. This follows from the equivalence
\[
	\id \simeq \oblv_{\coFil}\circ \addCoFil
\]
and the following commutative diagram
\[
\xymatrix{
	\matheur{P}\Pcoalg(\matheur{C}) \ar[rr]^{\coBarP_{\matheur{P}}} && \matheur{O}\Palg(\matheur{C})  \\
	\matheur{P}\Pcoalg(\matheur{C}^{\coFil^{>0}, \stab}) \ar[rr]^{\coBarP_{\matheur{P}}} \ar[u]^{\oblv_{\coFil}} && \matheur{O}\Palg(\matheur{C}^{\coFil^{>0}, \stab}) \ar[u]_{\oblv_{\coFil}} 
}
\]
which is due to the fact that $\oblv_{\coFil}$ is also co-continuous, being a right adjoint, and symmetric monoidal, by Lemma~\ref{lem:oblv_coFil_sym_mon}.

This concludes the proof of Proposition~\ref{prop:addCoFil_fund}.
\end{proof}

\begin{rmk}
	If we remove the stability conditions, the ``dual'' diagram of~\eqref{eq:add_cofil_trick_fund} (i.e. replacing co-operad by operad, $\BarO$ by $\coBarP$ etc.) is the one appearing in~\cite{gaitsgory_study_2017}, which makes the trick of adding a filtration work. Here, the stability conditions are what needed to make the same diagram commute in the co-operad setting.
\end{rmk}

\subsubsection{} We will now come to the main result of this section.

\begin{prop}
\label{prop:coBar_as_limit}
Let $\matheur{C}$ be a pro-nilpotent category, $\matheur{P}$ a co-operad with Koszul dual $\matheur{O}$, and $A \in \matheur{P}\Pcoalg(\matheur{C})$. Then, $\coBarP_{\matheur{P}} A$ could be canonically written as a limit
\[
	\coBarP_{\matheur{P}} A \simeq \lim_{n\in \mathbb{Z}_{>0}^{\to, \op}} \coBarP_{\matheur{P}}^n A
\]
for some
\[
	(\coBarP_\matheur{P}^n A)_{n\in \mathbb{Z}_{>0}} \in \matheur{O}\Palg(\matheur{C}^{\coFil^{>0}, \stab})
\]
such that
\[
	\assgr((\coBarP^n_{\matheur{P}} A)_{n\in \mathbb{Z}_{>0}}) \simeq \Free_{\matheur{O}}(\oblv_{\matheur{P}}(A)^{\deg=1}) \in \matheur{O}\Palg(\matheur{C}^{\gr^{>0}, \decay}).
\]
\end{prop}
\begin{proof}
The top square of~\eqref{eq:add_cofil_trick_fund} implies that for any $A \in \matheur{P}\Pcoalg(\matheur{C})$, $\coBarP_{\matheur{P}} A$ could be written as a limit
\[
	\coBarP_{\matheur{P}} A \simeq \oblv_{\coFil} \circ \coBarP_{\matheur{P}} \circ \addCoFil(A) \simeq \lim_{n \in \mathbb{Z}_{>0}^{\to, \op}} (\coBarP^n_\matheur{P} A)
\]
for some
\[
	(\coBarP_\matheur{P}^n A)_{n\in \mathbb{Z}_{>0}} \in \matheur{O}\Palg(\matheur{C}^{\coFil^{>0}, \stab}).
\]
Moreover, the bottom two squares of~\ref{eq:add_cofil_trick_fund}, coupled with~\eqref{eq:addCoFil_triv} imply that
\[
	\assgr((\coBarP_{\matheur{P}}^n A)_{n\in \mathbb{Z}_{>0}}) \simeq \coBarP_\matheur{P}(\cotriv_{\matheur{P}}\circ \oblv_\matheur{P}(A)^{\deg = 1}) \simeq \Free_\matheur{O}(\oblv_{\matheur{P}}(A)^{\deg=1}).
\]
\end{proof}

Specializing to the case where $\matheur{P} = \coLie$ we get the desired result stated at the beginning of this subsection.

\begin{cor}
\label{cor:coChev_as_limit}
Let $\matheur{C}$ be a pro-nilpotent category and $\mathfrak{a} \in \coLie(\matheur{C})$. Then $\coChev \mathfrak{a}$ could be canonically written as a limit
\[
	\coChev \mathfrak{a} \simeq \lim_{n\in \mathbb{Z}^{\to, \op}_{>0}} \coChev^n \mathfrak{a}
\]
for some
\[
	(\coChev^{n} \mathfrak{a})_{n\in \mathbb{Z}_{>0}} \in \ComAlg(\matheur{C}^{\coFil^{>0}, \stab})
\]
such that
\[
	\assgr((\coChev^n \mathfrak{a})_{n\in \mathbb{Z}_{>0}}) \simeq \Sym (\mathfrak{a}^{\deg =1}[-1])_+ \in \ComAlg(\matheur{C}^{\gr^{>0}, \decay}).
\]
\end{cor}

\begin{rmk}
\label{rmk:limit_coBar_converges_strongly}
The limit in Proposition~\ref{prop:coBar_as_limit}, and hence also the one in Corollary~\ref{cor:coChev_as_limit}, ``converges in a strong sense.'' Namely, for each $p$, we have
\[
	\ev_p(\coBarP_{\matheur{P}} A) \simeq \ev_p(\lim_{n\in \mathbb{Z}_{>0}^{\to, \op}} \coBarP_{\matheur{P}}^n A) \simeq \ev_p(\coBarP_{\matheur{P}}^N A)
\]
for some $N\gg 0$ due to the stabilizing condition.
\end{rmk}

\section{Some generalities on homological stability} \label{sec:generalities_homological_stability}

In this section, we set up the general formalism (subsections~\S\ref{subsec:filtered_and_graded}--\S\ref{subsec:stabilization}) and prove general results (subsections~\S\ref{subsec:relation_to_hom_stab}--\S\ref{subsec:stable_homology_simple}) needed for the discussion of homological stability. The upshot of this section is that stable homology (when homological stability does occur) can be captured by taking the colimit along the stabilizing morphisms, a process we call stabilization. The latter, however, makes sense regardless of whether homological stability occurs or not, and in fact, it is the analog of the group completion operation in spaces. Moreover, being a categorical construction operating at the chain level, it enjoys nice categorical properties. We should, therefore, think of it as the generalization of stable homology, and study it carefully, even if the ultimate aim is to understand homological stability and stable homology.

We will now give a detailed overview of the section. We start, in~\S\ref{subsec:filtered_and_graded} and \S\ref{subsec:augmented_unital_algebras}, with the fundamental building block, the notion of a (multi-)graded/filtered object/algebra in a given category. Here, the grading is given by $\mathbb{Z}^m_{\geq 0}$ rather than just $\mathbb{Z}$. In the geometric applications that we have in mind, these gradings come from the natural gradings on the spaces $(\gConf{\mathbf{d}}{\mathbf{n}}(X))_{\mathbf{d} \in \multgr{m}}$. The discussion here is only a slight extension of~\cite{gaitsgory_study_2017}*{Vol. II, Chapter 6, \S1.3}, where the case of $\mathbb{Z}$-graded and filtered case is treated. 

In~\S\ref{subsec:graded_unital_graded_augmented}, we introduce the notion of a graded-unital algebra, a special kind of graded algebras which upgrades to a filtered algebra so that questions about homological stability can be discussed. \S\ref{subsec:change_of_gradings} discusses the formalism needed to compare objects filtered/graded by different gradings. Finally, in~\S\ref{subsec:stabilization}, the notion of stabilization of a filtered object is introduced. In particular, in~\S\ref{subsubsec:telescope_construction}, we explain how one stabilizes graded-unital algebras introduced in~\S\ref{subsec:graded_unital_graded_augmented}.

The highlight of this section comes in~\S\ref{subsec:relation_to_hom_stab}, which proves that this stabilization captures stable homology (when homological stability does occur). In~\S\ref{subsec:homological_stability_of_a_limit}, we provide a general strategy for attacking homological stability questions for objects presented as a sequential limit, whose successive fibers are known. Coupled with the result in~\S\ref{sec:coBar_as_limit}, this will be used in~\S\ref{sec:cohomological_chevalley_complex} to prove homological stability for the cohomological Chevalley complexes of a large class of $\coLie$-coalgebras by reducing to the case of trivial $\coLie$-coalgebras. In the last subsection~\S\ref{subsec:stable_homology_simple}, we provide a calculation of the stable homology in a particularly simple case.

\subsection{Filtered and graded objects} 
\label{subsec:filtered_and_graded}
In this subsection, we introduce the categories of (multi-)filtered and (multi-)graded objects in a given category. These categories provide the natural setting where one can discuss homological stability. The readers might find the discussion here parallel to that of~\S\ref{subsec:cofiltered_graded_generalities}. Indeed, this subsection deals with a slight generalization of the dual of the picture presented there, and is, in turn, a slight generalization of~\cite{gaitsgory_study_2017}*{Vol. II, Chapter 6, \S1.3}.

\subsubsection{} We start with the indexing categories. Let
\[
	\multgrplus{m} = \multgr{m} - \{0\}
\]
be the sub-semigroup of $\multgr{m}$. It can be viewed as a (discrete) symmetric monoidal category, where the monoidal structure is given by addition.

Consider also the symmetric monoidal category $\multfil{m}$ whose objects are the same as for $\multgr{m}$, morphisms given by the natural partially ordered structure, and the monoidal structure given by the monoid structure on $\multgr{m}$. Similarly, consider the full (non-unital) monoidal subcategory
\[
	\multfilplus{m} = \multfil{m} - \{0\}.
\]

\subsubsection{} \label{subsubsec:notation_graded_degree_1k} We will employ the following notation: for each $1 \leq k \leq m$, we will use $\unit_k$ to denote
\[
	(0, \dots, 0, 1, 0, \dots, 0) \in \multgr{m}
\]
where the only ``$1$'' appears in the $k$-th position.

\subsubsection{} Recall the following general construction: for any symmetric monoidal categories $\matheur{C}$ and $\Gamma$ with $\matheur{C}$ being cocomplete, we can associate to them a new symmetric monoidal category
\[
	\matheur{C}^\Gamma = \Fun(\Gamma, \matheur{C})
\]
of all functors from $\Gamma$ to $\matheur{C}$, where the monoidal structure on $\matheur{C}^\Gamma$ is given by Day convolution (see~\cite{lurie_higher_2017-1}*{\S2.2.6}). More explicitly, for $V, W \in \Fun(\Gamma, \matheur{C})$, we have following diagram
\[
\xymatrix{
	\Gamma \times \Gamma \ar[d]_{\otimes} \ar[dr] \ar[r]^{(V, W)} & \matheur{C} \times \matheur{C} \ar[d]_{\otimes} \\
	\Gamma \ar@{.>}[r]^{V\otimes W} & \matheur{C}
}
\]
where $V\otimes W$ is defined to be the left Kan extension of the diagonal arrow along the vertical arrow on the left.

\begin{expl}
When $\Gamma = \mathbb{Z}$ is the (discrete) monoidal category, $\matheur{C}^{\mathbb{Z}}$ is the category of graded objects in $\matheur{C}$. Similarly, when $\Gamma = \mathbb{Z}^{\to}$ is the monoidal category with morphisms given by the ordering on $\mathbb{Z}$, $\matheur{C}^{\mathbb{Z}^\to}$ is the category of filtered objects in $\matheur{C}$. These categories are symmetric monoidal categories in the usual way.
\end{expl}

\subsubsection{} In this paper, we will deal with categories of the form
\[
	\matheur{C}^{\multgr{m}}, \quad \matheur{C}^{\multgrplus{m}}, \quad
	\matheur{C}^{\multfil{m}}, \quad
	\matheur{C}^{\multfilplus{m}} \teq\label{eq:flavors_gr_fil}
\]
where $\matheur{C}$ is a stable symmetric monoidal category. We will call objects in the first two and the last two categories multi-graded and multi-filtered respectively. To keep the terminology simple, we will simply refer to them as graded and filtered objects, unless confusion might arise.

Note that we have the following inclusions (i.e. fully-faithful embeddings) of categories
\[
	\matheur{C}^{\multgrplus{m}} \subset \matheur{C}^{\multgr{m}} \qquad \text{\and} \qquad \matheur{C}^{\multfilplus{m}} \subset \matheur{C}^{\multfil{m}}.
\]
For later use, we record the following straightforward lemma.

\begin{lem}
\label{lem:multgrplus_multfilplus_pronilpotent}
Let $\matheur{C}$ be a stable symmetric monoidal category. Then, the symmetric monoidal categories
\[
	\matheur{C}^{\multgrplus{m}} \quad\text{and}\quad \matheur{C}^{\multfilplus{m}}
\]
are pro-nilpotent.
\end{lem}

\subsubsection{}
When $\matheur{C}$ is a unital symmetric monoidal category with monoidal unit $\munit_{\matheur{C}}$, then so are $\matheur{C}^{\multgr{m}}$ and $\matheur{C}^{\multfil{m}}$. The units are $\munit_{\matheur{C}}$ in graded degrees $\mathbf{0}$ in the first case and the constant object $\munit_{\matheur{C}}$ living in each graded-degree in the second case. The symmetric monoidal categories $\matheur{C}^{\multgrplus{m}}$ and $\matheur{C}^{\multfilplus{m}}$ are, however, non-unital.

\subsubsection{Rees construction}
We have a natural embedding of categories
\[
	\multgr{m} \hookrightarrow \multfil{m}.
\]
Restricting and left-Kan extending along this functor give us a pair of adjoint functors
\[
	\Rees: \matheur{C}^{\multgr{m}} \rightleftarrows \matheur{C}^{\multfil{m}}: \filToGr.
\]

\subsubsection{} When $m=1$, the functor $\Rees$ is indeed the usual Rees construction: for any object $V\in \matheur{C}^{\multgr{m}}$
\[
	\Rees(V)_{d} \simeq \bigoplus_{d' \leq d} V_{d'}
\]
and moreover, when $d_1 \leq d_2$,
\[
	\Rees(V)_{d_1} \simeq \bigoplus_{d'\leq d_1} V_{d'} \to \bigoplus_{d'\leq d_2} V_{d'} \simeq \Rees(V)_{d_2}
\]
is the obvious map.

\begin{lem}
The functor $\Rees$ is symmetric monoidal. As a result, $\filToGr$ is right lax monoidal.\footnote{Right-lax (resp. left-lax) monoidal functors are sometimes called lax (resp. op-lax) monoidal in the literature.}
\end{lem}
\begin{proof}
Since the monoidal structures on both sides are defined by left-Kan extension, $\Rees$ is symmetric monoidal, and hence, its right adjoint, $\filToGr$ is right-lax monoidal.
\end{proof}

The same discussion applies equally well to $\matheur{C}^{\multgrplus{m}}$ and $\matheur{C}^{\multfilplus{m}}$.

\subsubsection{Associated-graded}
\label{subsubsec:associated_graded}
One common way to extract information out of a $\filtered$-filtered objects in a stable $\infty$-category is to take its associated-graded. This functor commutes with both limits, colimits, and moreover, it is conservative. In addition, when $\matheur{C}$ is symmetric monoidal, so is $\assgr$ (see~\cite{gaitsgory_study_2017}*{Vol. II, Chapter 6, \S1.3}).

\subsubsection{}
We will now construct a generalization of this functor in the case of multi-filtered objects. First, consider the following functor
\[
	(\grToFiltriv)_{\matheur{C}, m}: \matheur{C}^{\multgr{m}} \to \matheur{C}^{\multfil{m}}
\]
which sends an object $V = (V_\mathbf{d})_{\mathbf{d}\in \multgr{m}}$ to a filtered object $(V_\mathbf{d})_{\mathbf{d}\in \multgr{m}}$ where all the maps are 0. Clearly, this functors commutes with both limits and colimits, and hence, it admits a left adjoint, which we will denote by
\[
	\assgr_{\matheur{C}, m}: \matheur{C}^{\multfil{m}} \to \matheur{C}^{\multgr{m}}.
\]

Unless confusion is likely to occur, we will omit the subscript $\matheur{C}$, or $m$, or both, from our notation.

\subsubsection{} 
The functor $\assgr_{\matheur{C}, m}$ we just defined is indeed closely related to the usual functor of taking associated graded. In fact, as we will soon see, it is the functor of taking successive associated graded with respect to all the gradings. This will help us establish nice formal properties of $\assgr_{\matheur{C}, m}$.

First, observe that in the case where $m=1$,
\[
	\assgr_{\matheur{C}}: \matheur{C}^{\filtered} \to \matheur{C}^{\graded}
\]
is indeed the functor of taking the associated-graded object of a filtered object. Namely,
\[
	\assgr(V)_d = \coFib(V_{d-1} \to V_d)
\]
where $V = (V_d)_{d\in \graded} \in \matheur{C}^{\filtered}$.

Next, observe that for any $m$, 
\[
	\matheur{C}^{\multgr{m}} \simeq (\matheur{C}^{\multgr{m-1}})^{\graded} \simeq (\matheur{C}^{\graded})^{\multgr{m-1}} \quad\text{and}\quad \matheur{C}^{\multfil{m}} \simeq (\matheur{C}^{\multfil{m-1}})^{\filtered} \simeq (\matheur{C}^{\filtered})^{\multfil{m-1}}.
\]
Since compositions of left adjoints are left adjoints, we have the following decomposition of $\assgr_{\matheur{C}, m}$
\[
	\assgr_{\matheur{C}, m} \simeq \assgr_{\matheur{C}^{\graded}, m-1} \circ \assgr_{\matheur{C}^{\multfil{m-1}}},
\]
which means that $\assgr_{\matheur{C}, m}$ could be computed by taking successive associated-graded in the usual sense.

By induction, this immediately gives us the following results.
\begin{lem}
\label{lem:ass_gr_is_nice}
The functor $\assgr_{\matheur{C}, m}$ commutes with limits, colimits, and is conservative.
\end{lem}

\begin{lem}
When $\matheur{C}$ is symmetric monoidal, so is the functor $\assgr_{\matheur{C}, m}$.
\end{lem}

\subsubsection{} The two statements above also hold for $\matheur{C}^{\multfilplus{m}}$ and $\matheur{C}^{\multgrplus{m}}$. Indeed, this can be seen by viewing these categories as full-subcategories of $\matheur{C}^{\multfil{m}}$ and $\matheur{C}^{\multgr{m}}$ and checking that the functors land in the correct place, which is trivially true. We will use the same notation $\assgr_{\matheur{C}, m}$ to denote the functor of taking associated-graded in this case.

\subsection{Unital and augmented algebras}
\label{subsec:augmented_unital_algebras}
For any symmetric monoidal category $\matheur{C}$ as above, we let
\[
	\ComAlg(\matheur{C}^{\multgr{m}}), \quad
	\ComAlg(\matheur{C}^{\multgrplus{m}}), \quad \ComAlg(\matheur{C}^{\multfil{m}}), \quad \ComAlg(\matheur{C}^{\multfilplus{m}})
\]
be the categories of graded and filtered commutative algebras respectively. Moreover, when the monoidal structure on $\matheur{C}$ is unital, we can consider the categories of unital and augmented algebras, denoted by 
\[
	\ComAlg^{\aug}(-), \quad \ComAlg^{\un}(-), \quad \ComAlg^{\un, \aug}(-),
\]
where ``$-$'' is the place-holder for $\matheur{C}^{\multgr{m}}$ or $\matheur{C}^{\multfil{m}}$.

\subsubsection{} 
As we have seen above, the monoidal structure on $\matheur{C}^{\multgrplus{m}}$ is not unital. However, the fully-faithful embedding
\[
	\matheur{C}^{\multgrplus{m}} \subset \matheur{C}^{\multgr{m}}
\]
brings us to the setting of~\S\ref{subsubsec:fake_un_aug}, which allows us to talk about the category $\ComAlg^{\un, \aug}(\matheur{C}^{\multgrplus{m}})$ of unital, augmented commutative algebra objects in $\matheur{C}^{\multgr{m}}$ whose augmentation ideal lives in $\matheur{C}^{\multgrplus{m}}$.

\subsubsection{} The various categories of algebras mentioned above naturally contain a specific class of polynomial algebras of interest to us: $\Lambda[\multgr{m}]$ and its associated non-unital algebra $\Lambda[\multgr{m}]_+ = \Lambda[\multgrplus{m}]$. Let us explain what we mean by these notations. Let  $\munit_{\matheur{C}}$ denote the monoidal unit of $\matheur{C}$. By $\Lambda[\multgr{m}]$, we mean the free unital commutative algebra in $\ComAlg(\matheur{C}^{\multgr{m}})$ generated by
\[
	\munit_{\unit_1} \oplus \munit_{\unit_2} \oplus \cdots \oplus \munit_{\unit_m} \in \matheur{C}^{\multgr{m}}
\]
i.e. $m$ copies of $\munit_{\matheur{C}}$, sitting in graded-degrees
\[
	\unit_1, \unit_2, \dots, \unit_m \in \multgr{m}.
\]
Note that $\Lambda[\multgr{m}] \in \ComAlg^{\un, \aug}(\matheur{C}^{\multgrplus{m}})$, and its augmentation ideal $\Lambda[\multgr{m}]_+ \simeq \Lambda[\multgrplus{m}]$ is an object of $\ComAlg(\matheur{C}^{\multgrplus{m}})$. In other words, we have the following augmentation map $\Lambda[\multgr{m}] \to \Lambda$ where the target lives in graded degree $\mathbf{0}$.

\subsubsection{}
Consider the object $\Lambda[\multfil{m}] \in \matheur{C}^{\multfil{m}}$, such that at each graded degree $\mathbf{d}$,
\[
	\Lambda[\multfil{m}]_{\mathbf{d}} \simeq \munit_{\matheur{C}, \mathbf{d}}
\]
and all the maps (from the filtration) are equivalences. We see at once that this is the monoidal unit of $\matheur{C}^{\multgr{m}}$ and hence, it has a natural structure of a commutative algebra.

As above, we let $\Lambda[\multfilplus{m}]$ be obtained from $\Lambda[\multfil{m}]$ by removing the $\Lambda$ at graded degree $\mathbf{0}$.

We see at once from the construction that these objects go to $\Lambda[\multgr{m}]$ and $\Lambda[\multgrplus{m}]$ respectively under $\filToGr$. Because of this reason, unless confusion is likely to occur, we will use $\Lambda[\multgr{m}]$ and $\Lambda[\multgrplus{m}]$ in place of $\Lambda[\multfil{m}]$ and $\Lambda[\multfilplus{m}]$ respectively.

\begin{rmk}
Note that $\Lambda[\multgr{m}] \in \ComAlg(\Vect^{\multfil{m}})$ is the monoidal unit of the category $\matheur{C}^{\multfil{m}}$, and is, therefore, a commutative algebra object. We can thus alternatively define $\Lambda[\multgr{m}] \in \ComAlg(\multgr{m})$ as $(\filToGr)(\munit_{\matheur{C}^{\multfil{m}}})$, which is also a commutative algebra object, since $\filToGr$ is right-lax monoidal.
\end{rmk}

\begin{rmk}
In the case where $\matheur{C} = \Vect$ (see~\eqref{eq:Vect_is_Shv(pt)} for the conventions), $\Lambda[\multgr{m}] \in \ComAlg(\Vect^{\multgr{m}})$ is precisely the polynomial algebra generated by $m$ variables sitting in appropriate graded-degrees. For the applications of this paper, it suffices to consider the case where $\matheur{C} = \Vect$. The readers who prefer to think about $\Vect$ rather than a general category $\matheur{C}$ are welcome to do so without losing any content. In fact, the general story goes exactly the same for $\Vect$ as for a general $\matheur{C}$. 
\end{rmk}

\subsubsection{Another take on $\Rees$ and $\filToGr$}
\label{subsubsec:another_take_Rees}
The language of graded algebras and graded modules over a graded algebra allows us to have a convenient interpretation of the functors $\Rees$ and $\filToGr$. Indeed, we have the following well-known symmetric monoidal identification
\[
	\Mod_{\Lambda[\multgr{m}]}(\matheur{C}^{\multgr{m}}) \simeq \matheur{C}^{\multfil{m}}, \teq\label{eq:filt_as_mod_cat_over_gr}
\]
where the multiplication by $\Lambda[\multgr{m}]$ gives rise precisely to the maps in the filtration. Under this identification, $\filToGr$ is given by forgetting the module structure, and $\Rees$ is given by tensoring up $-\otimes \Lambda[\multgr{m}]$.

\begin{rmk}
The identification~\eqref{eq:filt_as_mod_cat_over_gr} can be seen by applying the Barr--Beck--Lurie monadicity theorem,~\cite{lurie_higher_2017}*{Theorem 4.7.3.5}, to the adjunction $\Rees \dashv (\filToGr)$. The fact that it is symmetric monoidal is a consequence of~\cite{lurie_higher_2017}*{Corollary 4.8.5.20}. Alternatively, using the fact that $\Rees$ and $-\otimes \Lambda[\multgr{m}]$ are symmetric monoidal and that~\eqref{eq:filt_as_mod_cat_over_gr} is compatible with these two functors, we see that~\eqref{eq:filt_as_mod_cat_over_gr} is symmetric monoidal when restricted to the essential images of these functors. Moreover, since the symmetric monoidal structures in these two categories commute with colimits in each variable and since these categories are generated under colimits by the images of $\Rees$ and $-\otimes \Lambda[\multgr{m}]$ respectively, we are done.
\end{rmk}

\subsubsection{Another take on $\assgr_{\matheur{C}, m}$}
\label{subsubsec:assgr_as_tensoring}
The description of $\matheur{C}^{\multfil{m}}$ above also gives us a convenient way of interpreting $\assgr$. Indeed, consider the following map of algebras in $\matheur{C}^{\multgr{m}}$
\[
	\Lambda[\multgr{m}] \to \munit_{\matheur{C}, \mathbf{0}}
\]
by modding out the augmentation ideal. The functor $\grToFiltriv$ is given by restricting along this map, and hence, its left adjoint is given by tensoring up, i.e.
\[
	\assgr_{\matheur{C}, m} \simeq - \otimes_{\Lambda[\multgr{m}]} \munit_{\matheur{C}}. \teq\label{eq:assgr_tensor}
\]
Note that here, we have implicitly used the identification
\[
	\matheur{C}^{\multgr{m}} \simeq \Mod_{\munit_{\matheur{C}, \mathbf{0}}}(\matheur{C}^{\multgr{m}})
\]
since $\munit_{\matheur{C}, \mathbf{0}}$ is the monoidal unit of $\matheur{C}^{\multgr{m}}$.

\subsubsection{} The same discussion carries over for $\matheur{C}^{\multgrplus{m}}$ and $\matheur{C}^{\multfilplus{m}}$. Indeed, we still have
\[
	\Mod_{\Lambda[\multgr{m}]}(\matheur{C}^{\multgrplus{m}}) \simeq \matheur{C}^{\multfilplus{m}}
\]
which is a full-subcategory of $\Mod_{\Lambda[\multgr{m}]}(\matheur{C}^{\multgr{m}})$, and the category $\matheur{C}^{\multgrplus{m}}$ is a full-subcategory of $\Mod_{\munit_{\matheur{C}}}(\matheur{C}^{\multgr{m}})$. These two categories are preserved under $\grToFiltriv$ and $\assgr_{\matheur{C}, m}$, and so, we have the same formula in this case as in~\eqref{eq:assgr_tensor}. In particular, it is important to note that we still tensor over $\Lambda[\multgr{m}]$.

\subsection{Graded-unital and graded-augmented}
\label{subsec:graded_unital_graded_augmented}
In this paper, we usual start with a graded object. But to be able to talk about homological stability, we need to have maps between the various graded-degrees, i.e. we need a filtered object. The equivalence~\eqref{eq:filt_as_mod_cat_over_gr} gives us the following convenient way of doing so.

\subsubsection{} Suppose we have a morphism in $\ComAlg^{\un}(\matheur{C}^{\multgr{m}})$
\[
	\Lambda[\multgr{m}] \to \matheur{A}.
\]
Then $\matheur{A}$ acquires the structure of a $\Lambda[\multgr{m}]$-module and hence, the equivalence~\eqref{eq:filt_as_mod_cat_over_gr} implies that $\matheur{A}$ comes from an object in $\ComAlg^\un(\matheur{C}^{\multfil{m}})$. We call such an object \emph{graded-unital}, and let $\ComAlg^{\grun}(\matheur{C}^{\multgr{m}})$ denote the category of all such objects. Since $\Lambda[\multgr{m}]$ is the monoidal unit of $\matheur{C}^{\multfil{m}}$, we have an equivalence
\[
	\ComAlg^{\grun}(\matheur{C}^{\multgr{m}}) \simeq \ComAlg^{\un}(\matheur{C}^{\multfil{m}}).
\]

\subsubsection{} Similarly, we will use $\ComAlg^{\grun, \graug}(\matheur{C}^{\multgr{m}})$ to denote the category of \emph{graded-unital augmented} algebras, where these objects are defined in the obvious way. As above, we have the following equivalence
\[
	\ComAlg^{\grun, \graug}(\matheur{C}^{\multgr{m}}) \simeq \ComAlg^{\un, \aug}(\matheur{C}^{\multfil{m}}).
\]

\subsubsection{} We have a natural forgetful functor
\[
	\ComAlg^{\grun, \graug}(\matheur{C}^{\multgr{m}}) \to \ComAlg^{\un, \aug}(\matheur{C}^{\multgr{m}}).
\]
We can thus define the category $\ComAlg^{\grun, \graug}(\matheur{C}^{\multgrplus{m}})$ to be the full-subcategory of $\ComAlg^{\grun, \graug}(\matheur{C}^{\multgr{m}})$ spanned by objects whose augmentation ideals live in $\matheur{C}^{\multgrplus{m}}$. Namely, we have the following pullback square of categories
\[
\xymatrix{
	\ComAlg^{\grun, \graug}(\matheur{C}^{\multgrplus{m}}) \ar[d] \ar[r] & \ComAlg^{\grun, \graug}(\matheur{C}^{\multgr{m}}) \ar[d] \\
	\ComAlg^{\un, \aug}(\matheur{C}^{\multgrplus{m}}) \ar[r] & \ComAlg^{\un, \aug}(\matheur{C}^{\multgr{m}})
}
\]

\subsubsection{} \label{subsubsec:quot_grun_comalg} Let $\matheur{A} \in \ComAlg^{\grun}(\matheur{C}^{\multgr{m}})$, one can ``quotient out'' the graded-unit by taking the tensor product, or equivalently, by taking the associated graded (see~\ref{subsubsec:assgr_as_tensoring})
\[
	\assgr \matheur{A} \simeq \matheur{A} \otimes_{\Lambda[\multgr{m}]} \munit_\matheur{C} \in \ComAlg^{\un}(\matheur{C}^{\multgr{m}})
\]
where $\munit_{\matheur{C}}$ on the right sits in graded-degree $(0, 0, \dots, 0)$, i.e. it is the monoidal unit $\munit_{\matheur{C}^{\multgr{m}}}$. 

When we want to emphasize the perspective that we are taking the quotient by the unit rather than just taking the associated graded, we will use $\Quot_{\grun}$ to denote this functor
\[
	\Quot_{\grun}: \ComAlg^{\grun}(\matheur{C}^{\multgr{m}}) \to \ComAlg^{\un}(\matheur{C}^{\multgr{m}}).
\]

\subsection{Change of gradings}
\label{subsec:change_of_gradings}
One of the goals of the paper is to compare relative tensors of graded algebras. Strictly speaking, however, these algebras live in different categories, since they are graded by different gradings (i.e. different $m$'s, in our notation).\footnote{The algebras we are talking about are the $\alg{m}{n}(X)$ in the introduction.} We thus need a functorial way to tie them together.

\subsubsection{} We have the following monoidal functors
\[
	\add: \multgr{m} \to \graded \quad\text{and}\quad \add^\to: \multfil{m} \to \filtered
\]
obtained by adding all the components together. This gives rise to the following pairs of adjoint functors
\[
	\add_!: \matheur{C}^{\multgr{m}} \rightleftarrows \matheur{C}^{\graded}: \add^!
\]
and
\[
	\add^\to_!: \matheur{C}^{\multfil{m}} \rightleftarrows \matheur{C}^{\filtered}: \add^{\to, !},
\]
obtained by left Kan extending and restricting along $\add$/$\add^\to$ respectively.

\subsubsection{} Since the monoidal structures on these graded and filtered categories are formed by Day's convolutions, which is a left Kan extension, it is clear that $\add_!$ and $\add^\to_!$ are monoidal, which implies that $\add^!$ and $\add^{\to, !}$ are right-lax monoidal.

\begin{lem}
\label{lem:add_!_vs_ass_gr}
In the following diagram,
\[
\xymatrix{
	\matheur{C}^{\multgr{m}} \ardis[dd]^{\grToFiltriv} \ardis[rr]^{\add_!} && \ardis[ll]^{\add^!} \matheur{C}^{\graded} \ardis[dd]^{\grToFiltriv} \\ \\
	\matheur{C}^{\multfil{m}} \ardis[uu]^{\assgr} \ardis[rr]^{\add^\to_!} && \ardis[ll]^{\add^{\to, !}}  \matheur{C}^{\filtered} \ardis[uu]^{\assgr}
}
\]
the right adjoints commute, and hence, so do the left adjoints. In particular,
\[
	\add_! \circ \assgr \simeq \assgr \circ \add^\to_!.
\]
\end{lem}

\subsubsection{} The functor $\add^\to_!$ has the following convenient interpretation in a way that is similar to the case of $\assgr$ considered above. Indeed, consider the following diagram
\[
\xymatrix{
	\Mod_{\Lambda[\multgr{m}]}(\matheur{C}^{\multgr{m}}) \ardis[rr]^{\add_!} \ar@{=}[d] && \ardis[ll]^{\add^!} \Mod_{\Lambda[\multgr{m}]}(\matheur{C}^{\graded}) \ardis[rr]^{-\otimes_{\Lambda[\multgr{m}]} \Lambda[\graded]} && \ardis[ll]^{\res^{\Lambda[\multgr{m}] \to \Lambda[\graded]}} \Mod_{\Lambda[\graded]}(\matheur{C}^{\graded}) \ar@{=}[d] \\
	\matheur{C}^{\multfil{m}} \ardis[rrrr]^{\add^\to_!} &&&& \ardis[llll]^{\add^{\to, !}} \matheur{C}^{\filtered}
}
\]
where, in the second term on the top, by $\Lambda[\multgr{m}]$, we mean $\add_! \Lambda[\multgr{m}]$, but this is only a mild abuse of notation since the two are essentially ``the same'' ring: only the gradings differ. Moreover, the adjunction pair $\add_!\dashv \add^!$ upgrades to one between module categories because $\add_!$ is monoidal and hence, $\add^!$ is right-lax monoidal. Finally, we see easily that the right adjoints commute, i.e.
\[
	\add^{\to, !} \simeq \add^! \circ \res
\]
and hence, so do the left adjoints. Thus, we get the following convenient interpretation of $\add^\to_!$

\begin{lem}
For each $V \in \matheur{C}^{\multfil{m}}$, we have a natural equivalence
\[
	\add^\to_! V \simeq (\add_! V) \otimes_{\Lambda[\multgr{m}]} \Lambda[\graded].
\]
\end{lem}

\begin{rmk}
This lemma, combined with~\S\ref{subsubsec:assgr_as_tensoring}, gives us an alternative proof of Lemma~\ref{lem:add_!_vs_ass_gr}. Namely,
\[
	\assgr(\add_!^\to (V)) \simeq \Lambda \otimes_{\Lambda[\graded]} \Lambda[\graded] \otimes_{\Lambda[\multgr{m}]} \add_! (V) \simeq \add_!(\Lambda\otimes_{\Lambda[\multgr{m}]} V).
\]
\end{rmk}

\subsection{Stabilization} \label{subsec:stabilization}
We will now come to the concept of stabilization. In the simplest case, where $m=1$, given an object
\[
	V = (\cdots \to V_1 \to V_2 \to \cdots) \in \matheur{C}^{\filtered},
\]
where the maps are not necessarily equivalences, we want to construct an object $\lbar{V} \in \matheur{C}$, which, in some sense, captures the stable part of the sequence $V$. From the categorical point of view, the obvious candidate is to take
\[
	\lbar{V} = \colim_{d\in \filtered} V_d,
\]
which happily turns out to capture precisely the stable homology of $V$, even though this construction makes sense in general. In this subsection, we will discuss various formal properties of this construction. The link to homological stability will be discussed in the next subsection.

\subsubsection{} We have the following pair of adjoint functors
\[
	\colim_{\mathbf{d}\in \multfil{m}}: \matheur{C}^{\multfil{m}} \rightleftarrows \matheur{C}: \const_{\multfil{m}},
\]
where the left adjoint is taking colimit and the right adjoint is the constant functor, i.e. sending $V$ to the constant filtered object $V$.

\begin{notation}
For any $V\in \matheur{C}^{\multfil{m}}$, we will use
\[
	\lbar{V} = \colim_{\mathbf{d}\in \multfil{m}} V_\mathbf{d}
\]
to denote the colimit. Moreover, unless confusion might occur, we will write $\const$ instead of $\const_{\multfil{m}}$.
\end{notation}

A couple of remarks are in order.

\begin{rmk} \label{rmk:diagonal_cofinal}
Since the diagonal embedding
\[
	\filtered \to \multfil{m}
\]
is co-final, one reduces the computation of $\lbar{V}$ to that of a sequential colimit.
\end{rmk}

\begin{rmk}
Using the same reasoning, we see that $\lbar{V}$ only depends the behavior of $V$ on $\matheur{C}^{\multfilplus{m}}$. Because of this, we will also use the notation $\lbar{V}$ even for the case where $V\in \matheur{C}^{\multfilplus{m}}$.
\end{rmk}

\subsubsection{} We have the following lemma, which is direct from the fact that colimits commute with colimits and that the definition of the monoidal structures on $\matheur{C}^{\multfil{m}}$ and $\matheur{C}^{\multfilplus{m}}$ is defined using colimits.
\begin{lem}
The functor
\[
	\lbar{(-)}: \matheur{C}^{\multfil{m}} \to \matheur{C}
\]
is symmetric monoidal. As a consequence, its right adjoint $\const$ is right-lax monoidal.
\end{lem}

\begin{cor} \label{cor:adjunction_for_algebras_colim}
We have the following pairs of adjoint functors
\begin{align*}
	\lbar{(-)}: \ComAlg(\matheur{C}^{\multfil{m}}) &\rightleftarrows \ComAlg(\matheur{C}): \const, \\
	\lbar{(-)}: \ComAlg^{\un}(\matheur{C}^{\multfil{m}}) &\rightleftarrows \ComAlg^{\un}(\matheur{C}): \const, \\
	\lbar{(-)}: \ComAlg^{\aug}(\matheur{C}^{\multfil{m}}) &\rightleftarrows \ComAlg^{\aug}(\matheur{C}): \const, \\
	\lbar{(-)}: \ComAlg^{\un, \aug}(\matheur{C}^{\multfil{m}}) &\rightleftarrows \ComAlg^{\un, \aug}(\matheur{C}): \const.
\end{align*}
\end{cor}

\begin{rmk}
Let $\matheur{A}$ be an algebra object in any of the categories appearing on the left side of Corollary~\ref{cor:adjunction_for_algebras_colim}, the algebra $\lbar{\matheur{A}}$ should be thought of as the stable part of $\matheur{A}$. Roughly speaking, this is because when $\matheur{A}$ exhibits a certain homological stability property, then $\lbar{\matheur{A}}$ captures precisely this stable homology. This is the content of \S\ref{subsec:relation_to_hom_stab}.
\end{rmk}

\subsubsection{} Since the functor $\lbar{(-)}$ is defined as a colimit, it behaves nicely with respect to the functor $\add^\to_!$ considered in~\S\ref{subsec:change_of_gradings}, and we have the following
\begin{lem}
\label{lem:add_!_vs_stabilization}
We have the following commutative diagram
\[
\xymatrix{
	\matheur{C}^{\multfil{m}} \ar[r]^{\add^\to_!} \ar[dr]_{\lbar{(-)}} & \matheur{C}^{\filtered} \ar[d]^{\lbar{(-)}} \\
	& \matheur{C}
}
\]
\end{lem}

\subsubsection{}
\label{subsubsec:oblv_gr_def}
From the formal point of view, the functor $\lbar{(-)}$ above has an analog in the graded case, which we now record for later use. Namely, we have the following pair of adjoint functors
\[
	\oblv_{\gr}: \matheur{C}^{\multgr{m}} \rightleftarrows \matheur{C}: \const_{\multgr{m}}
\]
where we see easily that $\oblv_{\gr}$ is the functor of taking direct sum. As above, we will write $\const$ instead of $\const_{\multgr{m}}$ unless confusion is likely to occur.

By the same token as for the functor $\lbar{(-)}$, $\oblv_{\gr}$ commutes with colimits, and is symmetric monoidal. And thus, the adjoint pair $\oblv_{\gr} \dashv \const$ upgrades to adjoint pairs between various categories of commutative algebras over $\matheur{C}^{\multgr{m}}$ and $\matheur{C}$ as in Corollary~\ref{cor:adjunction_for_algebras_colim}.

\subsubsection{Telescope construction}
\label{subsubsec:telescope_construction}
Let $\matheur{A} \in \ComAlg^{\grun}(\matheur{C}^{\multgr{m}})$, i.e. a unital graded commutative algebra equipped with the structure of a $\Lambda[\multgr{m}]$-algebra. Then, by~\S\ref{subsec:graded_unital_graded_augmented}, $\matheur{A}$ can be viewed as an object in $\ComAlg^{\un}(\matheur{C}^{\multfil{m}})$, and hence, we can talk about its stabilization $\lbar{\matheur{A}} \in \ComAlg^{\un}(\matheur{C})$, obtained by a telescope construction. Similarly, if $\matheur{A} \in \ComAlg^{\grun, \graug}(\matheur{C}^{\multgr{m}})$, we get $\lbar{\matheur{A}} \in \ComAlg^{\un, \aug}(\matheur{C})$.

\begin{rmk}
At the level of spaces, we learned this construction from~\cite{mcduff_homology_1976}. In their setting, the authors want to find a model for the group completion of an $E_n$-algebra (naturally graded by the monoid of connected components). Roughly speaking, by choosing a special point on a connected component and using it to act on the original $E_n$-algebra, we can form a sequential diagram. The group completion is then computed as the homotopy colimit of this diagram, which has an explicit model as the telescope construction.

The passage from a graded space to a filtered one using the action of a point is similar to our passage from $\ComAlg^{\grun}(\matheur{C}^{\multgr{m}})$ to $\ComAlg^{\un}(\matheur{C}^{\multfil{m}})$ using the action of $\Lambda[\multgr{m}]$. Moreover, in both cases, the final object is constructed as the homotopy colimit of the resulting diagram. This is the reason why we call this the telescope construction.
\end{rmk}

\subsection{Relation to homological stability} \label{subsec:relation_to_hom_stab}
We will now turn to the link between the stabilization construction via colimit above and stable homology. The main result of this subsection, Proposition~\ref{prop:stable_homology_vs_stabilization}, states that the stabilization $\lbar{V}$ of an object $V\in \matheur{C}^{\multfil{m}}$ captures the stable homology of $V$ when $V$ has homological stability. To start, we need to make sense of what homological stability means. For that, we will assume that $\matheur{C}$ is equipped with a $t$-structure.\footnote{As before, for the applications of this paper, it suffices to consider the case where $\matheur{C} = \Vect$. We choose to state the results in more generality since the proof is the same for $\Vect$ as for the more general case. The readers who prefer to think about $\Vect$ will not miss anything.} The term cohomology will be with respect to this $t$-structure, i.e. for any object $V \in \matheur{C}$, ${}^t\Ho^0(V) = \tau_{\leq 0} \tau_{\geq 0} V$ and more generally, ${}^t\Ho^k(V) = {}^t\Ho^0(V[k])$.

\begin{defn}[Homological stability]
Let $V\in \matheur{C}^{\multfil{m}}$. We say that $V$ satisfies weak homological stability if for any integer $c$, there exists an integer $D$ such that the map
\[
	\tau_{\leq c} V_{(d, d, \dots, d)} \to \tau_{\leq c} V_{(d+1, d+1, \dots, d+1)}
\]
is an equivalence for all $d\geq D$.

We say that $V$ satisfies (strong) homological stability if for any integer $c$, there exists an integer $D$ such that the maps
\[
	\tau_{\leq c} V_{(d_1, d_2, \dots, d_k, \dots, d_m)} \to \tau_{\leq c} V_{(d_1, d_2, \dots, d_k + 1, \dots, d_m)}
\]
are equivalences for all $1\leq k\leq m$ and for all $(d_1, \dots, d_m)$ with $d_k \geq D$.
\end{defn}

It is straightforward to see that (strong) homological stability implies weak homological stability.

\begin{defn}
Let $V \in \matheur{C}^{\multfil{m}}$ satisfy weak homological stability. Then for each integer $c$, we use $\tau_{\leq c}^{\stab} V$ to denote the stable homology of $V$ in cohomological degrees $\leq c$. More explicitly,
\[
	\tau^{\stab}_{\leq c} V = \tau_{\leq c} V_{(d, \dots, d)}
\]
when $d \gg 0$.
\end{defn}

We now come to the main result of this subsection. In what follows, we will assume that taking infinite direct sums in $\matheur{C}$ is $t$-exact.

\begin{prop}
\label{prop:stable_homology_vs_stabilization}
Let $V\in \matheur{C}^{\multfil{m}}$ satisfy weak homological stability. Then, for any integer $c$, the natural map
\[
	\tau_{\leq c}^{\stab} V \to \tau_{\leq c} \lbar{V}
\]
is an equivalence.
\end{prop}
\begin{proof}
Because of Remark~\ref{rmk:diagonal_cofinal}, we can reduce to the case where $m = 1$. By throwing away elements at the beginning of the sequence if necessary, we can further assume that for all $d\in \filtered$, the map
\[
	\tau_{\leq c + 1} V_d \to \tau_{\leq c + 1} V_{d+1}
\]
is an equivalence. The desired conclusion comes from the following lemma.
\end{proof}

\begin{lem}
\label{lem:stable_homology_as_Vbar}
Let $V \in \matheur{C}^{\filteredplus}$ and $c$ an integer such that for all $d\in \filtered$, the natural map
\[
	\tau_{\leq c+1} V_d \to \tau_{\leq c+1}V_{d+1}
\]
is an equivalence. Then, the natural map
\[
	\tau_{\leq c} V_d \to \tau_{\leq c} \lbar{V}
\]
is an equivalence for all $d$. Equivalently, the natural map
\[
	\tau^{\stab}_{\leq c} V \to \tau_{\leq c} \lbar{V}
\]
is an equivalence.
\end{lem}
\begin{proof}
Taking the colimit of the following fiber sequence
\[
	\tau_{\leq c + 1} V_d \to V_d \to \tau_{> c+1} V_d \to \cdots,
\]
we get the following fiber sequence
\[
	\tau^{\stab}_{\leq c+1} V \to \lbar{V} \to \colim_d (\tau_{>c+1} V_d) \to \cdots.
\]

Now, it suffices to show that
\[
	\colim_{d}(\tau_{>c+1} V_d)
\]
lives in cohomological degrees $\geq c + 1$. The fact that this colimit fits into the following cofiber sequence
\[
	\bigoplus_d \tau_{>c+1} V_d \to \bigoplus_d \tau_{> c+1} V_d \to \colim_d (\tau_{> c+1} V_d) \to \cdots
\]
gives us the desired conclusion.
\end{proof}

\begin{rmk}
\label{rmk:stable_homology_as_Vbar}
Note that when $V\in \matheur{C}^{\multfil{m}}$ satisfies weak homological stability, then the $\tau_{\leq c}^{\stab} V$ naturally map into each other and form a tower
\[
	\cdots \to \tau_{\leq c-1}^{\stab} V \to \tau_{\leq c}^{\stab} V \to \tau_{\leq c+1}^{\stab} V \to \cdots
\]
The discussion above implies that if we let
\[
	V^{\stab} = \colim_{c} \tau_{\leq c}^{\stab} V,
\]
then for any $c$,
\[
	\tau_{\leq c}^{\stab} V \simeq \tau_{\leq c} V^{\stab} 
\] 
and moreover, the natural map
\[
	V^{\stab} \to \lbar{V}
\]
is an equivalence.

Being independent from homological stability, however, the general construction of $\lbar{V}$ is more robust, and is generally easier to manipulate. For example, when $V$ is an algebra object, $\lbar{V}$ canonically inherits this algebra structure (see Corollary~\ref{cor:adjunction_for_algebras_colim}). We will see more instances of this phenomenon later.
\end{rmk}

\subsubsection{} Using the result above and Corollary~\ref{cor:adjunction_for_algebras_colim}, we see that if we start with a graded-unital algebra $\matheur{A}$ satisfying weak homological stability. Then, the ``stable homology'' $\lbar{\matheur{A}} \simeq \matheur{A}^{\stab}$ carries a natural unital algebra structure.

\begin{rmk} \label{rmk:abuse_stable_homology}
	Even if an object $V\in \matheur{C}^{\multfil{m}}$ exhibits no homological stability phenomenon, by abuse of terminology, we will still refer to the object $\lbar{V}$ as the stable (co)homology of $V$.
\end{rmk}

\subsection{Homological stability of a limit}
\label{subsec:homological_stability_of_a_limit}
In this subsection, we record a simple lemma which provides a method to study the homological stability of a filtered-object that is itself a sequential limit. It allows us to reduce the study of homological stability of the original object to that of the successive fibers of the terms in the limit.

\begin{lem} \label{lem:homological_stability_of_limit}
Let $V \in \matheur{C}^{\multfil{m}}$ such that 
\[
	V \simeq \lim(\dots \leftarrow V_{-1} \leftarrow V_0 \leftarrow V_1 \leftarrow V_2 \leftarrow \cdots)
\]
where $V_1, V_2, \dots \in \matheur{C}^{\multfil{m}}$, and where for each $\mathbf{d}$, $V_{i, \mathbf{d}} \simeq 0$ when $i\ll 0$.
Let
\[
	F_i = \Fib(V_i \to V_{i-1}) \in \matheur{C}^{\multfil{m}}.
\]
Let $c\in \mathbb{Z}, 1\leq k\leq m$, and $\mathbf{d} \in \multgr{m}$ be such that the following map is an equivalence
\[
	\tau_{\leq c+1} F_{i, \mathbf{d}} \to \tau_{\leq c+1} F_{i, \mathbf{d} + \mathbf{1}_k}, \forall i. \teq\label{eq:map_stability_for_F}
\]

Then, the same is true for
\[
	\tau_{\leq c} V_\mathbf{d} \to \tau_{\leq c} V_{\mathbf{d} + \unit_k}. \teq\label{eq:map_stability_for_V}
\]
\end{lem}

\begin{proof}
Since $\tau_{\leq c} \simeq i_{\leq c} \circ \tr_{\leq c}$ and $i_{\leq c}$ is fully-faithful, it suffices to prove~\eqref{eq:map_stability_for_V} but with $\tau_{\leq c}$ being replaced by $\tr_{\leq c}$. Being a right adjoint, $\tr_{\leq c}$ commutes with limits, and hence, it suffices to show that for each $i$, the following map is an equivalence,
\[
	\tr_{\leq c} V_{i, \mathbf{d}} \to \tr_{\leq c} V_{i, \mathbf{d} + \unit_k},
\]
which, as above, is equivalent to the following map being an equivalence for each $i$
\[
	\tau_{\leq c} V_{i, \mathbf{d}} \to \tau_{\leq c} V_{i, \mathbf{d} + \unit_k}.
\]

Now, we induct on $i$, using the following diagram whose rows are distinguished triangles
\[
\xymatrix{
	F_{i, \mathbf{d}+\unit_k} \ar[r] & V_{i, \mathbf{d} + \unit_k} \ar[r] & V_{i-1, \mathbf{d} + \unit_k} \\
	F_{i, \mathbf{d}} \ar[u] \ar[r] & V_{i, \mathbf{d}} \ar[u] \ar[r] & V_{i-1, \mathbf{d}} \ar[u]
}
\]
The base case is given when $i\ll 0$, where all terms in the diagram vanish and the statement is trivially true. The induction step is routine.
\end{proof}

\subsubsection{}
In practice, the maps~\eqref{eq:map_stability_for_F}, and hence also~\eqref{eq:map_stability_for_V}, are equivalences for a range $\mathbf{d} \geq \mathbf{D}$ for some $\mathbf{D}$. This Lemma is thus really about homological stability of a sequential limit as we mentioned above. The bound is where homological stability, with respect to homological degrees $\leq c$, occurs.

In other words, for the applications that we will encounter, the Lemma above says that if $V \in \matheur{C}^{\multgr{m}}$ is a sequential limit, whose associated graded satisfies homological stability, then so does $V$. Moreover, the homological degree at which we have homological stability (i.e. the number $c$ in the above lemma) for $V$ is only $1$ worse than that of the associated graded (i.e. $c+1$ vs $c$).

We have the following variant of the Lemma above
\begin{lem}
\label{lem:homological_stability_of_limit_improvement}
In the situation of Lemma~\ref{lem:homological_stability_of_limit}, suppose further that
\[
	{}^t\Ho^{c+2}(F_{i, \mathbf{d}}) \to {}^t\Ho^{c+2} (F_{i, \mathbf{d}})
\]
is injective. Then we can improve the homological degree at which stability occurs. Namely, the following map is an equivalence
\[
	\tau_{\leq c+1} V_{\mathbf{d}} \to \tau_{\leq c+1} V_{\mathbf{d}+\unit_k}.
\]
\end{lem}
\begin{proof}
The proof goes the same way as in Lemma~\ref{lem:homological_stability_of_limit}. The improvement comes from using the four-lemma for injective maps.
\end{proof}

\subsubsection{} In the case where taking infinite products in $\matheur{C}$ is $t$-exact, such as the case of $\Vect$, we can push it one step further.

\begin{lem}
\label{lem:homological_stability_of_limit_improvement_w_t-exact_infty_prod}
When taking infinite products in $\matheur{C}$ is $t$-exact, in the situation of Lemma~\ref{lem:homological_stability_of_limit_improvement}, the following map is an equivalence (as in Lemma~\ref{lem:homological_stability_of_limit_improvement})
\[
	\tau_{\leq c+1} V_{\mathbf{d}} \to \tau_{\leq c+1} V_{\mathbf{d}+\unit_k}
\]
and moreover, the map
\[
	{}^t\Ho^{c+2}(V_{\mathbf{d}}) \to {}^t\Ho^{c+2}(V_{\mathbf{d}+\unit_k})
\]
is injective.
\end{lem}
\begin{proof}
By induction, and repeated application of the four-lemma for injective morphisms, we get that the following map is an equivalence
\[
	\tau_{\leq c+1} V_{i, \mathbf{d}} \to \tau_{\leq c+1} V_{i, \mathbf{d}+\unit_k}
\]
and that
\[
	{}^t\Ho^{c+2}(V_{i, \mathbf{d}}) \to {}^t\Ho^{c+2}(V_{i, \mathbf{d} + \unit_k})
\]
is injective. Now, it suffices to show that taking sequential limit preserves injectivity at cohomological degree $c+2$.

Indeed, consider the following diagram
\[
\xymatrix{
	\lim_{i} V_{i, \mathbf{d}} \ar[d] \ar[r] & \prod V_{i, \mathbf{d}} \ar[d] \ar[r] & \prod V_{i-1, \mathbf{d}} \ar[d] \\
	\lim_{i} V_{i, \mathbf{d}+\unit_k} \ar[r] & \prod V_{i, \mathbf{d} + \unit_k} \ar[r] & \prod V_{i-1, \mathbf{d}+\unit_k} 
}
\]
where the rows are exact sequences. Now, taking cohomology and applying four-lemma, we get the desired conclusion.
\end{proof}

\subsection{Stable homology for objects of the form $W \otimes \Lambda[\multgr{m}]$}
\label{subsec:stable_homology_simple}
In this subsection, we give a description for $\lbar{V}$ when $V$ is of the form $W \otimes \Lambda[\multgr{m}]$. In this case $\lbar{V}$ has a particularly simple description.

\begin{prop} \label{prop:stable_piece_easy_case_tensor}
Let $W \in \matheur{C}^{\multgr{m}}$ and $V = W \otimes \Lambda[\multgr{m}] \in \matheur{C}^{\multfil{m}}$. Then, we have a natural equivalence
\[
	\lbar{V} \simeq \oblv_{\gr} (W) \simeq \bigoplus_{\mathbf{d} \in \multgr{m}} W_\mathbf{d}.
\]
\end{prop}
\begin{proof}
Consider the following diagram
\[
\xymatrix{
	\matheur{C}^{\multfil{m}} \ardis[dd]^{\filToGr} \ardis[rr]^{\lbar{(-)}} && \ardis[ll]^{\const} \matheur{C} \ardis[ddll]^{\const} \\ \\
	\matheur{C}^{\multgr{m}} \ardis[uu]^{\Rees} \ardis[uurr]^{\oblv_{\gr}}
} \teq\label{eq:adjunction_grtoFil_oplus_colim}
\]
where parallel arrows are adjoint pairs. Moreover, the right adjoints commute, and hence, so do the left adjoints.

Recall that by~\S\ref{subsubsec:another_take_Rees}, we have 
\[
	\Rees \simeq - \otimes \Lambda[\multgr{m}]. \teq\label{eq:gr2Fil_is_tensor}
\]
But now, the commutativity of the left adjoints in the diagram above gives us the desired conclusion
\[
	\lbar{V} \simeq \lbar{W\otimes[\multgr{m}]} \simeq \lbar{\Rees(W)} \simeq \oblv_{\gr} W \simeq \bigoplus_{\mathbf{d} \in \multgr{m}} W_\mathbf{d}.
\]
\end{proof}

\begin{rmk}
It is easy to see that the same result holds for $\matheur{C}^{\multgrplus{m}}$. Note, however, that $V$ is still of the form $W \otimes \Lambda[\multgr{m}]$ rather than $W\otimes \Lambda[\multgr{m}]_+$. The only difference is that $W$ is now in $\matheur{C}^{\multgrplus{m}}$ instead, which is a property, rather than an extra piece of data. See also~\S\ref{subsubsec:assgr_as_tensoring} for a similar discussion.
\end{rmk}

\begin{rmk}
Note that the object $V = W\otimes \Lambda[\multgr{m}] \in \matheur{C}^{\multfil{m}}$ where $W\in \matheur{C}^{\multgr{m}}$ does not necessarily have homological stability. However, as in Remark~\ref{rmk:abuse_stable_homology}, we still call $\lbar{V} \simeq W$ the stable homology of $V$.
\end{rmk}

\begin{cor} \label{cor:stable_piece_com_alg}
Let $\matheur{A} = \matheur{B} \otimes \Lambda[\multgr{m}] \in \ComAlg(\matheur{C}^{\multfil{m}})$ where $\matheur{B} \in \ComAlg(\matheur{C}^{\multgr{m}})$. Then, we have a natural equivalence of algebras in $\matheur{C}$
\[
	\lbar{\matheur{A}} \simeq \oblv_{\gr} (\matheur{B}) \in \ComAlg(\matheur{C}).
\] 
\end{cor}

\begin{rmk}
Note that if $\matheur{B} \in \ComAlg^\un(\matheur{C}^{\multgr{m}})$, then $\matheur{A} = \matheur{B} \otimes \Lambda[\multgr{m}] \in \ComAlg^{\grun}(\matheur{C}^{\multfil{m}})$ (and similarly for $\ComAlg^{\graug}$ and $\ComAlg^{\grun, \graug}$). $\lbar{\matheur{A}}$ then also carries the relevant unital/augmented structure, which is compatible with that of $\matheur{B}$ via the equivalence in Corollary~\ref{cor:stable_piece_com_alg} above.
\end{rmk}

\section{Cohomological Chevalley complex}
\label{sec:cohomological_chevalley_complex}
Using the general framework set up above, we will now study various aspects of the cohomological Chevalley complex associated to a $\coLie$-coalgebra. We will start with a homological stability result in~\S\ref{subsec:homological_stability_of_coChev}, whose proof, as indicated earlier, reduces to the case of trivial (aka. abelian) coLie-coalgebras. In~\S\ref{subsec:cofiltration_at_infty}, we will provide a means to studying the limiting cohomology when we have homological stability. Finally, we will, in~\S\ref{subsec:decategorification_coChev}, indicate how various decategorification processes could be used to extract information from the stable cohomology.

\subsection{Homological stability} \label{subsec:homological_stability_of_coChev}
We will now prove homological stability for the cohomological Chevalley complexes of a large class of coLie-coalgebras. The actual work needed at this stage is in fact small: the results here are essentially formal consequences of the machinery developed above. We start with the following
\begin{defn}
A $\coLie$-coalgebra $\mathfrak{a} \in \coLie(\matheur{C}^{\multgrplus{m}})$ is said to be unital if it is equipped with a $\coLie$-coalgebra map
\[
	\bigoplus_{k=1}^m \munit_{\matheur{C}}[1]_{\unit_k} \to \mathfrak{a}, \teq\label{eq:unit_map_coLie}
\]
where the object on the LHS is equipped with the trivial coLie-structure.\footnote{Recall that the subscript $\unit_k$ in $\munit_{\matheur{C}}[1]_{\unit_k}$ says that $\munit_{\matheur{C}}[1]$ is placed in graded degree $\unit_k$. See also~\S\ref{subsubsec:notation_graded_degree_1k}.} We will use $\coLie^\un(\matheur{C}^{\multgrplus{m}})$ to denote the category of unital $\coLie$-coalgebras.
\end{defn}

\subsubsection{} \label{subsubsec:unital_coLie_stabilization_map}
By Koszul duality, we see at once that when $\mathfrak{a} \in \coLie^\un(\matheur{C}^{\multgrplus{m}})$, $\coChev^{\un} \mathfrak{a} \in \ComAlg^{\grun}(\matheur{C}^{\multgr{m}})$ which automatically upgrades to an object $\coChev^{\un} \mathfrak{a} \in \ComAlg^{\un}(\matheur{C}^{\multfil{m}})$, see~\eqref{eq:filt_as_mod_cat_over_gr}. In particular, for each $\mathbf{d} \in \multgr{m}$ and $1\leq k\leq m$, we have a naturally defined map
\[
	(\coChev^{\un} \mathfrak{a})_{\mathbf{d}} \to (\coChev^{\un} \mathfrak{a})_{\mathbf{d} + \unit_k}.
\]
We are thus in the position to formulate homological stability.

\subsubsection{}
In what follows, to simplify the notation, we will write $\munit_{\matheur{C}}[1]^{\oplus m}$ to denote the $\coLie$-coalgebra on the LHS of~\eqref{eq:unit_map_coLie}. For example, when $\matheur{C} = \Vect$, we will write $\Lambda[1]^{\oplus m}$. More generally, we will drop the gradings from the notation, unless it is not clear from the context.

\subsubsection{}
\label{subsubsec:quot_un_for_coLie}
Since $\oblv_{\coLie}$ is a left adjoint, it commutes with colimits. In other words, at the level of the underlying objects, colimits of coLie-coalgebras could be computed inside the underlying category.

Let
\[
	\Quot_\un: \coLie^\un(\matheur{C}^{\multgrplus{m}}) \to \coLie(\matheur{C}^{\multgrplus{m}})
\]
denote the functor of removing (or more accurately, quotient-ing out) the unit of the $\coLie$-coalgebra by taking the cofiber of the map~\eqref{eq:unit_map_coLie}.

\subsubsection{}
\label{subsubsec:comparing_quot_un}
By Koszul duality, we have the following equivalence of categories
\[
	\coPrim[1]: \ComAlg^{\grun}(\matheur{C}^{\multgrplus{m}}) \simeq \coLie^{\un}(\matheur{C}^{\multgrplus{m}}): \coChev.
\]
Indeed, this is because
\[
	\coChev (\munit^{\oplus m}_\matheur{C}[1]) \simeq \Sym(\munit^{\oplus m}_\matheur{C})_+ \simeq \Lambda[\multgrplus{m}].
\]
Moreover, the following diagram commutes since $\coChev$ is an equivalence of categories, and hence, in particular, preserves colimits.
\[
\xymatrix{
	\ComAlg^{\grun}(\matheur{C}^{\multgrplus{m}}) \ar[d]_{\Quot_{\grun}} \ardis[r] & \coLie^\un(\matheur{C}^{\multgrplus{m}}) \ar[d]^{\Quot_\un} \ardis[l]^>>>>>{\coChev} \\
	\ComAlg(\matheur{C}^{\multgrplus{m}}) \ardis[r] & \coLie(\matheur{C}^{\multgrplus{m}}) \ardis[l]^>>>>>>{\coChev}
}
\]
Here, $\Quot_\un$ is discussed above, and $\Quot_{\grun}$ in~\S\ref{subsubsec:quot_grun_comalg}.

Note that we have a similar commutative diagram for $\coChev^\un$ (see also~\S\ref{subsubsec:non_unital_quotient}).

\subsubsection{} From now to the end of this subsection, to keep the notation and conditions simple, we will restrict ourselves to the case where $\matheur{C} = \Vect$.
\footnote{The readers are invited to formulate the general results for a general symmetric monoidal stable $\infty$-category $\matheur{C}$ equipped with a $t$-structure. Essentially, the conditions one has to spell out are about the interaction between the monoidal structure on $\matheur{C}$ and the given $t$-structure.}

We are interested in the following special class of $\coLie$-coalgebra.

\begin{defn}
An object $\mathfrak{a} \in \coLie^{\un}(\Vect^{\multgrplus{m}})$ is said to be \emph{strongly unital} if it lives in cohomological degrees $\geq -1$, and moreover, the unit map induces an equivalence at cohomological degree $-1$, namely, we have an equivalence
\[
	\Lambda^{\oplus m} \simeq \Ho^{-1}(\mathfrak{a})
\]
i.e. $\Quot_\un(\mathfrak{a})$ lives in cohomological degrees $\geq 0$.
\end{defn}

As mentioned before, the proof of homological stability will be reduced to the case of trivial coLie-coalgebras. In this case, we record the following lemma which is immediate by inspection.

\begin{lem}
\label{lem:homolgical_stab_coChev_cotriv_coLie}
Let $\mathfrak{a} \in \coLie^\un(\Vect^{\multgrplus{m}})$ be a strongly unital trivial $\coLie$-coalgebra. Fix $c_0 \in \mathbb{N}, 1 \leq k \leq m$ and suppose that there exist $s, s_1, \dots, s_{k-1}, s_{k+1}, \dots, s_m \in \mathbb{N}$ such that for any non-zero $v\in \Ho^*(\Quot_\un(\mathfrak{a}))$ of cohomological degree $c-1\leq c_0$ and graded degree $\mathbf{d} \in \multgrplus{m}$, we have
\[
	\mathbf{d}_k \leq sc + \sum_{i\neq k} s_i \mathbf{d}_i.\teq\label{eq:vanishing_generation_condition}
\]

Then, for all $c\leq c_0$ and $\mathbf{d} \in \multgrplus{m}$ such that
\[
	\mathbf{d}_k \geq sc + \sum_{i\neq k} s_i \mathbf{d}_i, \teq \label{eq:condition_on_d_k}
\]
the following natural map (see~\S\ref{subsubsec:unital_coLie_stabilization_map}) is an equivalence
\[
	\tau_{\leq c} (\coChev^\un \mathfrak{a})_\mathbf{d} \to \tau_{\leq c} (\coChev^\un \mathfrak{a})_{\mathbf{d} + \unit_k}. \teq \label{eq:homological_stability_tau_leq_c}
\]
More generally, for any $c$, the following map is injective
\[
	\Ho^c(\coChev^{\un} \mathfrak{a})_\mathbf{d} \to \Ho^c(\coChev^{\un} \mathfrak{a})_{\mathbf{d}+\unit_k}. \teq \label{eq:general_injectivity_statement}
\]
\end{lem}
\begin{proof}
Since the $\coLie$-coalgebra we are dealing with is abelian,
\[
	\coChev^{\un} \mathfrak{a} \simeq \Sym \mathfrak{a}[-1],
\]
and hence, the statement we need to prove is simply about symmetric algebras. But now, the injectivity statement for~\eqref{eq:general_injectivity_statement} is automatic since multiplying by a generator living in an even cohomological degree is injective in a free commutative algebra.

For the other statement regarding~\eqref{eq:homological_stability_tau_leq_c}, it suffices to show that for each $c\leq c_0$ and $\mathbf{d}_k$ satisfying~\eqref{eq:condition_on_d_k}, the map
\[
	\Ho^c(\Sym \mathfrak{a}[-1])_\mathbf{d} \simeq \Ho^c(\coChev^{\un} \mathfrak{a})_\mathbf{d} \to \Ho^c(\coChev^{\un} \mathfrak{a})_{\mathbf{d} + \unit_k} \simeq \Ho^c(\Sym \mathfrak{a}[-1])_{\mathbf{d} + \unit_k},
\]
induced by multiplying with $\munit_{\unit_k}$ via the unital structure of $\mathfrak{a}$, is surjective. But this is obvious since~\eqref{eq:vanishing_generation_condition} guarantees that any monomial living in $\Ho^c(\Sym \mathfrak{a}[-1])_{\mathbf{d} + \unit_k}$ has at least one $\munit_{\unit_k}$ factor and we are done.
\end{proof}

\begin{expl}
When $\mathfrak{a}$ is such that $\coChev^{\un} \mathfrak{a} = \Lambda[\multgr{m}]$, $\Quot_\un(\mathfrak{a}) = 0$, and hence, for any $c_0$, we can take the $s$ and $s_i$'s to be $0$. Lemma~\ref{lem:homolgical_stab_coChev_cotriv_coLie} then implies that we have homological stability for all homological degrees for all graded-degrees. This is, of course, obvious by looking directly at $\Lambda[\multgr{m}]_+$.
\end{expl}

\begin{expl}
\label{expl:typical_triv_coLie}
Suppose that $\mathfrak{a}$ is a strongly unital trivial $\coLie$-coalgebra supported only at graded degrees $\unit_1, \dots, \unit_m$. Then, for any $c_0$ and $k$, we can take $s = 1, s_i = 0, \forall i \neq k$. It follows that the map below is an equivalence
\[
	\tau_{\leq c}(\coChev^{\un} \mathfrak{a})_\mathbf{d} \to \tau_{\leq c}(\coChev^{\un} \mathfrak{a})_{\mathbf{d}+\unit_k}
\]
whenever $\mathbf{d}_k \geq c$.
\end{expl}

\subsubsection{}
Recall that the result of~\S\ref{subsec:cofil_coBar} says that $\coChev$ (and hence, $\coChev^{\un}$) could be computed as a sequential limit, whose successive fibers form the $\coChev^{\un}$ of the trivial coLie-coalgebra with the same underlying object. Lemma~\ref{lem:homolgical_stab_coChev_cotriv_coLie} above establishes homological stability of $\coChev^{\un}$ for trivial coLie-coalgebras. Applying Lemma~\ref{lem:homological_stability_of_limit_improvement_w_t-exact_infty_prod}, with a simple reindexing, we obtain the following result regarding homological stability of $\coChev^{\un}$ in general. 

%NOTE: Let's say $m=1$. A better proof would be obtained by considering $\coChev^{\un} a \otimes_{\Lambda[x]} \Lambda$. On the one hand, this object is obtained by taking the cone $\coChev^{\un}(a) \to \coChev^{\un}(a)$ given by multiplying with $x$. On the other hand, this object is precisely $\coChev^{\un}(\Quot_\un(a))$. The latter description, by via this sequential limit realization of $coChev$, gives us the cohomological bound. Then, homological stability is done simply by looking at long exact sequence $\coChev^un(a) \to \coChev^un(a) \to \coChev^un Quot_un(a).

\begin{thm}
\label{thm:homological_stability_coChev}
Let $\mathfrak{a} \in \coLie^\un(\Vect^{\multgrplus{m}})$ be a strongly unital $\coLie$-coalgebra. Fix $c_0 \in \mathbb{N}$, $1\leq k\leq m$ and suppose that there exist $s, s_1, \dots, s_{k-1}, s_{k+1}, \dots, s_m \in \mathbb{N}$ such that for any non-zero $v\in \Ho^*(\Quot_\un(\mathfrak{a}))$ of cohomological degree $c-1\leq c_0$ and graded degree $\mathbf{d} \in \multgrplus{m}$, we have
\[
	\mathbf{d}_k \leq sc + \sum_{i\neq k} s_i \mathbf{d}_i. \teq\label{eq:inequality_slope}
\]

Then, for all $c\leq c_0$ and $\mathbf{d} \in \multgrplus{m}$ such that
\[
	\mathbf{d}_k \geq sc + \sum_{i\neq k} s_i \mathbf{d}_i,
\]
the following map is an equivalence
\[
	\tau_{\leq c}(\coChev^{\un} \mathfrak{a})_{\mathbf{d}} \to \tau_{\leq c} (\coChev^{\un} \mathfrak{a})_{\mathbf{d} + \unit_k},
\]
and the following map is injective
\[
	\Ho^{c+1}(\coChev^{\un} \mathfrak{a})_\mathbf{d} \to \Ho^{c+1}(\coChev^{\un} \mathfrak{a})_{\mathbf{d}+\unit_k}.
\]
\end{thm}

\begin{expl}
Let $\mathfrak{a} \in \coLie^\un(\Vect^{\multgrplus{m}})$ be a strongly unital $\coLie$-coalgebra supported at graded degrees $\unit_1, \dots, \unit_m$. Then, as in Example~\ref{expl:typical_triv_coLie}, for any $c_0$ and $k$, we can take $s=1, s_i = 0, \forall i \neq k$. It follows that whenever $\mathbf{d}_k \geq c$ the map below is an equivalence
\[
	\tau_{\leq c} (\coChev^{\un} \mathfrak{a})_\mathbf{d} \to \tau_{\leq c} (\coChev^{\un} \mathfrak{a})_{\mathbf{d}+\unit_k}
\]
and the map below is injective
\[
	\Ho^{c+1}(\coChev^{\un} \mathfrak{a})_\mathbf{d} \to \Ho^{c+1}(\coChev^{\un} \mathfrak{a})_{\mathbf{d}+\unit_k}.
\]
\end{expl}

We have the following easy Corollary of the Theorem above.
\begin{cor} \label{cor:qualitative_homological_stab_fin_dim}
Let $\mathfrak{a} \in \coLie^\un(\Vect^{\multgrplus{m}})$ be a strongly unital $\coLie$-coalgebra such that $\Ho^*(\mathfrak{a})$ is finite dimensional. Then, $\matheur{A} = \coChev^\un \mathfrak{a}$ satisfies homological stability and moreover, the stable homology $\lbar{\matheur{A}}$ is a commutative algebra with finite dimensional $\Ho^i(\lbar{\matheur{A}})$ for each $i$.
\end{cor}

\begin{rmk}
\label{rmk:can_do_better_than_finite_dim}
It is obvious how Theorem~\ref{thm:homological_stability_coChev} allows us to extend beyond the finite dimensional case of  Corollary~\ref{cor:qualitative_homological_stab_fin_dim}. Indeed, there is no finite dimensionality restriction in the theorem. For the applications in this paper, however, the finite dimensional case suffices.
\end{rmk}

\begin{rmk}
\label{rmk:action_of_T}
Suppose that our unital $\coLie$-coalgebra $\mathfrak{a}$ carries the action of some operator $T$ (which will be the Frobenius in our application), then by functoriality, $\coChev^{\un} \mathfrak{a}$ also carries this structure. Suppose the action of $T$ on the unit of $\mathfrak{a}$ is the identity, then we see that $\coChev^{\un} \mathfrak{a}$, as an object of $\ComAlg(\Vect^{\multfil{m}})$, carries the action of $T$ as well. In other words, the action of $T$ is compatible with maps in the filtration. As a result, $T$ acts on the stabilization $\lbar{\coChev^{\un} \mathfrak{a}}$. In particular, when homological stability does occur, $T$ acts naturally on the stable homology.

In what follows, unless otherwise specified, whenever we have an action of an operator $T$ on a $\coLie$-coalgebra $\mathfrak{a}$, we will always assume that the action on the unit is the identity.
\end{rmk}

\subsection{Stable homology} \label{subsec:cofiltration_at_infty}
Let $\mathfrak{a} \in \coLie^\un(\matheur{C}^{\multgrplus{m}})$ and $\matheur{A} = \coChev \mathfrak{a}$. The goal of this subsection is to understand $\lbar{\matheur{A}}$ in terms of $\mathfrak{a}$ itself. When we are in the situation where $\lbar{\matheur{A}}$ satisfies homological stability, Lemma~\ref{lem:stable_homology_as_Vbar} says that the homology of $\lbar{\matheur{A}}$ is the stable homology (see also Remark~\ref{rmk:stable_homology_as_Vbar}). It is thus crucial to understand $\lbar{\matheur{A}}$.

Note that in this section, we do not concern ourselves with questions about homological stability. Thus, we will work with a general symmetric monoidal stable $\infty$-category $\matheur{C}$, since that does not bring in any extra complexity.

\subsubsection{} We start with the simplest case, where $\mathfrak{a}$ is abelian, or more generally, when the unit of $\mathfrak{a}$ splits. 

\begin{prop}
\label{prop:stable_homology_split_coLie}
Let $\mathfrak{a} \in \coLie^\un(\matheur{C}^{\multgrplus{m}})$ such that the unit of $\mathfrak{a}$ splits (as objects in $\coLie^\un(\matheur{C}^{\multgrplus{m}})$), i.e. we have the following equivalence of $\coLie$-coalgebras
\[
	\mathfrak{a} \simeq \Lambda^{\oplus m}[1] \oplus \mathfrak{b}.
\]
Then, we have the following equivalence
\[
	\lbar{\coChev \mathfrak{a}} \simeq \lbar{\coChev^\un \mathfrak{a}} \simeq \oblv_{\gr} (\coChev^\un \mathfrak{b}) \in \ComAlg^\un(\matheur{C})
\]
\end{prop}
\begin{proof}
Since $\coChev \mathfrak{a}$ and $\coChev^\un \mathfrak{a}$ only differ at $\mathbf{0}$, their colimits over $\multfil{m}$ are the same, which gives the first equivalence.

For the second equivalence, note that for any $\coLie$-coalgebras $\mathfrak{u}, \mathfrak{v} \in \coLie(\matheur{C}^{\multgrplus{m}})$, we have
\[
	\coChev(\mathfrak{u} \oplus \mathfrak{v}) \simeq \coChev(\mathfrak{u}) \sqcup \coChev(\mathfrak{v}),
\]
where the coproduct is taken inside $\ComAlg(\matheur{C}^{\multgrplus{m}})$. Indeed, $\coChev$ exhibits an equivalence of categories, and hence, preserves colimits. Thus, we have (see~\S\ref{subsec:unital_vs_nonunital})
\[
	\coChev^\un(\mathfrak{u} \oplus \mathfrak{v}) \simeq \coChev^\un(\mathfrak{u}) \otimes \coChev^\un(\mathfrak{v}).
\]
Applying this to the case at hand, we have
\[
	\coChev^\un(\mathfrak{a}) \simeq \Lambda[\multgr{m}] \otimes \coChev^\un(\mathfrak{b}).
\]
Corollary~\ref{cor:stable_piece_com_alg} then gives us the desired conclusion.
\end{proof}

\subsubsection{} In the case where the unit of $\mathfrak{a} \in \coLie^\un(\matheur{C}^{\multgrplus{m}})$ does not split, we do not have a simple expression for the stable piece. However, we have the following result.

\begin{lem}
\label{lem:assgr_coChev}
Let $\mathfrak{a} \in \coLie^\un(\matheur{C}^{\multgrplus{m}})$. Then, we have the following natural equivalence
\[
	\assgr_{\matheur{C}, m} (\coChev^\un(\mathfrak{a})) \simeq \coChev^\un(\Quot_\un \mathfrak{a}).
\]
\end{lem}
\begin{proof}
The lemma follows directly from~\S\ref{subsubsec:assgr_as_tensoring}. Indeed, we have
\[
	\assgr_{\matheur{C}, m} (\coChev^\un(\mathfrak{a})) \simeq \coChev^\un(\mathfrak{a}) \otimes_{\Lambda[\multgr{m}]} \munit_{\matheur{C}} \simeq \Quot_{\grun} (\coChev^\un (\mathfrak{a})) \simeq \coChev^\un(\Quot_\un \mathfrak{a}), \teq\label{eq:assgr_coChevun_sequence}
\]
where the first and last equivalences are due to~\S\ref{subsubsec:assgr_as_tensoring} and~\S\ref{subsubsec:comparing_quot_un} respectively. The middle equivalence is by definition.
\end{proof}

\begin{rmk}
Note that the Lemma above is a special feature in the unital setting, i.e. $\coChev^\un$ rather than $\coChev$. Indeed, due to~\S\ref{subsubsec:assgr_as_tensoring},
\[
	\assgr_{\matheur{C}, m}(\coChev \mathfrak{a}) \simeq \coChev(\mathfrak{a}) \otimes_{\Lambda[\multgr{m}]} \munit_\matheur{C}.
\]
However,
\[
	\Quot_{\grun}(\coChev\mathfrak{a}) \simeq (\coChev \mathfrak{a}) \sqcup_{\Lambda[\multgrplus{m}]} 0 \simeq  (\coChev^\un \mathfrak{a} \otimes_{\Lambda[\multgr{m}]} \munit_\matheur{C})_+.
\]
And these two do not agree, in fact, already in the case where $m=1$ and $\mathfrak{a} = \Lambda[1]$ sitting in graded-degree 1. Of course, this does not cause us any problem, since we can always use $\coChev^{\un}$ instead and since, as we mentioned earlier, as far as homological stability and stable homology are concerned, there is no difference.
\end{rmk}

\subsection{Decategorification} \label{subsec:decategorification_coChev}
In this section, we will provide various decategorification processes which can be used to extract information out of the functor $\coChev$. More precisely, let $\mathfrak{a} \in \coLie^\un(\Vect^{\multgrplus{m}})$ be a strongly unital $\coLie$-coalgebra. Given a ring $R$ and a decategorification procedure, i.e. an assignment to each equivalence class of object in $\matheur{C}$ an element in $R$,
\[
	\DeCat: \matheur{C} \to R
\]
where $\matheur{C}$ is $\Vect$, $\Vect^{\multgr{m}}$ or $\Vect^{}\multgrplus{m}$ (plus some appropriate finiteness conditions), we want to understand $\DeCat(\coChev^\un \mathfrak{a})$ and $\DeCat(\lbar{\coChev^\un \mathfrak{a}})$ (when $\mathfrak{a}$ satisfies some finiteness condition) in relation to various push-out operations. Moreover, we want to understand how $\DeCat$ of $\coChev^\un(-)$ and of $\lbar{\coChev^\un(-)}$ behave with respect to taking quotients of $\coLie$-coalgebras. For the geometric applications in this paper, this part is used as a link between quotients of algebras and quotients (in the ring $R$) of their decategorifications, i.e. what we called densities in the introduction.

\subsubsection{} For a fixed decategorification $\DeCat$, we say that it is additive, if for an exact triangle
\[
	V' \to V \to V'' \to V[1],
\]
we have
\[
	\DeCat(V) = \DeCat(V') + \DeCat(V'').
\]

\subsubsection{} \label{subsubsec:decategorification_procedures} In this paper, instances of $\DeCat$ we are interested in are as follows
\begin{enumerate}[(i)]
	\item \emph{Graded Euler characteristics.} Let $V\in \Vect^{\multgr{m}}$ such that each graded piece is perfect (i.e. bounded and finite dimensional). Then, the graded Euler characteristic of $V$ is defined to be
	\[
		\chi^\gr(V) = \sum_{\mathbf{d}\in \multgr{m}} \chi(V_{\mathbf{d}}) t^{|\mathbf{d}|} \in \mathbb{Z}\llrrb{t}.
	\]
	Here, the $\chi$ on the RHS is the usual Euler characteristics of a perfect complex.
	
	\item \textit{Graded traces.} More generally, we let $V\in \Vect^{\multgr{m}}$ as above, now, also equipped with an action of an operator $T$. Then, the graded trace of $T$ on $V$ is defined to be
	\[
		\chi_T^\gr(V) = \sum_{\mathbf{d} \in \multgr{m}} \chi_T(V_{\mathbf{d}}) t^{|\mathbf{d}|} \in \mathbb{Z} \llrrb{t}.
	\]
	As above, the $\chi_T$ on the RHS is the usual trace. For the applications that we have in mind, $V$ is equipped with the action of the Frobenius, and we are interested in computing the Frobenius trace.
	
	Clearly, when $T = \id$, we recover graded Euler characteristic above $\chi_{\id} = \chi$. 
	
	\item \textit{\Poincare{} polynomials.} Let $V\in \Vect$ such that $\Ho^i(V)$ is finite dimensional for each $i$. Then, we define the \Poincare{} polynomial of $V$ to be
	\[
		\Poinc(V) = \sum_{i\in \mathbb{Z}} \dim \Ho^i(V) t^i \in \mathbb{N}\llrrb{t^{\pm 1}}.
	\]
	
	\item \textit{Virtual \Poincare{} polynomials.} For the geometric applications we have in mind, $V$ is equipped with the weight filtration $W$ (from the action of the Frobenius). Let $\Gr_W^i(V)$ denote the sub-quotient of $\Ho^*(V)$ with weight $i$. We define the virtual \Poincare{} polynomial of $V$ to be
	\[
		\PoincVir(V) = \sum_{i \in \mathbb{Z}} \chi(\Gr_W^i(V)) t^i \in \mathbb{Z}\llrrb{t^{\pm 1}}.
	\]
	
	\item \textit{Trace.} Let $V\in \Vect$ be equipped with an action of an operator $T$ such that
	\begin{enumerate}[--]
		\item $\Ho^c(V)$ is finite dimensional for each $i$.
		\item For each $\lambda \in \mathbb{C}$ and every integer $c$, let $d_{\lambda, c}$ denote the dimension of the generalized $\lambda$-eigenspace of $T$.\footnote{Here, we implicitly choose an identification $\Lambda \simeq \mathbb{C}$.} Then the sum
		\[
			|V|_T = \sum_{\lambda, c} d_{\lambda, c}|\lambda|
		\]
		converges.
	\end{enumerate}
	Then the trace of $T$ on $V$ is defined in the usual way, i.e.
	\[
		\chi_T(V) = \sum_{\lambda, c} (-1)^c d_{\lambda, c} \lambda.
	\]
	We call such a graded vector space $V$ satisfying the two conditions above $T$-summable.\footnote{We took this notion from~\cite{gaitsgory_weils_2014}*{Defn. 6.3.1}.}
\end{enumerate}

Note that in the above, everything except for \Poincare{} polynomials is additive. In this paper, the first three instances will be used to decategorify $\coChev^\un \mathfrak{a}$ whereas the last two are for $\lbar{\coChev^\un \mathfrak{a}}$.

\begin{rmk}
We will also use the notation, $\DeCat(-, t)$ to emphasize that the formal variable is $t$. For example, $\chi^\gr(V, t), \chi^\gr_T(V, t), \Poinc(V, t)$ etc.
\end{rmk}

\subsubsection{}
As we mentioned in the introduction, in many cases, the process of taking quotients of $\coLie$-coalgebras, which is equivalent to taking pushouts (i.e. relative tensor) of the corresponding commutative algebras via $\coChev^\un$, transforms to taking quotients (i.e. dividing, literally) of power series at the decategorified level. In other words, pushouts of commutative algebras categorify quotients of power series.

\subsubsection{}
Observe that in all of the above, the expressions defining the decategorification processes make sense as soon as each of the coefficients make sense. Suppose we are interested in understanding $\DeCat(\coChev^\un(\mathfrak{a}))$ and $\DeCat(\lbar{\coChev^\un(\mathfrak{a})})$, then we will need some finiteness conditions on $\mathfrak{a}$. For example, it is easy to see, using~\S\ref{subsec:cofil_coBar}, that when $\mathfrak{a} \in \coLie(\Vect^{\multgrplus{m}})$ is strongly unital, and $\Ho^*(\mathfrak{a})$ is finite dimensional then all the sums defining $\chi^{\gr}, \chi^\gr_T,$ and $\Poinc$ are well-defined. In the case of $\PoincVir$, since the coefficients gather multiple cohomological degrees, it suffices to make sure that each weight can only appear in finitely many cohomological degrees. When $\Ho^*(\mathfrak{a})$ is finite dimensional, it is easy to see that this is satisfied when $\Quot_\un \mathfrak{a}$ has weight of at least 1, since tensor is additive with respect to weight.

\subsubsection{}
The common theme is that if $\DeCat$ is additive, $\DeCat(\coChev^\un -)$ is, in some sense, ``multiplicative.'' The main point is that by~\S\ref{subsec:cofil_coBar}, one can reduce the study of $\coChev$ to that of $\Sym$, which is multiplicative. 

We learned this approach from~\cite{gaitsgory_weils_2014}*{\S6.3}, which we will refer to for the proofs. The main difference is that our decategorification processes, with the exception of trace (the last item on the list above), have milder finiteness conditions. This is simply because we have this variable $t$ to ``spread'' things out, so we do not have to worry about convergence, since each coefficient will be a finite sum. Another difference is that we consider the $\coLie$-coalgebra structure instead of just the underlying complex (see also~\cite{gaitsgory_weils_2014}*{Prop. 6.1.10.}). However, for additive categorification processes, this makes no difference, since~\S\ref{subsec:cofil_coBar} allows us to reduce to the trivial case. We thus have the following result, whose proof is the same as that of~\cite{gaitsgory_weils_2014}*{Prop. 6.3.4}.

\begin{prop}
\label{prop:graded_euler_decat}
Let $\mathfrak{a} \in \coLie(\Vect^{\multgrplus{m}})$ be a $\coLie$-coalgebra such that $\Ho^*(\mathfrak{a})$ is finite dimensional. Assume that $\mathfrak{a}$ is equipped with an action of an operator $T$. Then, we have
\[
	\chi_T^\gr(\coChev^{\un} \mathfrak{a}, t) = \chi^{\gr}_T(\Sym(\mathfrak{a}[-1]), t) = \exp\left(-\sum_{n>0} \frac{1}{n} \chi_{T^n}^\gr(\mathfrak{a}, t^n)\right).
\]
\end{prop}
%Note: This is done through he addFil trick for Bar, and not addCoFil trick for coBar.

\begin{rmk}
Note that compared to~\cite{gaitsgory_weils_2014}*{Prop. 6.3.4.}, we have the extra minus sign because our $\mathfrak{a}$ and their $V$ differ by a cohomological shift (see also~\cite{gaitsgory_weils_2014}*{Prop. 6.1.10}).
\end{rmk}

\begin{proof}[Proof (sketch).]
The second equality is standard, see, for example,~\cite{gaitsgory_weils_2014}*{Prop. 6.3.4.}. It remains to show the first equality. However, this is clear from Corollary~\ref{cor:coChev_as_limit} and Remark~\ref{rmk:limit_coBar_converges_strongly}.
\end{proof}

\subsubsection{}
Note that since the exponent is additive in $\mathfrak{a}$, and since $\exp$ is multiplicative, $\chi_T^\gr(\coChev^\un, \mathfrak{a})$ is also multiplicative. Namely, we have the following 

\begin{cor}
\label{cor:quotient_decat_graded_euler}
Consider the following push-out diagram of $\coLie$-coalgebras equipped with the action of some operator $T$
\[
\xymatrix{
	\mathfrak{a} \ar[d] \ar[r] & \mathfrak{b} \ar[d] \\
	0 \ar[r] & \mathfrak{c}
}
\]
where $\mathfrak{a}$ and $\mathfrak{b}$ are with finite dimensional $\Ho^*(\mathfrak{a})$ and $\Ho^*(\mathfrak{b})$ (which implies the same for $\mathfrak{c}$). Then, we have
\[
	\frac{\chi_T^\gr(\coChev^\un \mathfrak{b})}{\chi_T^\gr(\coChev^\un \mathfrak{a})} = \chi_T^\gr(\coChev^\un \mathfrak{b} \otimes_{\coChev^\un \mathfrak{a}} \Lambda) = \chi_T^\gr(\coChev^\un \mathfrak{c}).
\]
\end{cor}

\subsubsection{}
We will now move on to studying $\Poinc$. Since $\Poinc$ is not additive, we can only treat the case where the $\coLie$-coalgebras involved are abelian, in which case, the statement is simply about symmetric powers.

\begin{prop}
\label{prop:quotient_decat_Poincare}
Let
\[
	\mathfrak{a} = \bigoplus_{k=1}^m \Lambda_{\unit_k}[1] \oplus \mathfrak{u} \oplus \mathfrak{v}, \quad \mathfrak{b} = \bigoplus_{k=1}^m \Lambda_{\unit_k}[1] \oplus \mathfrak{u}, \quad \mathfrak{c} = \bigoplus_{k=1}^m \Lambda_{\unit_k}[1] \oplus \mathfrak{v}
\]
be strongly unital $\coLie$-coalgebras, where $\Ho^*(\mathfrak{u})$ and $\Ho^*(\mathfrak{v})$ are finite dimensional, and where the direct sum are taken in the category of $\coLie$-coalgebras.\footnote{I.e. no interaction between different summands.} Then
\[
	\frac{\Poinc(\lbar{\coChev^\un \mathfrak{a}})}{\Poinc(\lbar{\coChev^\un \mathfrak{b}})} = \Poinc(\lbar{\coChev^\un \mathfrak{c}}) = \Poinc(\oblv_{\gr}(\coChev^{\un} \mathfrak{v})).
\]
\end{prop}
\begin{proof}
We have,
\[
	\coChev^\un \mathfrak{a} \simeq \coChev^\un \mathfrak{b} \otimes_{\Lambda[\multgr{m}]} \coChev^\un \mathfrak{c},
\]
and hence
\[
	\lbar{\coChev^\un \mathfrak{a}} \simeq \lbar{\coChev^\un \mathfrak{b}} \otimes_{\Lambda} \lbar{\coChev^\un \mathfrak{c}}.
\]
The first equality follows immediately from the multiplicative nature of $\Poinc$. The second equality is due to Proposition~\ref{prop:stable_homology_split_coLie}.
\end{proof}

\subsubsection{}
We will now turn to $\PoincVir$. The upshot of Proposition~\ref{prop:poinc_vir_decat} below is that when the decategorification procedure is additive, the decategorification of the stable piece can be extracted from the decategorification of the original object, but with the unit removed (hence, the appearance of $\Quot_\un$ in the statement). This is not surprising since $\Quot_\un$ of the $\coLie$-coalgebra is equivalent to taking associated graded (see Lemma~\ref{lem:assgr_coChev}).

\begin{prop}
\label{prop:poinc_vir_decat}
Let $\matheur{A} \in \ComAlg^{\grun}(\Vect^{\multgr{m}})$ where $\matheur{A} = \coChev^{\un} \mathfrak{a}$ for some strongly unital $\coLie$-coalgebra $\mathfrak{a} \in \coLie^\un(\Vect^{\multgrplus{m}})$. Assume further that $\mathfrak{a}$ is equipped with a weight filtration, $\Ho^*(\mathfrak{a})$ is finite dimensional, and moreover, the weights of $\Ho^*(\Quot_\un(\mathfrak{a}))$ are at least $1$. Then, we have
\[
	\PoincVir(\lbar{\coChev^\un \mathfrak{a}}, t) = \PoincVir(\Sym(\oblv_{\gr}(\Quot_\un \mathfrak{a}[-1])), t) = \exp\left(-\sum_{n>0} \frac{1}{n} \PoincVir(\Quot_\un(\mathfrak{a}), t^n)\right).
\]
\end{prop}
\begin{proof}
We have,
\begin{align*}
	\PoincVir(\lbar{\coChev^\un \mathfrak{a}}, t)
	&= \PoincVir(\oblv_{\gr} (\assgr(\coChev^\un \mathfrak{a})), t) \\
	&= \PoincVir(\oblv_{\gr} (\coChev^\un (\Quot_\un \mathfrak{a})), t) \\
	&= \PoincVir(\oblv_{\gr}(\Sym(\Quot_\un \mathfrak{a}[-1])), t) \\
	&= \PoincVir(\Sym(\oblv_{\gr}(\Quot_\un \mathfrak{a}[-1])), t) \\
	&= \exp\left(-\sum_{n>0} \frac{1}{n} \PoincVir(\Quot_\un(\mathfrak{a}), t^n)\right),
\end{align*}
where
\[
	\oblv_{\gr}: \Vect^{\multgr{m}} \to \Vect
\]
is the functor of taking direct sum (see also~\S\ref{subsubsec:oblv_gr_def}). 

The finite dimensionality of $\Ho^*(\mathfrak{a})$ ensures that homological stability happens for $\matheur{A}$, by Corollary~\ref{cor:qualitative_homological_stab_fin_dim}. In this case, information about the stable homology can be recovered by the associated graded, since $\PoincVir$ is additive, and we get the first equality. The second equality comes from Lemma~\ref{lem:assgr_coChev}. The third equality is due to~\S\ref{subsec:cofil_coBar}. The fourth equality is due to the fact that $\oblv_{\gr}$ is symmetric monoidal. And finally, as above, the last equality can be seen directly, since we are now simply dealing with symmetric algebras. Note that throughout, the conditions imposed on $\mathfrak{a}$ ensure that for each power of $t$, the coefficient is a finite sum, and hence we do not need to worry about well-definedness.
\end{proof}

\begin{cor}
\label{cor:quotient_decat_poinc_vir}
Suppose we have the following pushout diagram
\[
\xymatrix{
	\matheur{A} \ar[d] \ar[r] & \matheur{B} \ar[d] \\
	\Lambda[\multgr{m}] \ar[r] & \matheur{B} \otimes_{\matheur{A}} \Lambda[\multgr{m}] = \matheur{Q}
}
\]
in $\ComAlg^{\grun}(\multgr{m})$. Assume further that, $\matheur{A} = \coChev^\un \mathfrak{a}, \matheur{B} = \coChev^\un \mathfrak{b}$, and $\matheur{Q} = \coChev^\un \mathfrak{q}$ where $\mathfrak{a}$, $\mathfrak{b}$, and $\mathfrak{q}$ satisfy the conditions of the Proposition above, and all the maps are compatible with Frobenius actions.

Then
\[
	\frac{\PoincVir(\lbar{\matheur{B}})}{\PoincVir(\lbar{\matheur{A}})} = \PoincVir(\lbar{\matheur{B}} \otimes_{\lbar{\matheur{A}}} \Lambda) = \PoincVir(\Sym(\oblv_{\gr}(0\sqcup_{\mathfrak{a}} \mathfrak{b})[-1])). \teq\label{eq:quotient_decat_poinc_vir}
\]
\end{cor}
\begin{proof}
We have
\begin{align*}
	\PoincVir(\lbar{\matheur{B}} \otimes_{\lbar{\matheur{A}}} \Lambda, t)
	&= \PoincVir(\lbar{\matheur{B} \otimes_{\matheur{A}} \Lambda[\multgr{m}]}, t) \\
	&= \PoincVir(\lbar{\coChev^{\un} \mathfrak{q}}, t) \\
	&= \exp\left(-\sum_{n>0} \frac{1}{n} \PoincVir(\Quot_\un(\mathfrak{q}), t^n)\right),
\end{align*}
where the two last equalities are by Proposition~\ref{prop:poinc_vir_decat} above. But now, we have established the first equality of~\eqref{eq:quotient_decat_poinc_vir} since $\PoincVir$ is additive and since we have the following pushout diagram
\[
\xymatrix{
	\Quot_\un(\mathfrak{a}) \ar[d] \ar[r] & \Quot_\un(\mathfrak{b}) \ar[d] \\
	0 \ar[r] & \Quot_\un(\mathfrak{q})
}
\]

For the second equality of~\eqref{eq:quotient_decat_poinc_vir}, observe that we have the pushout diagram
\[
\xymatrix{
	\mathfrak{a} \ar[d] \ar[r] & \mathfrak{b} \ar[d] \\
	\bigoplus_{k=1}^m \Lambda_{\unit_k}[1] \ar[d] \ar[r] & \mathfrak{q} \ar[d] \\
	0 \ar[r] & \Quot_\un(\mathfrak{q})
}
\]
which implies that
\[
	\Quot_\un(\mathfrak{q}) \simeq 0 \sqcup_{\mathfrak{a}} \mathfrak{b}. \teq\label{eq:quot_un_coLie_q}
\]
\end{proof}

\subsubsection{} Finally, we will turn to trace, $\chi_T$. As mentioned above, more care has to be taken in this case since we have to deal with convergence issues.

\begin{prop}
\label{prop:trace_decat}
Let $V \in \Vect^{\multfil{m}}$ living in cohomological degrees $\geq 0$ and satisfies homological stability. Assume that $V$ is equipped with an action of an operator $T$ such that
\[
	\oblv_{\gr}(\assgr(V)) = \oblv_{\gr} (\Lambda \otimes_{\Lambda[\multgr{m}]} V)
\]
is $T$-summable. Then $\lbar{V}$ is $T$-summable, and moreover
\begin{align*}
	\chi_T(\lbar{V}) 
	&= \chi_T(\oblv_{\gr}(\assgr(V))) = \chi_T(\oblv_{\gr}(\Lambda\otimes_{\Lambda[\multgr{m}]} V)) \\
	&= \chi_T^{\gr}(\assgr(V), 1) = \chi_T^{\gr}(\Lambda \otimes_{\Lambda[\multgr{m}]} V, 1).
\end{align*}

\end{prop}
\begin{proof}
Note that the equalities between the second and fourth terms, third and fifth terms are tautological.\footnote{The 1's appearing in the fourth and fifth terms are used to denote the evaluation at $t=1$.} Moreover, the equality between the second and third terms follow from how associated graded is computed. It thus remains to show that the first and second terms are equal.

For any $\mathbf{d} \in \multgr{m}$, using the fact that $\oblv_{\gr}(\assgr(V))$ is $T$-summable, it is easy to see that $V_{\mathbf{d}}$ is also $T$-summable (see also~\cite{gaitsgory_weils_2014}*{Remark 6.3.3}), and moreover,
\[
	|V_\mathbf{d}|_T \leq |\oblv_{\gr}(\assgr(V))|_T.
\]
Since $V$ satisfies homological stability, we know that for any cohomological degree $c$, $\tau_{\leq c}\lbar{V}$ is also $T$-summable, with
\[
	|\tau_{\leq c} \lbar{V}|_T \leq |\oblv_{\gr}(\assgr(V))|_T
\]
since
\[
	\tau_{\leq c} \lbar{V} \simeq \tau_{\leq c} V_\mathbf{d}
\]
for some $\mathbf{d}$. But now, this implies that $\lbar{V}$ is $T$-summable since
\[
	|\lbar{V}|_T = \lim_{c\to \infty} |\tau_{\leq c} \lbar{V}|_T \leq |\oblv_{\gr}(\assgr(V))|_T,
\]
and we are done.
\end{proof}

\section{Factorization homology}
\label{sec:fact_hom}

The discussion we have so far in this paper has no geometry. In this section, we will discuss factorization algebras and factorization homology, which is the geometric tool needed to bridge the two worlds of geometry and algebra. In~\S\ref{subsec:graded_com_fact_}--\S\ref{subsec:multiplicative_trace}, we give a brief review of the theory and provide the first link to the geometric objects $\gConf{m}{\mathbf{n}}(X)$ that we are interested in. The materials appearing in these subsections are mostly adapted from~\cite{francis_chiral_2011,gaitsgory_weils_2014,gaitsgory_atiyah-bott_2015,gaitsgory_eisenstein_2015}. In~\S\ref{subsec:homological_stability_fact_coh}, we prove the second main result of this paper, which is a statement about homological stability of factorization cohomology. Finally, \S\ref{subsec:variants_of_otimesstar} provides a simple construction which allows one to produce new commutative factorization algebras from a given one. It will be used in~\S\ref{sec:cohomology_of_Z_m_n} to compute the commutative factorization algebras responsible for the cohomology of $\gConf{m}{\mathbf{n}}(X)$.

\subsection{Graded commutative factorization algebras}
\label{subsec:graded_com_fact_}
We will take factorization homology with coefficients in graded commutative factorization algebras. In the introduction of this paper, we briefly reviewed the notion of the $\Ran$ space as well as commutative factorization algebras. We will now visit the graded variant of these notions. Most of what we present in this section comes from~\cite{gaitsgory_eisenstein_2015}*{\S4.1} or are otherwise straight-forward adaptations of~\cite{francis_chiral_2011,gaitsgory_weils_2014,gaitsgory_atiyah-bott_2015}.

\subsubsection{Colored $\Ran$ space/prestack}
Let $X$ be a scheme. $\Ran(X, \multgrplus{m})$ is the prestack defined as follows: for a test scheme $S$,
\[
	\Ran(X, \multgrplus{m})(S) = \{I \subset X(S), \phi: I \to \multgrplus{m}\}.
\]
For $\mathbf{d} \in \multgrplus{m}$, we define $\Ran(X, \multgrplus{m})^{\mathbf{d}}$ to be the component of $\Ran(X, \multgrplus{m})$ consisting of those points $(I, \phi)$ such that
\[
	\sum_{i\in I} \phi(i) = \mathbf{d}.
\]
We have
\[
	\Ran(X, \multgrplus{m}) = \bigsqcup_{\mathbf{d} \in \multgrplus{m}} \Ran(X, \multgrplus{m})^{\mathbf{d}}.
\]

\subsubsection{} Unlike the usual (i.e. ungraded) $\Ran$ space, the colored variant is essentially finite dimensional in the following sense: we have the following canonical map
\[
	\Ran(X, \multgrplus{m})^\mathbf{d} \to \Sym^\mathbf{d} X
\]
that is, by~\cite{gaitsgory_eisenstein_2015}*{Lemma 4.1.3}, an isomorphism after a sheafification in the topology generated by finite surjective maps (see~\S\ref{subsubsec:intro_define_spaces} for the definition of $\Sym^\mathbf{d} X$). Thus, the two spaces have the same sheaf and sheaf cohomology theory. The colored $\Ran$ space serves as a combinatorial model for the latter which is easier to manipulate.

\subsubsection{}
\label{subsubsec:Ran_n_def}
Fix an $m$-tuple of positive integers $\mathbf{n}$, we define $\Ran(X, \multgrplus{m})_\mathbf{n}$ to be the open subfunctor of $\Ran(X, \multgrplus{m})$ fitting into the following pullback diagram (see~\S\ref{subsubsec:intro_define_general_spaces} for the definition of $\gConf{m}{\mathbf{n}}(X)$)
\[
\xymatrix{
	\Ran(X, \multgrplus{m})^\mathbf{d}_\mathbf{n} \ar[r] \ar[d] & \gConf{\mathbf{d}}{\mathbf{n}}(X) \ar[d] \\
	\Ran(X, \multgrplus{m})^\mathbf{d} \ar[r] & \Sym^\mathbf{d} X
}
\]

Similarly to the above,
\[
	\Ran(X, \multgrplus{m})^\mathbf{d}_\mathbf{n} \to \gConf{\mathbf{d}}{\mathbf{n}}(X)
\]
is an isomorphism after a sheafification.

\subsubsection{}
For later use, we will adopt the following notation.
\begin{notation}
\label{notation:gConf_altogether}
We will use (see~\S\ref{subsubsec:intro_define_general_spaces} for the definition of $\gConf{m}{\mathbf{d}}$)
\[
	\gConf{m}{\mathbf{n}}(X) = \bigsqcup_{\mathbf{d} \in \multgr{m}} \gConf{\mathbf{d}}{\mathbf{n}}(X)
\]
to denote the disjoint union of all the schemes $\gConf{\mathbf{d}}{\mathbf{n}}(X)$, where when $\mathbf{d} = \mathbf{0}$ we choose, by convention, that $\gConf{\mathbf{0}}{\mathbf{n}}(X) = \pt$. In particular,
\[
	\gConf{m}{\infty}(X) = \Sym X = \bigsqcup_{\mathbf{d} \in \multgr{m}} \Sym^\mathbf{d} X,
\]
where the $\infty$ in the subscript means $\mathbf{n} = (\infty, \dots, \infty)$.

We will use $\gConf{m}{\mathbf{n}}(X)_+$ to denote
\[
	\gConf{m}{\mathbf{n}}(X)_+ = \gConf{m}{\mathbf{n}}(X) - \gConf{\mathbf{0}}{\mathbf{n}}(X) = \bigsqcup_{\mathbf{d}\in \multgrplus{m}} \gConf{\mathbf{d}}{\mathbf{n}}(X),
\]
and similarly for $(\Sym X)_+$.
\end{notation}

Similarly, for the $\Ran$ side, we have
\begin{notation}
We will use
\[
	\Ran(X, \multgrplus{m})_\mathbf{n} = \bigsqcup_{\mathbf{d} \in \multgrplus{m}} \Ran(X, \multgrplus{m})_\mathbf{n}^\mathbf{d} \subset \Ran(X, \multgrplus{m}).
\]
to collect all the components $\Ran(X, \multgrplus{m})^\mathbf{d}_\mathbf{n}$.
\end{notation}

\subsubsection{} We define $\gConfOpen{\mathbf{d}}(X)$ to be the open subscheme of $\Sym^\mathbf{d} X$ such that no points are allowed to collide (i.e. all multiplicities have to be at most 1). And similarly, we define
\[
	\RanOpen(X, \multgrplus{m}) = \bigsqcup_{\mathbf{d} \in \multgrplus{m}} \RanOpen(X, \multgrplus{m})^{\mathbf{d}}
\]
to be the open subprestack of $\Ran(X, \multgrplus{m})$ where $\RanOpen(X, \multgrplus{m})^\mathbf{d}$ fits into the following pullback diagram
\[
\xymatrix{
	\RanOpen(X, \multgrplus{m})^{\mathbf{d}} \ar[r] \ar[d] & \gConfOpen{\mathbf{d}}(X) \ar[d] \\
	\Ran(X, \multgrplus{m})^\mathbf{d} \ar[r] & \Sym^\mathbf{d} X
}
\]

\subsubsection{Graded commutative factorization algebras}
Similar to the case of the usual $\Ran$ space, $\Ran(X, \multgrplus{m})$ has a natural commutative semi-group structure given by
\[
	\union: \Ran(X, \multgrplus{m})^k \to \Ran(X, \multgrplus{m}).
\]
One can thus define the $\otimesstar$-symmetric monoidal structure on $\Shv(\Ran(X, \multgrplus{m}))$ by convolution in the usual way: for $\matheur{F}_1, \dots, \matheur{F}_k \in \Shv(\Ran(X, \multgrplus{m}))$, we define
\[
	\matheur{F}_1 \otimesstar\dots \otimesstar\matheur{F}_k = \union_!(\matheur{F}_1 \boxtimes \dots\boxtimes \matheur{F}_k).
\]

\subsubsection{} As in the non-graded case, this allows us to define the category
\[
	\ComAlgstar(\Ran(X, \multgrplus{m})) =  \ComAlg(\Shv(\Ran(X, \multgrplus{m}))^{\otimesstar})
\]
of commutative algebra objects, as well as the full subcategory
\[
	\Factstar(X, \multgrplus{m}) \subset \ComAlgstar(\Ran(X, \multgrplus{m})),
\]
consisting of factorizable objects, which will call commutative factorization algebras.

For the reader's convenience, let us quickly review what this means. An object $\matheur{F} \in \ComAlgstar(\Ran(X, \multgrplus{m}))$ is equipped with multiplication maps
\[
	\union_!(\matheur{F} \boxtimes \dots \boxtimes \matheur{F}) \to \matheur{F}
\]
which, by adjunction, give rise to maps of the form
\[
	\matheur{F} \boxtimes \dots\boxtimes \matheur{F} \to \union^! \matheur{F}. \teq\label{eq:adjoint_mult_Ran}
\]
We say that $\matheur{F}$ is factorizable if these maps induce equivalences on $\Ran(X, \multgrplus{m})^k_{\disj}$, i.e. we have
\[
	(\matheur{F} \boxtimes \dots\boxtimes \matheur{F})|_{\Ran(X, \multgrplus{m})^k_{\disj}} \simeq (\union^! \matheur{F})|_{\Ran(X, \multgrplus{m})^k_{\disj}}.
\]

\subsubsection{Graded commutative algebras}
The category $\Shv(X)$ of sheaves on $X$ is equipped with the $\otimesshriek$-symmetric monoidal structure. Thus, we can talk about categories of graded commutative algebras in it as in~\S\ref{subsec:augmented_unital_algebras}, i.e. $\ComAlg(\Shv(X)^{\multgrplus{m}})$ and $\ComAlg(\Shv(X)^{\multgr{m}})$ etc. Sometimes, to emphasize that we are using the $\otimesshriek$-monoidal structure, we will write $\ComAlgshriek$ instead of just $\ComAlg$.

Alternatively, we can view $\Shv(X)^{\multgrplus{m}}$ as follows. Consider
\[
	\multgrplus{m} \times X = \bigsqcup_{\mathbf{d} \in \multgrplus{m}} X,
\]
i.e. disjoint copies of $X$ indexed by $\multgrplus{m}$. Using the $\otimesshriek$-monoidal structure on $X$, and keeping track of the grading, we obtain a symmetric monoidal structure on $\Shv(\multgrplus{m} \times X)$ in an obvious way, which will still be denoted by $\otimesshriek$. Moreover, we have a natural equivalence of symmetric monoidal categories
\[
	\Shv(X)^{\multgrplus{m}} \simeq \Shv(\multgrplus{m} \times X).
\]

Using this identification, we will not distinguish between these two categories in what follows, unless confusion is likely to occur.

\subsubsection{} 
Let
\[
	\delta: \multgrplus{m} \times X \to \Ran(X, \multgrplus{m})
\]
be the ``diagonal'' embedding which sends $x_{\mathbf{d}} \in X(S)$ to the pair $(I, \phi) = (\{x\}, x\mapsto \mathbf{d})$. The functor
\[
	\delta^!: \Shv(\Ran(X, \multgrplus{m})) \to \Shv(X)^{\multgrplus{m}}
\]
commutes with limits, being a right adjoint to $\delta_!$, and moreover, it is symmetric monoidal. Thus, it upgrades to a functor between the corresponding categories of commutative algebras, which will still be denoted by $\delta^!$. Since $\oblv_{\ComAlg}$ is conservative and preserves limits (being a right adjoint to $\Free_{\ComAlg}$), the functor $\delta^!$ between the commutative algebra categories also commutes with limits. Thus, it admits a left adjoint, which will be denoted by $\delta_?$, i.e. we have the following pair of adjoint functors
\[
	\delta_?: \ComAlgshriek(\Shv(X)^{\multgrplus{m}}) \rightleftarrows \ComAlgstar(\Shv(\Ran(X, \multgrplus{m}))): \delta^!.
\]

\begin{prop}[\cite{gaitsgory_weils_2014}*{Thm. 5.6.4}]
\label{prop:comalg_vs_comfact}\footnote{The proof given in~\cite{gaitsgory_weils_2014} is for the non-graded case. However, the same proof carries over to this setting verbatim.}
The pair of functors $\delta_?$ and $\delta^!$ induces an equivalence of categories
\[
	\delta_?: \ComAlgshriek(\Shv(X)^{\multgrplus{m}}) \simeq \Factstar(X, \multgrplus{m}): \delta^!
\]
\end{prop}

\begin{rmk}
The curious reader can take a look at~\cite{gaitsgory_semi-infinite_2017}*{\S2.3} for a concrete implementation of the functor $\delta_?$. We do not need this concrete description in this paper.
\end{rmk}

\subsubsection{} Let $\pi: X \to \pt$ denote the structure map of $X$. For a similar reason as above, we have the following pair of adjoint functors
\[
	\pi_?: \ComAlgshriek(\Shv(X)^{\multgrplus{m}}) \rightleftarrows \ComAlg(\Vect^{\multgrplus{m}}): \pi^!.
\]

\begin{prop}[\cite{gaitsgory_weils_2014}*{Expl. 5.6.9}]
We have the following commutative diagram
\[
\xymatrix{
	\ComAlgshriek(\Shv(X)^{\multgrplus{m}}) \ar[dd]_{\pi_?} \ar@{=}[r]^>>>>>{\delta_?}_<<<<<{\delta^!} & \Factstar(X, \multgrplus{m}) \ar[ddl]^{\hspace{1em}C^*_c(\Ran(X, \multgrplus{m}),-)} \\ \\
	\ComAlg(\Vect^{\multgrplus{m}})
} \teq\label{eq:pi_?_vs_C^*_c(Ran)}
\]
\end{prop}

More concretely, this Proposition states that the functor of taking cohomology of compact support along the $\Ran$ space implements the abstractly defined functor $\pi_?$. We will thus call $\pi_?$ the functor of taking factorization homology.

\subsection{Koszul duality}
\label{subsec:koszul_duality_ran}
We will now review the interaction between $\pi_?, \pi^!$, and Koszul duality. 

\subsubsection{} Note that since $\Shv(X)^{\multgrplus{m}}$ is pro-nilpotent (see Lemma~\ref{lem:multgrplus_multfilplus_pronilpotent}), we have an equivalence of categories by Koszul duality
\[
	\coPrim[1]: \ComAlgshriek(\Shv(X)^{\multgrplus{m}}) \rightleftarrows \coLieshriek(\Shv(X)^{\multgrplus{m}}): \coChev.
\]

Since the functor $\pi^!$ is monoidal, its left adjoint, $\pi_!$ is left-lax monoidal, and hence, the adjoint pair $\pi_! \dashv \pi^!$ upgrades to a pair of adjoint functors
\[
	\pi_!: \coLieshriek(\Shv(X)^{\multgrplus{m}}) \rightleftarrows \coLie(\Vect^{\multgrplus{m}}): \pi^!.
\]

Since the monoidal functor $\pi^!$ commutes with limits (being a right adjoint), the right adjoints (and hence, also the left adjoints) in the diagram below commute
\[
\xymatrix{
	\ComAlgshriek(\Shv(X)^{\multgrplus{m}}) \ardis[r]^<<<<<{\coPrim[1]} \ardis[d]^{\pi_?} & \ardis[l]^<<<<<{\coChev} \coLieshriek(\Shv(X)^{\multgrplus{m}}) \ardis[d]^{\pi_!} \\
	\ComAlg(\Vect^{\multgrplus{m}}) \ardis[u]^{\pi^!} \ardis[r]^>>>>>>>{\coPrim[1]} & \ardis[l]^>>>>>>>>{\coChev} \coLie(\Vect^{\multgrplus{m}}) \ardis[u]^{\pi^!}
}
\]
But since the horizontal functors are equivalences of categories, everything commutes in the diagram above, and we have
\[
	\pi_? \circ \coChev \simeq \coChev \circ \pi_!,
\]
and
\[
	\pi^! \circ \coPrim[1] \simeq \coPrim[1] \circ \pi^!.
\]

\subsection{Factorization cohomology}
\label{subsec:fact_cohomology}
We will now explain the process of taking factorization cohomology. Since the theory of sheaves on prestacks is biased toward the functors $(-)^!$ and $(-)_!$, in general, one can only talk about homology rather than cohomology. However, suppose $X$ admits a compactification
\[
	j: X \to \lbar{X},
\]
we can make sense of it in such a way that is compatible with all the functors we have defined so far. This is the goal of this subsection.

\subsubsection{}
Let
\[
	j: X \to \lbar{X}
\]
be as above.

First, observe that both functors 
\[
	j^* \simeq j^!: \Shv(\lbar{X}) \rightleftarrows \Shv(X): j_*
\]
are symmetric monoidal, where $j^* \simeq j^!$ since $j$ is an open embedding. Hence, we have the following diagram
\[
\xymatrix{
	\ComAlgshriek(\Shv(X)^{\multgrplus{m}}) \ardis[d]^{j_*} \ardis[r]^>>>>>{\coPrim[1]} & \ardis[l]^>>>>>{\coChev} \coLieshriek(\Shv(X)^{\multgrplus{m}}) \ardis[d]^{j_*} \\
	\ComAlgshriek(\Shv(\lbar{X})^{\multgrplus{m}})\ardis[u]^{j^*} \ardis[r]^>>>>>{\coPrim[1]} & \coLieshriek(\Shv(\lbar{X})^{\multgrplus{m}}) \ardis[l]^>>>>>{\coChev} \ardis[u]^{j^*}
}
\]
where parallel arrows are adjoint pairs. Moreover, a priori, left/right adjoints commute with left/right adjoints, but since all the horizontal maps are equivalences everything commutes.

\subsubsection{} Arguing similarly, we have the following diagram, where everything commutes
\[
\xymatrix{
	\ComAlgshriek(\Shv(X)^{\multgrplus{m}}) \ardis[d]^{j_*} \ardis[r]^>>>>>{\delta_?} & \ardis[l]^>>>>>{\delta^!} \Factstar(X, \multgrplus{m}) \ardis[d]^{j_*} \\
	\ComAlgshriek(\Shv(\lbar{X})^{\multgrplus{m}})\ardis[u]^{j^*} \ardis[r]^>>>>>{\delta_?} & \Factstar(\lbar{X}, \multgrplus{m}) \ardis[l]^>>>>>{\delta^!} \ardis[u]^{j^*}
}
\]

\begin{defn}
We define the functor of taking factorization cohomology
\[
	\pi_{?*}: \ComAlgshriek(\Shv(X)^{\multgrplus{m}}) \to \ComAlgshriek(\Vect^{\multgrplus{m}})
\]
to be
\[
	\pi_{X, ?*} = \pi_{\lbar{X}, ?} \circ j_*
\]
where $\pi_X$ and $\pi_{\lbar{X}}$ are structure maps of $X$ and $\lbar{X}$ respectively. Namely,
\[
\xymatrix{
	\ComAlgshriek(\Shv(X)^{\multgrplus{m}}) \ar[r]^{j_*} \ar@/_1pc/[rr]_{\pi_{?*}} & \ComAlgshriek(\Shv(\lbar{X})^{\multgrplus{m}}) \ar[r]^{\pi_?} & \ComAlg(\Vect^{\multgrplus{m}}).
}
\]
\end{defn}

\begin{rmk}
Observe that
\begin{align*}
	\pi_{?*} 
	&\simeq \pi_? \circ j_* \simeq C^*_c(\Ran(\lbar{X}, \multgrplus{m}), \delta_?(j_*(-))) \\ 
	&\simeq C^*_c(\Ran(\lbar{X}, \multgrplus{m}), j_*(\delta_?(-)))
	\simeq C^*(\Ran(X, \multgrplus{m}), \delta_?(-)),
\end{align*}
where the last term can be made sense of using the fact that as far as sheaves and cohomology are involved, $\Ran(X, \multgrplus{m})^\mathbf{d}$ is the same as $\Sym^\mathbf{d} X$. This implies that $\pi_{?*}$ does not depend on the choice of the compactification $\lbar{X}$ of $X$.
\end{rmk}

\subsubsection{}
Note that the discussion above also implies that our functor $\pi_{?*}$ behaves nicely with respect to $\coChev$. Indeed, let $\mathfrak{a} \in \coLieshriek(\Shv(X)^{\multgrplus{m}})$. Then
\[
	\pi_{?*} (\coChev \mathfrak{a}) \simeq \pi_? (j_* (\coChev \mathfrak{a})) \simeq \pi_?(\coChev(j_* \mathfrak{a})) \simeq \coChev(\pi_!(j_* \mathfrak{a})) \simeq \coChev(\pi_* \mathfrak{a}), \teq\label{eq:coChev_vs_pi_?*}
\]
where we have used $\pi$ to denote both the structure map of $X$ and of $\lbar{X}$. Note also that for $\lbar{X}$, $\pi_* \simeq \pi_!$ since $\lbar{X}$ is proper, and hence $\pi_! \circ j_* \simeq \pi_* \circ j_* \simeq \pi_*$, which is used in the last equivalence above.

The $\coLie$-coalgebra structure of $\pi_* \mathfrak{a}$ (in~\eqref{eq:coChev_vs_pi_?*}) is induced by the fact that
\[
	\pi_*: \Shv(X)^{\multgrplus{m}} \to \Vect^{\multgrplus{m}}
\]
is left-lax symmetric monoidal. Indeed, this is because
\[
	\pi_* \simeq \pi_!\circ j_*,
\]
where $j_*$ is monoidal, and $\pi_!$ is left-lax monoidal. 

\subsubsection{} When our $\coLie$-coalgebra has the form $\pi^! \mathfrak{a}$ for some $\mathfrak{a} \in \Vect^{\multgrplus{m}}$, we have
\[
	\pi_* \pi^! \mathfrak{a} \simeq C^*(X, \omega_X) \otimes \mathfrak{a} \simeq \pi_* \omega_X \otimes \mathfrak{a}. \teq\label{eq:push_forward_constant_coLie}
\]
Since $C^*(X, -)$ is left-lax monoidal, $C^*(X, \omega_X) \in \Vect$ has a natural co-commutative coalgebra structure since $\omega_X$ does, namely
\[
	\omega_X \otimesshriek \omega_X \simeq \omega_X.
\]
The $\coLie$-coalgebra structure on $\pi_* \omega_X\otimes \mathfrak{a}$ is then induced by that of $\mathfrak{a}$ and the co-commutative co-algebra structure on $C^*(X, \omega_X)$. See~\cite{gaitsgory_study_2017}*{Vol. II, Chapter 6, \S1.2} and~\cite{ho_free_2017}*{Example. 4.2.8} for a discussion of the dual situation, where one tensors a $\Lie$-algebra with a commutative algebra to get a new $\Lie$-algebra. See also~\cite{ho_free_2017}*{\S5.3} for a computational example where the resulting $\Lie$-algebra is abelian.

\begin{rmk}
\label{rmk:fact_coh_triv_coalg}
When the co-algebra $C^*(X, \omega_X)$ is trivial (for eg. when $X = \mathbb{A}^d$), $\pi_*\omega_X \otimes \mathfrak{a}$ is abelian. In many cases, this can be exploited to compute the factorization cohomology (for eg. in the case where $X = \mathbb{A}^d$).
\end{rmk}

\subsubsection{}
\label{subsubsec:unital_variant_pi_?*}
Let $\matheur{A} \in \ComAlg^{\un, \aug}(\Shv(X)^{\multgrplus{m}})$ (see~\S\ref{subsec:unital_vs_nonunital}) with augmentation ideal $\matheur{A}_+$. We will use
\[
	\pi_{?*}^{\un} \matheur{A} = \Lambda \oplus \pi_{?*} (\matheur{A}_+)
\]
to denote the unital factorization cohomology. By definition, when the Koszul dual of $\matheur{A}_+$ is $\mathfrak{a}$, we have
\[
	\pi_{?*}^{\un} \matheur{A} \simeq \coChev^{\un}(\pi_* \mathfrak{a}).
\]

When $\matheur{A} \in \ComAlg(\Shv(X)^{\multgrplus{m}})$, we will abuse notation and write
\[
	\pi_{?*}^{\un} \matheur{A} = \Lambda \oplus \pi_{?*} \matheur{A}.
\]
If $\mathfrak{a}$ is the Koszul dual of $\matheur{A}$, then we still have
\[
	\pi_{?*}^{\un} \matheur{A} \simeq \coChev^{\un}(\pi_* \mathfrak{a}).
\]

\subsection{A multiplicative trace formula}
\label{subsec:multiplicative_trace}
We will now come to an analog of the Grothendieck-Lefschetz trace formula, but for the functor $\pi_{?*}$ instead of $\pi_*$. If $\pi_*$ is an additive functor, in some sense, we will see that $\pi_{?*}$ is multiplicative. The result presented here is a straightforward adaptation of~\cite{gaitsgory_weils_2014}*{\S6.5}. Our case is simpler since we do not have to worry about convergence issues.

We start with the following finiteness condition.
\begin{defn}
For a category $\matheur{C}$, we will use $\matheur{C}^{c, \multgr{m}}$ to denote the fullsubcategory of $\matheur{C}^{\multgr{m}}$ consisting objects that are compact in each degree. More concretely, in the cases of interest $\Shv(X)$ and $\Vect$, we have
\begin{enumerate}[(i)]
	\item $\Shv(X)^{c, \multgr{m}}$ consists of graded-sheaves that are constructible in each graded-degree;
	\item $\Vect^{c, \multgr{m}}$ consists of graded-chain complexes such that are bounded (on both sides) at each graded-degree.
\end{enumerate}
\end{defn}

\subsubsection{} It is easy to see that for $\Shv(X)^{\multgrplus{m}}$ and $\Vect^{\multgrplus{m}}$, $\coChev$ and $\coPrim[1]$ preserve compactness. Indeed, this is because when we restrict to each graded degree, the limits/colimits computing them are finite (essentially the same argument used in the case of pro-nilpotent categories). But now, for schemes, since the functor $\pi_*$ preserves compactness, so does the functor $\pi_{?*}$ since it's computed as
\[	
	\pi_{?*} \simeq \coChev \circ \pi_* \circ \coPrim[1].
\]

\begin{prop}
\label{prop:multiplicative_trace}
When the scheme $X$ and $\matheur{A} \in \ComAlg^{\un, \aug}(\Shv(X)^{c, \multgrplus{m}})$ are obtained from pulling back objects defined over a finite field $\Fq$. Then $\pi_{?*} \matheur{A}$ naturally carries a Frobenius action and we have the following identity
\[
	\chi^{\gr}_{\Frob^{-1}_q}(\pi_{?*}^{\un} \matheur{A}, t) = \prod_{x\in |X|} \chi^{\gr}_{\Frob_{x}^{-1}}(i_x^! \matheur{A}, t^{\deg(x)}),
\]
where $\Frob^{-1}$ denotes the inverse of the geometric Frobenius, and $|X|$ denotes the set of all closed points of $X$.
\end{prop}

Before we start the proof of Proposition~\ref{prop:multiplicative_trace}, let us recall the following (Verdier-dual) form of the Grothendieck-Lefschetz trace formula, see, for example~\cite{gaitsgory_weils_2014}*{Theorem 1.3.2}.

\begin{thm}[Grothendieck]
\label{thm:Grothendieck_Lefschetz_dual}
Let $X$ be a scheme and $\matheur{F} \in \Shv(X)$ a constructible sheaf. Assume that they both are pullbacks of objects defined over a finite field $\Fq$. Then, we have the following equality
\[
	\chi_{\Frob^{-1}_q}(C^*(X, \matheur{F})) = \sum_{x\in X(\Fq)} \chi_{\Frob^{-1}_q}(i_x^! \matheur{F}),
\]
where $\chi_{\Frob_q^{-1}}(-)$ denotes the trace of $\Frob_q^{-1}$.\footnote{This is sometimes denoted as $\tr(\Frob_q^{-1}, -)$.}
\end{thm}

\begin{proof}[Proof of Proposition~\ref{prop:multiplicative_trace}]

Let $\mathfrak{a}$ be the Koszul dual of $\matheur{A}_+$. Then for each $x\in |X|$, $i_x^! \mathfrak{a}$ is the Koszul dual of $i_x^! \matheur{A}_+$, since $i_x^!$ is symmetric monoidal and commute with limits. In particular,
\[
	i_x^! \matheur{A} \simeq \coChev^{\un}(i^!_x \mathfrak{a}). \teq\label{eq:i_x^!_vs_coChev}
\]

We have
\begin{align*}
	&\quad\,\, \chi^{\gr}_{\Frob_q^{-1}}(\pi_{?*}^{\un} \matheur{A}, t) \\
	&= \chi^{\gr}_{\Frob_q^{-1}}(\coChev^{\un} \pi_* \mathfrak{a}, t) \tag{\S\ref{subsubsec:unital_variant_pi_?*}} \\
	&= \exp\left(-\sum_{n>0}\frac{1}{n} \chi^{\gr}_{\Frob_q^{-n}}(\pi_* \mathfrak{a}, t^n)\right) \tag{Proposition~\ref{prop:graded_euler_decat}} \\
	&= \exp\left(-\sum_{n>0} \sum_{x\in X(\mathbb{F}_{q^n})} \frac{1}{n} \chi^{\gr}_{\Frob_q^{-n}}(i_x^! \mathfrak{a}, t^n) \right) \tag{Theorem~\ref{thm:Grothendieck_Lefschetz_dual}}\\
	&= \exp\left(-\sum_{n>0} \sum_{\substack{x\in |X| \\ \deg x | n}} \frac{\deg x}{n} \chi^{\gr}_{\Frob_x^{-n/\deg x}}(i_x^! \mathfrak{a}, t^n)\right) \\
	&= \exp\left(-\sum_{x\in |X|} \sum_{m > 0} \frac{1}{m} \chi^{\gr}_{\Frob_x^{-m}}(i_x^! \mathfrak{a}, t^{m\deg x})\right) \\
	&= \prod_{x\in |X|} \exp\left(-\sum_{m>0} \frac{1}{m} \chi^{\gr}_{\Frob_x^{-m}}(i_x^! \mathfrak{a}, t^{m\deg x})\right) \\
	&= \prod_{x\in |X|} \chi^{\gr}_{\Frob_x^{-1}}(i_x^! \matheur{A}, t^{\deg x}). \tag{Proposition~\ref{prop:graded_euler_decat} and~\eqref{eq:i_x^!_vs_coChev}}
\end{align*}

\end{proof}

\subsection{Homological stability}
\label{subsec:homological_stability_fact_coh}
We will now come to the second main result of the paper: homological stability of factorization cohomology with commutative factorization algebra coefficients. Thanks to the preparation that we have done so far, the result is now quite straightforward to prove.

\subsubsection{}
In this subsection, we will work with irreducible schemes of dimension $d$. Let $X$ be such a scheme. Then
\[
	\dim \Ho^{-2d}(X, \omega_X) = \dim \Ho^{2d}_c(X, \Lambda) = 1,
\]
where the first equality is due to Verdier duality. Moreover, we have a natural map of algebras
\[
	C^*_c(X, \Lambda) \to \Ho^{2d}_c(X, \Lambda)[-2d] \simeq \Lambda[-2d](-d)
\]
inducing an equivalence at cohomological degree $2d$, where the RHS is equipped with the trivial algebra structure. Dually, we have the following map of co-algebras
\[
	\Lambda[2d](d) \to C^*(X, \omega_X),
\]
inducing an equivalence at cohomological degree $-2d$, where the co-algebra structure on the LHS is trivial.

\subsubsection{}
In parallel to the situation in~\S\ref{subsec:homological_stability_of_coChev}, we have the following
\begin{defn}
\label{defn:d_shifted_unital}
An $\mathfrak{a}\in \coLie(\Vect^{\multgrplus{m}})$ is said to be $d$-shifted unital if it is equipped with a map of $\coLie$-coalgebras\footnote{Note that the extra Tate twist is to make sure that we will eventually fall into the setting of Remark~\ref{rmk:action_of_T}.}
\[
	\bigoplus_{k=1}^{m} \Lambda_{\unit_k}[-2d+1](-d) \to \mathfrak{a},
\]
where the LHS is (necessarily) equipped with the trivial $\coLie$-coalgebra structure.

Such an $\mathfrak{a}$ is said to be $d$-shifted strongly unital if it lives in cohomological degrees $\geq 2d-1$ and the map above induces an equivalence at cohomological degree $2d-1$.
\end{defn}

\begin{lem}
\label{lem:shifted_strongly_unital_to_strongly_unital}
Let $X$ be an irreducible scheme of dimension $d$, and $\mathfrak{a} \in \coLie(\Vect^{\multgrplus{m}})$ is a $d$-shifted strongly unital $\coLie$-coalgebra. Then
\[
	C^*(X, \omega_X) \otimes \mathfrak{a} \in \coLie^{\un}(\Vect^{\multgrplus{m}})
\]
is strongly unital.
\end{lem}
\begin{proof}
Indeed, the unit map is given by
\[
	\bigoplus_{k=1}^m \Lambda_{\unit_k}[1] \simeq \Lambda[2d](d) \otimes \bigoplus_{k=1}^m \Lambda_{\unit_k}[-2d+1](-d) \to C^*(X, \omega_X) \otimes \mathfrak{a}.
\]
\end{proof}

\subsubsection{} We are now in the purview of Theorem~\ref{thm:homological_stability_coChev}, which could be used to investigate homological stability of the factorization cohomology of $X$ with coefficient in $\pi^! (\coChev \mathfrak{a}) \simeq \coChev(\pi^! \mathfrak{a})$ since, by~\eqref{eq:coChev_vs_pi_?*} and~\eqref{eq:push_forward_constant_coLie},
\[
	\pi_{?*} (\pi^! \coChev \mathfrak{a}) \simeq \coChev(C^*(X, \omega_X) \otimes \mathfrak{a}).
\]
In particular, the following result follows immediately from Corollary~\ref{cor:qualitative_homological_stab_fin_dim}.

\begin{thm}
\label{thm:homological_stab_fact_coh}
Let $X$ be an irreducible scheme of dimension $d$. Let $\matheur{A} \in \ComAlg(\Vect^{\multgrplus{m}})$ such that $\mathfrak{a} = \coPrim[1](\matheur{A})$ is a finite dimensional $d$-shifted strongly unital $\coLie$-coalgebra. Then,
\[
	\pi_{?*}^{\un} \pi^! \matheur{A} \simeq \coChev^{\un}(C^*(X, \omega_X) \otimes \mathfrak{a})
\]
satisfies homological stability.
\end{thm}

A couple of remarks are in order.

\begin{rmk}
\label{rmk:bounded_generation}
Morally speaking, Theorem~\ref{thm:homological_stab_fact_coh} states that factorization cohomology of an irreducible scheme $X$ of dimension $d$ with coefficients in a graded commutative algebra $\matheur{A}$ satisfies homological stability if the cotangent fiber of $\matheur{A}$ is finite dimensional.\footnote{Cotangent fiber in the sense~\cite{gaitsgory_weils_2014}*{\S6.1}. This is linked to our $\coPrim[1]$ via~\cite{gaitsgory_weils_2014}*{Thm. 6.1.10}.} The extra unital assumption is there just so that we can construct stabilization maps and hence to make sense of questions about homological stability.

In this sense, this result is closely related to the one of~\cite{kupers_$e_n$-cell_2018}. Indeed, the finiteness condition on $\mathfrak{a}$ (or more generally, the bound on the cohomological amplitudes stated in Theorem~\ref{thm:homological_stability_coChev}) could be viewed as the analog of the bounded generation property of loc. cit. As we mentioned in Remark~\ref{rmk:can_do_better_than_finite_dim}, this Theorem can be extended beyond the case of finite dimensionality. We choose not to do that since the numerics involved would look somewhat complicated,\footnote{Even though it would be just some shift of the conditions listed in Theorem~\ref{thm:homological_stability_coChev}.} and moreover, in practice, we can just apply Theorem~\ref{thm:homological_stability_coChev} directly to get stability if we care about the stable range.
\end{rmk}

\begin{rmk}
Note that by Theorem~\ref{thm:homological_stability_coChev}, the homological stability range of $\pi_{?*} \matheur{A}$ depends solely on the cohomological amplitudes of $\Ho^*(X, \omega_X)$ and $\Ho^*(\mathfrak{a})$. In~\S\ref{sec:cohomology_of_Z_m_n}, we will present many examples where the ranges for homological stability can be easily extracted from Theorem~\ref{thm:homological_stability_coChev}.
\end{rmk}

\subsection{Variants}
\label{subsec:variants_of_otimesstar}
For the application of this paper, we also need various variants of the $\otimesstar$-monoidal structure, along with the corresponding notions of commutative algebras and commutative factorization algebras. More specifically, we will need an analog of the $\otimesstar$-monoidal structure on $\Shv(\Ran(X, \multgrplus{m})_\mathbf{n})$. Just like in the case of $\Ran(X, \multgrplus{m})$, this is done by equipping $\Ran(X, \multgrplus{m})_\mathbf{n}$ with the structure of a commutative semi-group, and then, the monoidal structure is defined as convolution. However, this time, it is a commutative semi-group object in the category $\Corr(\PreStk)$ of correspondences in prestacks.

This theory is developed in~\cite{raskin_chiral_2015}. We used it in~\cite{ho_free_2017} to study the cohomology of configuration spaces. The construction we give here is a variant of the one given there.

\subsubsection{} Since $\Ran(X, \multgrplus{m})$ is a semi-group object in $\PreStk$, it is automatically a semi-group object in $\Corr(\PreStk)$, i.e. the multiplication maps are given by
\[
\xymatrix{
	& \Ran(X, \multgrplus{m})^k \ar@{=}[dl] \ar[dr]^{\union} \\
	\Ran(X, \multgrplus{m})^k && \Ran(X, \multgrplus{m})
}
\]
From this, we can equip $\Ran(X, \multgrplus{m})_\mathbf{n}$ with the structure of a semi-group object by pulling back the multiplication map (i.e. $\union$) along the open morphism
\[
	\iota_\mathbf{n}: \Ran(X, \multgrplus{m})_\mathbf{n} \hookrightarrow \Ran(X, \multgrplus{m}).
\]
Namely, the commutative semi-group structure on $\Ran(X, \multgrplus{m})_\mathbf{n}$ is given by
\[
\xymatrix{
	& \Ran(X, \multgrplus{m})^k \times_{\Ran(X, \multgrplus{m})} \Ran(X, \multgrplus{m})_\mathbf{n} \ar[dl]_j \ar[dr]^{\union} \\
	(\Ran(X, \multgrplus{m})_\mathbf{n})^k && \Ran(X, \multgrplus{m})_\mathbf{n}
} \teq\label{eq:semi_group_Ran_n}
\]
Note that $\union$ is pseudo-proper, since it is a pullback of a pseudo-proper morphism. Moreover, we see easily that $j$ is an open embedding.

\subsubsection{} We are now ready to define the $\otimesstar$-monoidal structure on $\Shv(\Ran(X, \multgrplus{m})_\mathbf{n})$: for $\matheur{F}_1, \dots, \matheur{F}_k \in \Shv(\Ran(X, \multgrplus{m})_\mathbf{n})$, we define
\[
	\matheur{F}_1 \otimesstar \dots \otimesstar \matheur{F}_k = \union_! j^! (\matheur{F}_1 \boxtimes \dots \boxtimes \matheur{F}_k) \simeq \union_! j^* (\matheur{F}_1 \boxtimes \dots \boxtimes \matheur{F}_k),
\]
where the equivalence is due to the fact that $j$ is an open embedding.

As in the case of $\Ran(X, \multgrplus{m})$, we will use
\[
	\ComAlgstar(\Ran(X, \multgrplus{m})_\mathbf{n}) = \ComAlg(\Shv(\Ran(X, \multgrplus{m})_\mathbf{n})^{\otimesstar})
\]
to denote the category of commutative algebra objects in $\Shv(\Ran(X, \multgrplus{m})_\mathbf{n})$ with respect to the $\otimesstar$-monoidal structure. 

We define the full subcategory
\[
	\Factstar(X, \multgrplus{m})_\mathbf{n} \subset \ComAlgstar(\Ran(X, \multgrplus{m})_\mathbf{n})
\]
consisting of factorizable objects, namely, those commutative algebras $\matheur{F}$ such that the
maps (obtained from the multiplication maps by adjunction)
\[
	j^! (\matheur{F} \boxtimes \dots \boxtimes \matheur{F}) \to \union^! \matheur{F} \teq\label{eq:adjoint_mult_Ran_n}
\]
induce equivalences
\[
	(j^! (\matheur{F} \boxtimes \dots \boxtimes \matheur{F}))|_{(\Ran(X, \multgrplus{m})_\mathbf{n})^k_{\disj}} \to (\union^! \matheur{F})|_{(\Ran(X, \multgrplus{m})_\mathbf{n})^k_{\disj}}.
\]

\subsubsection{} We will now study the interaction between the $\otimesstar$-monoidal structures on $\Shv(\Ran(X, \multgrplus{m})_\mathbf{n})$ and on $\Shv(\Ran(X, \multgrplus{m}))$. We start with the following
\begin{lem}
The functor
\[
	\iota^!_\mathbf{n} \simeq \iota^*_\mathbf{n}: \Shv(\Ran(X, \multgrplus{m})) \to \Shv(\Ran(X, \multgrplus{m})_\mathbf{n})
\]
is symmetric monoidal with respect to the $\otimesstar$-monoidal structures on both sides. As a result, its right adjoint $\iota_{\mathbf{n},*}$ is right-lax monoidal.
\end{lem}
\begin{proof}
This follows immediately from the base-change theorem for pseudo-proper morphisms~\S\ref{subsubsec:pseudo_proper_basechange} and the fact that the semi-group structure on $\Ran(X, \multgrplus{m})_\mathbf{n}$ is defined via pulling-back along the pseudo-proper map $\union$ for $\Ran(X, \multgrplus{m})$.
\end{proof}

\begin{cor}
The adjoint pair $\iota^*_\mathbf{n} \dashv \iota_{\mathbf{n}, *}$ upgrades to a pair of adjoint functors between the corresponding commutative algebra objects:
\[
	\iota^*_\mathbf{n} \simeq \iota^!_\mathbf{n}: \ComAlgstar(\Ran(X, \multgrplus{m})) \rightleftarrows \ComAlgstar(\Ran(X, \multgrplus{m})_\mathbf{n}): \iota_{\mathbf{n}, *}.
\]
\end{cor}

\subsubsection{} The adjoint pair above also behaves nicely with respect to factorizability.
\begin{lem}
The functors $\iota_\mathbf{n}^*$ and $\iota_{\mathbf{n}, *}$ preserve factorizability. Hence, we have a pair of adjoint functors
\[
	\iota^!_\mathbf{n} \simeq \iota^*_\mathbf{n}: \Factstar(X, \multgrplus{m}) \rightleftarrows \Factstar(X, \multgrplus{m})_\mathbf{n}: \iota_{\mathbf{n},*}.
\]
\end{lem}
\begin{proof}
We start with the case for $\iota_\mathbf{n}^!$. The statement follows from the following pullback diagram
\[
\xymatrix{
	(\Ran(X, \multgrplus{m})_\mathbf{n})^k_\disj \ar[d] \ar[r] & \Ran(X, \multgrplus{m})^k \times_{\Ran(X, \multgrplus{m})} \Ran(X, \multgrplus{m})_\mathbf{n} \ar[d]^j \\
	\Ran(X, \multgrplus{m})^k_{\disj}\ar[r] & \Ran(X, \multgrplus{m})^k
}
\]
Indeed, for $\matheur{F} \in \Factstar(X, \multgrplus{m})$, the maps
\[
	j^! (\matheur{F} \boxtimes \dots \boxtimes \matheur{F}) \to \union^! \matheur{F},
\]
of~\eqref{eq:adjoint_mult_Ran_n} are obtained from the corresponding maps for $\Ran(X, \multgrplus{m})$,~\eqref{eq:adjoint_mult_Ran}, by pulling along $j$. Proving factorizability amounts to showing that pulling this back further to $(\Ran(X, \multgrplus{m})_\mathbf{n})^k_\disj$ along the top horizontal arrow gives an equivalence. But by functoriality of $(-)^!$, we can pullback along the other circuit of the diagram above. As an intermediate step, however, when we pull back to $\Ran(X, \multgrplus{m})^k_{\disj}$, we get an equivalence, provided by the given factorizability condition on $\Ran(X, \multgrplus{m})$, and we are done.

The case for $\iota_{\mathbf{n}, *}$ is argued similarly, using base change for $(-)_*$ and $(-)^!$ in place of functoriality of $(-)^!$.
\end{proof}

\subsubsection{} 
Now, we want to understand what $\iota_{\mathbf{n}, *} \iota^!_\mathbf{n}$ does to commutative factorization algebras. More specifically, we want to understand the following composition of functors (the dotted arrow below)
\[
\xymatrix{
	\ComAlgshriek(\Shv(X)^{\multgrplus{m}}) \ar[r]^>>>>>{\delta_?}_>>>>>\simeq \ar@{.>}[dd] & \Factstar(X, \multgrplus{m}) \ar[d]^{\iota_\mathbf{n}^!} \\
	& \Factstar(X, \multgrplus{m})_\mathbf{n} \ar[d]^{\iota_{\mathbf{n}, *}} \\
	\ComAlgshriek(\Shv(X)^{\multgrplus{m}}) & \Factstar(X, \multgrplus{m}) \ar[l]_>>>>>>{\delta^!}^>>>>>>\simeq
}
\]

Consider the following pullback diagram
\[
\xymatrix{
	(\multgrplus{m} \times X)_\mathbf{n} \ar[r]^<<<<<{\delta} \ar[d]_{\iota_\mathbf{n}} & \Ran(X, \multgrplus{m})_\mathbf{n} \ar[d]_{\iota_\mathbf{n}} \ar[r] & \gConf{m}{\mathbf{n}}(X) \ar[d] \\
	\multgrplus{m} \times X \ar[r]^<<<<<<\delta & \Ran(X, \multgrplus{m}) \ar[r] & \gConf{m}{\infty}(X)
}
\]
where $(\multgrplus{m} \times X)_\mathbf{n} \subset \multgrplus{m} \times X$ consists of copies of $X$ indexed by $\mathbf{d}$ where for at least one $k, 1\leq k\leq m$, we have $\mathbf{d}_k < \mathbf{n}_k$. Now,
\[
	\delta^! \circ \iota_{\mathbf{n}, *} \circ \iota_\mathbf{n}^! \circ \delta_? \simeq \iota_{\mathbf{n}, *} \circ \delta^! \circ \iota_\mathbf{n}^! \circ \delta_? \simeq \iota_{\mathbf{n}, *} \circ \iota_\mathbf{n}^! \circ \delta^! \circ \delta_? \simeq \iota_{\mathbf{n}, *} \circ \iota_\mathbf{n}^!,
\]
where the first equivalence is by base change~\S\ref{subsubsec:base_change_for_*}, the second by commutativity of the diagram, and the third by Proposition~\ref{prop:comalg_vs_comfact}. But now, $\iota_{\mathbf{n}, *} \circ \iota_\mathbf{n}^!$ (on the left side of the diagram) is easy to understand: for $\matheur{F} \in \Shv(X)^{\multgrplus{m}}$ we have
\[
	(\iota_{\mathbf{n}, *}(\iota_\mathbf{n}^! \matheur{F}))_{\mathbf{d}} \simeq \begin{cases}
		0, & \text{if } \mathbf{d}_k \geq \mathbf{n}_k, 1\leq k\leq m, \\
		\matheur{F}_{\mathbf{d}}, & \text{otherwise}.
	\end{cases}
\]

We have thus proved the following
\begin{prop}
\label{prop:taking_quotient_setting_some_d_to_0}
Let $\matheur{A} \in \ComAlgshriek(\Shv(X)^{\multgrplus{m}})$, then the commutative algebra $\matheur{A}' = \delta^!(\iota_{\mathbf{n}, *}(\iota_\mathbf{n}^!(\delta_?(\matheur{A}))))$ is obtained from $\matheur{A}$ by setting all components $\matheur{A}_\mathbf{d}$ to $0$, for all $\mathbf{d}$ such that $\mathbf{d}_k \geq \mathbf{n}_k$, for all $1\leq k\leq m$.
\end{prop}

\section{Cohomology of $\gConf{m}{\mathbf{n}}(X)$ as factorization cohomology}
\label{sec:cohomology_of_Z_m_n}

We will now use the theory developed so far to study cohomology of the spaces $\gConf{m}{\mathbf{n}}(X)$ (see Notation~\ref{notation:gConf_altogether}). As we have seen, the cohomology of $\gConf{m}{\mathbf{n}}(X)_+$ is equivalent, i.e. quasi-isomorphic, to that of $\Ran(X, \multgrplus{m})_\mathbf{n}$, and similarly for $\gConf{m}{\infty}(X)_+$ and $\Ran(X, \multgrplus{m})$. We can thus perform our manipulations on the $\Ran$ side.

Since this section denotes different (but related) objects with the same letter, differing only by the fonts used, let us quickly orient the readers with the notation. In general, we adopt the following convention: the ``fancier'' the object, the ``fancier'' the font used to denote it. We start with augmented unital algebras $\algOr{m}{\mathbf{n}}$ in $\Vect^{\multgr{m}}$, along with its augmentation ideal $(\algOr{m}{\mathbf{n}})_+$, and its Koszul dual $\liealgOr{m}{\mathbf{n}}$. These objects are very simple, as reflected by the font. !-pulling back to $X$ gives us corresponding objects in $\Shv(X)^{\multgr{m}}$, which we will denote by $\algOr{m}{\mathbf{n}}(X), (\algOr{m}{\mathbf{n}}(X))_+$, and $\liealgOr{m}{\mathbf{n}}(X)$ respectively. Taking factorization cohomology of $(\algOr{m}{\mathbf{n}}(X))_+$, we obtain $(\alg{m}{\mathbf{n}}(X))_+$, which is the augmentation ideal of $\alg{m}{\mathbf{n}}(X)$, and which has Koszul dual $\liealg{m}{\mathbf{n}}(X)$. These objects are more complicated, also reflected by the font. Indeed, $\alg{m}{\mathbf{n}}(X)$ is precisely the cohomology of $\gConf{m}{\mathbf{n}}(X)$. Finally, we use $\algb{m}{\mathbf{n}}(X)$ to denote the stabilization of $\alg{m}{\mathbf{n}}(X)$.

We will now briefly review the content of this section. We start, in~\S\ref{subsec:commutative_factorization_algebras_Zmn}, with the definition of the algebras $\algOr{m}{\mathbf{n}}$ and show that the factorization cohomology with coefficients in $\algOr{m}{\mathbf{n}}(X)$ computes the cohomology of $\gConf{m}{\mathbf{n}}(X)$. The algebras $\algOr{m}{\mathbf{n}}$ are of a very simple nature: they are the quotient of a symmetric algebra $\algOr{m}{\infty}$ by one monomial. In~\S\ref{subsec:densities_algOr}, we perform the simple computation of $\Lambda \otimes_{\algOr{m}{\infty}} \algOr{m}{\mathbf{n}}$ where coincidences readily manifest themselves. Via factorization cohomology, these coincidences get translated to the level of cohomology of the spaces $\gConf{m}{\mathbf{n}}(X)$. In~\S\ref{subsec:coLie_coalgebras_Zmn}, we compute the Koszul duals of the $\algOr{m}{\mathbf{n}}$'s. It is shown in~\S\ref{subsec:homological_stability_Zmn} that homological stability for $\gConf{m}{\mathbf{n}}(X)$ is now a direct consequence of Theorem~\ref{thm:homological_stability_coChev}. In~\S\ref{subsec:densities_at_infty}, building on the coincidences among the densities of $\alg{m}{\mathbf{n}}(X)$, we prove coincidences at the limits, i.e. of $\algb{m}{\mathbf{n}}(X)$. In~\S\ref{subsec:decategorifications_Zmn}, we decategorify these coincidences to get interesting identities, some of which are new, others are known previously but only as combinatorial coincidences. Finally, in~\S\ref{subsec:L_infty_algebras_computation}, through a simple $\Linfty$-algebra computation, we deal with the case of \Poincare{} polynomials, which is less well-behaved. Using the same technique, we obtain a simple expression for the stable homological density $\Lambda \otimes_{\algb{m}{\infty}(X)} \algb{m}{\mathbf{n}}(X)$.

Throughout this section, we fix an irreducible scheme $X$ of dimension $d$. In what follows, the algebras that we define will have their cohomological gradings depending on this $d$. We will, however, suppress it from the notation unless confusion is likely to occur.

\subsection{Commutative factorization algebras} 
\label{subsec:commutative_factorization_algebras_Zmn}
We will now produce commutative factorization algebras whose factorization cohomology groups compute the cohomology of $\gConf{m}{\mathbf{n}}(X)$.

\subsubsection{}
Let
\[
	\freeAlgOr{m} \in \ComAlg^{\un, \aug}(\Vect^{\multgr{m}})
\]
be the free (unital) commutative algebra generated by
\[
	\bigoplus_{k=1}^m \Lambda_{\unit_k}[-2d](-d).\teq\label{eq:generators_of_freeAlgOr}
\]
We write symbolically
\[
	\freeAlgOr{m} = \Lambda[x_1, \dots, x_m]
\]
where $x_k$ sits in graded degree $\unit_k$, cohomological degree $2d$, and $-d$ Tate twist. Its augmentation ideal is
\[
	(\freeAlgOr{m})_+ = \Lambda[x_1, \dots, x_m]_+ \in \ComAlg(\Vect^{\multgrplus{m}}).
\]

\subsubsection{} 
For an $m$-tuple of positive integers $\mathbf{n}$, let
\[
	\algOr{m}{\mathbf{n}} = \freeAlgOr{m}/(\sqcap_{k=1}^m x_k^{\mathbf{n}_k}) = \Lambda[x_1, \dots, x_m]/(\sqcap_{k=1}^m x_k^{\mathbf{n}_k}).
\]
In other words, $\algOr{m}{\mathbf{n}}$ is obtained from $\freeAlgOr{m}$ by setting $(\freeAlgOr{m})_{\mathbf{d}}$ to $0$ for all $\mathbf{d} \in \multgrplus{m}$ such that $\mathbf{d}_k \geq \mathbf{n}_k$ for all $1\leq k\leq m$. Similarly, we let $(\algOr{m}{\mathbf{n}})_+$ be the augmentation ideal of $\algOr{m}{\mathbf{n}}$.

\subsubsection{}
Let $\pi: X \to \pt$ be the structure map of $X$. We let
\[
	\freeAlgOr{m}(X) = \pi^! \freeAlgOr{m} = \omega_X \otimes \freeAlgOr{m} \quad\text{and} \quad \algOr{m}{\mathbf{n}}(X) = \pi^! \algOr{m}{\mathbf{n}} = \omega_X \otimes \algOr{m}{\mathbf{n}},
\]
where $\omega_X$ is the dualizable sheaf on $X$.

Note that the generators of $\freeAlgOr{m}(X)$ and $\algOr{m}{\mathbf{n}}(X)$ live in cohomological degrees $0$ and have no Tate twist, since the generators of $\freeAlgOr{m}$ and $\algOr{m}{\mathbf{n}}$ live in cohomological degrees $2d$ and have $-d$ Tate twist. However, since these are algebras with respect to the $\otimesshriek$-monoidal structure, their products do not live in cohomological degrees $0$ and have non-trivial Tate twists.

\subsubsection{} We will make use of the following
\begin{notation}
\label{notation:sOmega}
Let $S = \bigsqcup S_i$ be a possibly infinite disjoint union of irreducible schemes. We will use $\sOmega_S$ to denote the sheaf $\omega_S[-2\dim S](-\dim S)$, by which we mean the sheaf
\[
	\bigoplus \omega_{S_i}[-2\dim S_i](-\dim S_i).
\]
Note that when $S$ is smooth, $\sOmega_S \simeq \Lambda_S$.

More generally, when we shift a sheaf by some factor of the dimension of $S$, we mean on each component, we shift using the dimension of that factor. The same convention applies to Tate twists as well.
\end{notation}

\subsubsection{} Since $\pi^!$ is symmetric monoidal,
\[
	\freeAlgOr{m}(X)_+ \in \ComAlgshriek(\Shv(X)^{\multgrplus{m}})
\]
is the free (non-unital) commutative algebra generated by\footnote{When $X$ is smooth, the formula becomes simply $\bigoplus_{k=1}^m \Lambda_{\unit_k}$.}
\[
	\bigoplus_{k=1}^m (\sOmega_X)_{\unit_k}.
\]
Thus, by inspecting the free commutative algebra on $\Shv(\Ran(X, \multgrplus{m}))$ using the $\otimesstar$-monoidal structure, we get
\[
	\delta_? \freeAlgOr{m}(X)_+ \simeq \sOmega_{\Ran(X, \multgrplus{m})} \in \Factstar(X, \multgrplus{m}).
\]

By Proposition~\ref{prop:taking_quotient_setting_some_d_to_0}, we see immediately that
\[
	\delta_? \algOr{m}{\mathbf{n}}(X)_+ \simeq \iota_{\mathbf{n}, *} \iota_\mathbf{n}^! \sOmega_{\Ran(X, \multgrplus{m})} \simeq \iota_{\mathbf{n}, *} \sOmega_{\Ran(X, \multgrplus{m})_\mathbf{n}} \in \Factstar(X, \multgrplus{m}).
\]

The discussion above implies the following result, which links the cohomologies of $\gConf{m}{\mathbf{n}}(X)$ to factorization cohomology (see Notation~\ref{notation:gConf_altogether} for the notation, and~\S\ref{subsec:fact_cohomology} for the discussion of factorization cohomology).

\begin{prop}
\label{prop:factorization_cohomology_vs_Zmn}
We have
\[
	\pi_{?*} (\algOr{m}{\mathbf{n}}(X)_+) \simeq C^*(\gConf{m}{\mathbf{n}}(X)_+, \sOmega_{\gConf{m}{\mathbf{n}}(X)}).
\]
In particular, when $X$ is smooth, we have
\[
	\pi_{?*} (\algOr{m}{\mathbf{n}}(X)_+) \simeq C^*(\gConf{m}{\mathbf{n}}(X)).
\]
\end{prop}

\begin{notation}
We will use $\alg{m}{\mathbf{n}}(X)$ to denote the algebra $C^*(\gConf{m}{\mathbf{n}}(X), \sOmega_{\gConf{m}{\mathbf{n}}(X)})$, and
\[
	\alg{m}{\mathbf{n}}(X)_+ = \pi_{?*} (\algOr{m}{\mathbf{n}}(X)_+) \simeq C^*(\gConf{m}{\mathbf{n}}(X)_+, \sOmega_{\gConf{m}{\mathbf{n}}(X)_+}),
\]
including the case where $\mathbf{n}=\infty$.
\end{notation}

\begin{rmk}
This result generalizes that of~\cite{knudsen_betti_2017}. The algebra structure on $\alg{m}{\mathbf{n}}(X)$ can also be constructed geometrically in a similar way as in~\cite{knudsen_betti_2017}*{\S5.2}. We do not need this geometric description in the sequel.
\end{rmk}

\subsection{Densities}
\label{subsec:densities_algOr}
We will now perform an easy, but fundamental, computation, Proposition~\ref{prop:density_alg_Or}, from which coincidences observed at the level of cohomology are an easy consequence, Theorem~\ref{thm:dependence_on_mn_of_graded_quotients}.

Consider the following pushout in $\ComAlg^{\un, \aug}(\Vect^{\multgr{m}})$
\[
\xymatrix{
	\freeAlgOr{m} \ar[d] \ar[r] & \algOr{m}{\mathbf{n}} \ar[d] \\
	\Lambda \ar[r] & \Lambda \otimes_{\freeAlgOr{m}} \algOr{m}{\mathbf{n}}
} \teq\label{eq:density_alg_Or}
\]
Since $\algOr{m}{\mathbf{n}}$ is defined by one equation, a simple computation yields the following result.

\begin{prop}
\label{prop:density_alg_Or}
We have the following equivalence
\[
	\Lambda \otimes_{\freeAlgOr{m}} \algOr{m}{\mathbf{n}} \simeq \Lambda_{\mathbf{0}} \oplus \Lambda_{\mathbf{n}}[1-2d|\mathbf{n}|](-d|\mathbf{n}|).
\]
\end{prop}
\begin{proof}
Consider the following monoidal self-equivalence $\sheerRight$ of $\Vect^{\multgr{m}}$, where
\[
	\sheerRight(V)_\mathbf{d}= V_\mathbf{d}\left[-2d|\mathbf{d}|\right]\left(-d|\mathbf{d}|\right),
\]
and its inverse\footnote{Here, $\sheerRight$ and $\sheerLeft$ stand for sheering right and left respectively. Note also that the definitions of these functors depend on $d = \dim X$ we fixed throughout.}
\[
	\sheerLeft(V)_\mathbf{d} = V_\mathbf{d}\left[2d|\mathbf{d}|\right]\left(d|\mathbf{d}|\right).
\]

Observe that $\sheerLeft(\freeAlgOr{m})$ and $\sheerLeft(\algOr{m}{\mathbf{n}})$ are in the heart of the $t$-structure of $\Vect^{\multgr{m}}$, and hence, we can compute the tensor using classical means. Namely, we have the following two-step resolution of $\sheerLeft(\algOr{m}{\mathbf{n}})$
\[
\xymatrix{
	0 \ar[r] & \Lambda[x_1, x_2, \dots, x_m] \ar[r] & \Lambda[x_1, x_2, \dots, x_m] \ar[r] &  \Lambda[x_1, \dots, x_m]/(\sqcap_{k=q}^m x_k^{\mathbf{n}_k}) \simeq \sheerLeft(\algOr{m}{\mathbf{n}}),
}
\]
where the first non-zero term is shifted to graded degrees $\mathbf{n}$. This gives us
\[
	\Lambda \otimes_{\sheerLeft(\freeAlgOr{m})} \sheerLeft(\algOr{m}{\mathbf{n}}) \simeq \Lambda_{\mathbf{0}} \oplus \Lambda_{\mathbf{n}}[1].
\]

Sheering back using $\sheerRight$, we get the desired result.
\end{proof}

\subsubsection{}
\label{subsubsec:crushing_gradings_simple}
The computation above shows that
\[
	\add_!(\Lambda \otimes_{\algOr{m}{\infty}} \algOr{m}{\mathbf{n}}) \in \ComAlg(\Vect^{\graded})
\]
depends only on $|\mathbf{n}|$. We can amplify this fact, by taking factorization cohomology, to obtain a statement about the cohomology of $\gConf{m}{\mathbf{n}}(X)$.

\begin{thm}
\label{thm:dependence_on_mn_of_graded_quotients}
The relative tensor
\[
	\add_!(\Lambda \otimes_{\alg{m}{\infty}(X)} \alg{m}{\mathbf{n}}(X)) \in \ComAlg^{\un, \aug}(\Vect^{\gradedplus})
\]
depends only on $|\mathbf{n}|$.
\end{thm}
\begin{proof}
It suffices to show the statement for
\[
	\add_!(0 \sqcup_{\alg{m}{\infty}(X)_+} \alg{m}{\mathbf{n}}(X)_+). 
\]
But this is equivalent to
\begin{align*}
	\add_!(0 \sqcup_{\pi_{?*} \algOr{m}{\infty}(X)_+} \pi_{?*} \algOr{m}{\mathbf{n}}(X)_+)
	&\simeq \add_!(\pi_{?*}(0 \sqcup_{\algOr{m}{\infty}(X)} \algOr{m}{\mathbf{n}}(X))) \\
	&\simeq \add_! \pi_{?*}(\Lambda \otimes_{\algOr{m}{\infty}(X)} \algOr{m}{\mathbf{n}}(X))_+ \\
	&\simeq \pi_{?*}(\add_!(\Lambda \otimes_{\algOr{m}{\infty}(X)} \algOr{m}{\mathbf{n}}(X)))_+,
\end{align*}
where the last equivalence is due to the fact that $\add_!$ is continuous and symmetric monoidal. Now, we are done, by Proposition~\ref{prop:density_alg_Or}.
\end{proof}

\begin{rmk}
In~\S\ref{subsec:densities_at_infty}, we will give an explicit equivalence between the $\add_!(\Lambda \otimes_{\algOr{m}{\infty}} \algOr{m}{\mathbf{n}})$'s, 
and hence, also the $\add_!(\Lambda \otimes_{\alg{m}{\infty}(X)} \alg{m}{\mathbf{n}}(X))$'s for various $m$ and $\mathbf{n}$ with $|\mathbf{n}|$ fixed. See also Remarks~\ref{rmk:canonical_maps_witnessing_coincidences} and~\ref{rmk:consider_more_general_add}.
\end{rmk}

\subsection{$\coLie$-coalgebras} 
\label{subsec:coLie_coalgebras_Zmn}
Via Koszul duality, there exist $\coLie$-coalgebras $\liealgOr{m}{\mathbf{n}}$ and $\trivLieAlgOr{m}$ in $\Vect^{\multgrplus{m}}$ such that
\[
	\coChev \liealgOr{m}{\mathbf{n}} \simeq (\algOr{m}{\mathbf{n}})_+ \quad\text{and}\quad \coChev \trivLieAlgOr{m} \simeq (\freeAlgOr{m})_+,
\]
or equivalently,
\[
	\coChev^\un \liealgOr{m}{\mathbf{n}} \simeq \algOr{m}{\mathbf{n}} \quad\text{and}\quad \coChev^\un \trivLieAlgOr{m} \simeq \freeAlgOr{m}.
\]

\subsubsection{} Similar to the case of algebras, we use $\liealgOr{m}{\mathbf{n}}(X)$ and $\liealgOr{m}{\infty}(X)$ to denote $\pi^!\liealgOr{m}{\mathbf{n}}$ and $\pi^!\liealgOr{m}{\infty}$ respectively. From~\eqref{eq:coChev_vs_pi_?} and~\eqref{eq:push_forward_constant_coLie}, we see that
\[
	\alg{m}{\mathbf{n}}(X)_+ \simeq \coChev(\pi_*(\liealgOr{m}{\mathbf{n}}(X))) \simeq \coChev(\pi_* \omega_X \otimes \liealgOr{m}{\mathbf{n}})
\]
and similarly, that
\[
	\alg{m}{\infty}(X)_+ \simeq \coChev(\pi_* \omega_X\otimes \liealgOr{m}{\infty}).
\]

\subsubsection{} We will use $\liealg{m}{\mathbf{n}}(X)$ and $\liealg{m}{\infty}(X)$ to denote $\pi_* \omega_X \otimes \liealgOr{m}{\mathbf{n}}$ and $\pi_* \omega_X \otimes \liealgOr{m}{\infty}$ respectively. What we have discussed above amounts to the following

\begin{prop} \label{prop:conf_vs_cochev}
We have
\[
	\alg{m}{\mathbf{n}}(X)_+ \simeq \coChev \liealg{m}{\mathbf{n}}(X),
\]
including the case where $\mathbf{n}=\infty$.
\end{prop}

\subsubsection{}
The goal now is to compute $\liealgOr{m}{\mathbf{n}}$ and $\liealgOr{m}{\infty}$. Since $\freeAlgOr{m}$ is a free commutative algebra, the corresponding $\coLie$-coalgebra $\trivLieAlgOr{m}$ is thus abelian (i.e. trivial co-multiplication map). Moreover, the underlying object in $\Vect^{\multgrplus{m}}$ given by
\[
	\trivLieAlgOr{m} \simeq \bigoplus_{k=1}^m \Lambda_{\unit_k}[-2d+1](-d).
\]
Note that there is an extra cohomological shift to the left by $1$ compared to~\eqref{eq:generators_of_freeAlgOr}. This is because for a trivial $\coLie$-coalgebra $L$, $\coChev L \simeq (\Sym L[-1])_+$.

\begin{lem}
\label{lem:density_colie_side_Or}
The algebra $(\Lambda\otimes_{\freeAlgOr{m}} \algOr{m}{\mathbf{n}})_+$ is both a free and a trivial algebra. Hence,
\begin{align*}
	\coPrim[1]((\Lambda\otimes_{\freeAlgOr{m}} \algOr{m}{\mathbf{n}})_+) 
	&\simeq \coFree_{\coLie}(\Lambda_{\mathbf{n}}[2-2d|\mathbf{n}|](-d|\mathbf{n}|)) \\
	&\simeq \cotriv_{\coLie} (\Lambda_{\mathbf{n}}[2 -2d|\mathbf{n}|](-d|\mathbf{n}|)).
\end{align*}
\end{lem}
\begin{proof}
The string of equivalences comes from the first assertion, since Koszul duality exchanges free objects and trivial objects. But now, the first assertion comes from the fact that a free commutative algebra generated by a one dimensional vector space living in an odd cohomological degree is trivial.
\end{proof}

\begin{cor}
\label{cor:computation_lieAlgOr}
We have the following equivalence
\[
	\oblv_{\coLie} \liealgOr{m}{\mathbf{n}} \simeq \bigoplus_{k=1}^m \Lambda_{\unit_k}[-2d+1](-d) \oplus \Lambda_{\mathbf{n}}[2-2d|\mathbf{n}|](-d|\mathbf{n}|).\teq\label{eq:oblv_coLie_amn}
\]
Moreover, $\liealgOr{m}{\mathbf{n}}$ is equipped with the structure of a $d$-shifted strongly unital (see Definition~\ref{defn:d_shifted_unital}).
\end{cor}
\begin{proof}
The pushout diagram~\eqref{eq:density_alg_Or} gives us the following pushout diagram in $\Vect^{\multgrplus{m}}$
\[
\xymatrix{
	\bigoplus_{k=1}^m \Lambda_{\unit_k}[-2d+1](-d) \ar[d] \ar[r] & \oblv_{\coLie} \liealgOr{m}{\mathbf{n}} \ar[d] \\
	0 \ar[r] & \Lambda_{\mathbf{n}}[2-2d|\mathbf{n}|](-d|\mathbf{n}|)
}
\]
By graded-degree considerations, we get the computation of $\oblv_{\coLie} \liealgOr{m}{\mathbf{n}}$.

The $d$-shifted strongly unital structure is obtained by applying Koszul duality to the quotient map
\[
	\algOr{m}{\infty} \to \algOr{m}{\mathbf{n}}.
\]
\end{proof}

The computation above, coupled with Lemma~\ref{lem:shifted_strongly_unital_to_strongly_unital}, implies the following
\begin{cor}
\label{cor:computation_of_liealg}
We have
\begin{align*}
	&\quad\, \oblv_{\coLie} (\liealg{m}{\mathbf{n}}(X)) \\
	&\simeq \left(\bigoplus_{k=1}^m C^*(X, \omega_X) \otimes \Lambda_{\unit_k}[-2d+1](-d)\right) \oplus (C^*(X, \omega_X)\otimes \Lambda_{\mathbf{n}}[2-2d|\mathbf{n}|](-d|\mathbf{n}|))
\end{align*}
and
\[
	\liealg{m}{\infty}(X) \simeq \bigoplus_{k=1}^m(C^*(X, \omega_X) \otimes \Lambda_{\unit_k}[-2d+1](-d)),
\]
with trivial $\coLie$-structure.

Moreover, both $\coLie$-coalgebras are strongly unital.
\end{cor}

\begin{cor}
\label{cor:computation_of_graded_quotients}
We have the following equivalences
\[
	\Lambda \otimes_{\alg{m}{\infty}(X)} \alg{m}{\mathbf{n}}(X) \simeq \Sym(C^*(X, \omega_X) \otimes \Lambda_{\mathbf{n}}[1-2d|\mathbf{n}|](-d|\mathbf{n}|)).
\]
\end{cor}
\begin{proof}
A priori, $\coChev^{\un}$ rather than $\Sym$ should appear on the RHS of the equivalence above. However, the $\coLie$-coalgebra under consideration is trivial, by Lemma~\ref{lem:density_colie_side_Or}, and we are done.
\end{proof}

\begin{rmk}
Note that since $|\mathbf{n}| \geq 2$ (by our convention), $\oblv_{\gr}(\Lambda \otimes_{\alg{m}{\infty}(X)} \alg{m}{\mathbf{n}}(X))$ is $\Frob_q^{-1}$-summable, since the expression inside $\Sym$ is finite dimensional and has positive weight.
\end{rmk}

\begin{rmk}
In the case of $X = \mathbb{A}^d$, the computation of Corollary~\ref{cor:computation_lieAlgOr} allows us to compute $\gConf{m}{\mathbf{n}}(\mathbb{A}^d)$ easily, since $C^*(X, \omega_X) \simeq \Lambda[2d](d)$ has trivial co-algebra structure in this case (see also Remark~\ref{rmk:fact_coh_triv_coalg}). Thus,
\[
	\alg{m}{\mathbf{n}}(\mathbb{A}^d) \simeq \Sym\left(\bigoplus_{k=1}^m \Lambda_{\unit_k} \oplus \Lambda_{\mathbf{n}}[1-2d|\mathbf{n}| + 2d](-d|\mathbf{n}|+d)\right).
\]
\end{rmk}

\subsection{Homological stability}
\label{subsec:homological_stability_Zmn}
The computation of $\liealgOr{m}{\mathbf{n}}$ above allows us to deduce homological stability for $\gConf{m}{\mathbf{n}}(X)$.

\begin{thm}
\label{thm:homological_stability_Zmn}
Let $X$ be an irreducible scheme of dimension $d\geq 1$. For each $1\leq k\leq m$ and $c\geq 0$, there exists a natural map
\[
	\Ho^c(\gConf{\mathbf{d}}{\mathbf{n}}(X), \sOmega_{\gConf{\mathbf{d}}{\mathbf{n}}(X)}) \to \Ho^c(\gConf{\mathbf{d} + \unit_k}{\mathbf{n}}(X), \sOmega_{\gConf{\mathbf{d}+\unit_k}{\mathbf{n}}(X)})
\]
that is
\begin{enumerate}[(i)]
	\item an equivalence when $\mathbf{d}_k \geq 2c$, and injective when $\mathbf{d}_k = 2c-1$, when $d=1, m = 1, n = 2$ (the case of configuration spaces of a curve),
	\item an equivalence when $\mathbf{d}_k \geq c$, and injective when $\mathbf{d}_k = c - 1$ otherwise.
\end{enumerate}
\end{thm}
\begin{proof}
The result is now a direct consequence of Theorem~\ref{thm:homological_stability_coChev}, using the computation of Corollary~\ref{cor:computation_of_liealg}. 

In the first case (i.e. the case of configuration spaces of a curve), we can take $s=2, s_i = 0, \forall i$ in~\eqref{eq:inequality_slope}. For the other cases, with $n$ finite, we can take $s=1, s_i=0, \forall i$.

The case where $n = \infty$ is treated similarly, and in fact, simpler, since the $\coLie$-coalgebra under consideration is trivial.
\end{proof}

\subsection{Stable homological densities}
\label{subsec:densities_at_infty}
Theorem~\ref{thm:dependence_on_mn_of_graded_quotients} and Corollary~\ref{cor:computation_of_graded_quotients} show that the quotients
\[
	\Lambda \otimes_{\alg{m}{\infty}(X)} \alg{m}{\mathbf{n}}(X) 
\]
depend only on $|\mathbf{n}|$. In this section, we will show how this implies the same statement about the stable homology. Namely, we want to show the following

\begin{thm}
\label{thm:equivalence_at_limit}
We have a natural equivalence
\[
	\Lambda \otimes_{\algb{m}{\infty}(X)} \algb{m}{\mathbf{n}}(X) \simeq \Lambda \otimes_{\algb{1}{\infty}(X)} \algb{1}{|\mathbf{n}|}(X).
\]
In particular, the relative tensor depends only on $|\mathbf{n}|$.
\end{thm}

\subsubsection{}
Before proving the Proposition, we will need some preparation. Choosing a smooth point on $X$, and hence, a smooth point on $\gConf{\mathbf{d}}{\infty}(X)$ for each $\mathbf{d}$, we get a map of algebras
\[
	\alg{m}{\infty}(X) \to \Lambda[\multgr{m}],
\]
which is compatible with the unit map
\[
	\Lambda[\multgr{m}] \to \alg{m}{\infty}(X)
\]
induced by the structure map $\gConf{m}{\infty}(X) \to \pt$. 

The rough idea now is to construct a natural equivalence
\[
	\Lambda[\graded] \otimes_{\alg{m}{\infty}(X)} \alg{m}{\mathbf{n}}(X) \to \Lambda[\graded] \otimes_{\alg{1}{\infty}(X)} \alg{1}{|\mathbf{n}|}(X).
\]
Theorem~\ref{thm:equivalence_at_limit} is then obtained by taking
\[
	\colim = \lbar{(-)}: \Vect^{\multfil{m}} \to \Vect.
\]

\subsubsection{} We will now make this idea rigorous by using~\S\ref{subsec:change_of_gradings}. As in there, to keep the notation less cluttered, we will drop $\add_!$ from the notation. For example, whenever we have a map where the target is $\graded$-graded and the source is $\multgr{m}$-graded, $\add_!$ is implicitly applied to the source. 

\subsubsection{}
First, we have the following natural map of algebras
\[
	\Lambda[\multgr{m}] \to \Lambda[\graded]
\]
by sending all generators of the LHS to the generator of the RHS. Similarly, we have the following maps of algebras
\[
	\algOr{m}{\infty} \to \algOr{1}{\infty}
\]
and
\[
	\algOr{m}{\mathbf{n}} \to \algOr{1}{|\mathbf{n}|},
\]
which, respectively, induce maps of algebras
\[
	\alg{m}{\infty}(X) \to \alg{1}{\infty}(X)
\]
and
\[
	\alg{m}{\mathbf{n}}(X) \to \alg{1}{|\mathbf{n}|}(X).
\]

These maps, in turn, induce a map between algebras
\[
	\Lambda[\multgr{m}] \otimes_{\alg{m}{\infty}(X)} \alg{m}{\mathbf{n}}(X) \to \Lambda[\graded] \otimes_{\alg{1}{\infty}(X)} \alg{1}{|\mathbf{n}|}(X),
\]
which factors through
\[
	\Lambda[\graded] \otimes_{\alg{m}{\infty}(X)} \alg{m}{\mathbf{n}}(X) \to \Lambda[\graded] \otimes_{\alg{1}{\infty}(X)} \alg{1}{|\mathbf{n}|}(X). \teq\label{eq:map_of_algs_graded_level}
\]
Since $\Lambda\otimes_{\Lambda[\graded]}-$ of the map~\eqref{eq:map_of_algs_graded_level} is an equivalence, by Theorem~\ref{thm:dependence_on_mn_of_graded_quotients}, so is~\eqref{eq:map_of_algs_graded_level} itself, since $\Lambda\otimes_{\Lambda[\graded]}-$ is the functor of taking associated graded, by~\S\ref{subsubsec:assgr_as_tensoring}, which is conservative, by Lemma~\ref{lem:ass_gr_is_nice}. We thus obtain the following

\begin{prop}
\label{prop:graded_quotients_coincidence}
The natural map of algebras
\[
	\Lambda[\graded] \otimes_{\alg{m}{\infty}(X)} \alg{m}{\mathbf{n}}(X) \to \Lambda[\graded] \otimes_{\alg{1}{\infty}} \alg{1}{|\mathbf{n}|}(X)
\]
is an equivalence.
\end{prop}

\begin{proof}[Proof of Theorem~\ref{thm:equivalence_at_limit}]
We will now complete the proof of Theorem~\ref{thm:equivalence_at_limit}. Indeed, taking colimit along $\filtered$, we get
\[
	\Lambda \otimes_{\algb{m}{\infty}(X)} \algb{m}{\mathbf{n}}(X) \simeq \lbar{\Lambda[\multgr{m}] \otimes_{\alg{m}{\infty}(X)} \alg{m}{\mathbf{n}}(X)} \simeq \lbar{\Lambda[\graded] \otimes_{\alg{m}{\infty}} \alg{m}{\mathbf{n}}(X)},
\]
where the middle item is obtained by taking colimit along $\multfil{m}$ and the last one along $\filtered$. The first equivalence is due to the fact that $\lbar{(-)}$ along $\multfil{m}$ is, by Corollary~\ref{cor:adjunction_for_algebras_colim}, a left adjoint, and hence, commutes with pushouts of algebras. The second equivalence is due to Lemma~\ref{lem:add_!_vs_stabilization}.
\end{proof}

A couple of remarks are in order.
\begin{rmk}
\label{rmk:canonical_maps_witnessing_coincidences}
Observe that we have the following diagram
\[
\xymatrix{
	\algOr{m}{\infty} \ar[d] \ar[r] & \algOr{m}{\mathbf{n}} \ar[d] \\
	\algOr{1}{\infty} \ar[d] \ar[r] & \algOr{1}{|\mathbf{n}|} \ar[d] \\
	\Lambda \ar[r] & \Lambda \otimes_{\algOr{m}{\infty}} \algOr{m}{\mathbf{n}} \simeq \Lambda \otimes_{\algOr{1}{\infty}} \algOr{1}{|\mathbf{n}|}
}
\]
where each square is a pushout square in $\ComAlg(\graded)$.\footnote{One way to see that the squares are indeed pushouts is to use the same trick as in Proposition~\ref{prop:density_alg_Or}.} This provides us with a canonical map witnessing the coincidences observed in Proposition~\ref{prop:density_alg_Or}.

By tensoring (over $\Lambda[\graded]$) the map in Proposition~\ref{prop:graded_quotients_coincidence} with $\Lambda$, we get a canonical map witnessing the coincidences observed in Theorem~\ref{thm:dependence_on_mn_of_graded_quotients}.
\end{rmk}

\begin{rmk}
\label{rmk:consider_more_general_add}
Theorem~\ref{thm:equivalence_at_limit} allow us to produce a natural equivalence
\[
	\Lambda \otimes_{\algb{m}{\infty}(X)} \algb{m}{\mathbf{n}}(X) \simeq \Lambda \otimes_{\algb{m'}{\infty}(X)} \algb{m'}{\mathbf{n}'}
\]
whenever $|\mathbf{n}| = |\mathbf{n}'|$ by first passing through
\[
	\Lambda \otimes_{\algb{1}{\infty}(X)} \algb{1}{|\mathbf{n}|}(X).
\]

We can link the cases of $(m, \mathbf{n})$ and $(m', \mathbf{n}')$ directly by considering a more general form of the functor $\add_!$. Indeed, instead of only consider $\add: \multgr{m} \to \graded$, one can consider, more generally, homomorphisms $\add: \multgr{m} \to \multgr{m'}$ in an analogous way. We can then define the functor $\add_!$ by left Kan extension as usual, and the whole argument goes through unchanged. We do not need this in the sequel.
\end{rmk}

\begin{rmk}
The stable homological density
\[
	\Lambda \otimes_{\algb{m}{\infty}(X)} \algb{m}{\mathbf{n}}(X)
\]
is in fact computable and has a simple expression. This gives an alternative proof of Theorem~\ref{thm:equivalence_at_limit}. Since it relies on $\Linfty$-algebras, we postpone the actual computation to Proposition~\ref{prop:stable_homological_density}.
\end{rmk}
\subsection{Decategorifications}
\label{subsec:decategorifications_Zmn}
Using~\S\ref{subsec:decategorification_coChev}, we will now see how Theorem~\ref{thm:dependence_on_mn_of_graded_quotients} and Theorem~\ref{thm:equivalence_at_limit} immediately give us the coincidences observed in quotients of (graded) Euler characteristics, and when our scheme $X$ comes from pulling back $X_0$ over a finite field $\Fq$, also of virtual \Poincare{} series and $L$-series. Since \Poincare{} series are not additive, more care has to be taken. We will return to this in \S\ref{subsec:L_infty_algebras_computation}

As before, throughout this subsection, we will assume that $X$ is an irreducible scheme of dimension $d$.

\begin{prop}
\label{prop:graded_euler_Zmn_quotients}
The quotient of graded Euler characteristics
\[
	\frac{\chi^\gr(\alg{m}{\mathbf{n}}(X))}{\chi^{\gr}(\alg{m}{\infty}(X))} = \chi^{\gr}(\Lambda \otimes_{\alg{m}{\infty}(X)} \alg{m}{\mathbf{n}}(X)),
\]
and hence, depends only $|\mathbf{n}|$.

When $X$ is a pullback of a scheme $X_0$ over $\Fq$, then all the objects inside has the action of the geometric Frobenius $\Frob$. We have the same statement for 
\[
	\frac{\chi^\gr_{\Frob_q^{-1}}(\alg{m}{\mathbf{n}}(X), t)}{\chi^{\gr}_{\Frob_q^{-1}}(\alg{m}{\infty}(X), t)} = \chi^{\gr}_{\Frob_q^{-1}}(\Lambda \otimes_{\alg{m}{\infty}(X)} \alg{m}{\mathbf{n}}(X), t) = \prod_{x\in |X|}(1-q^{-d|\mathbf{n}| \deg x}t^{|\mathbf{n}|\deg x}). \teq\label{eq:densities_chi_gr_frob}
\]
\end{prop}
\begin{proof}
The equality comes directly from Corollary~\ref{cor:quotient_decat_graded_euler}. The fact that it depends only on $|\mathbf{n}|$ is a consequence of Theorem~\ref{thm:dependence_on_mn_of_graded_quotients}. The case of $\chi^{\gr}_F$ is treated the same way (see also Remark~\ref{rmk:action_of_T}).

The second equality involving $\chi^{\gr}_{\Frob_q^{-1}}$ follows from Proposition~\ref{prop:multiplicative_trace} and Proposition~\ref{prop:density_alg_Or}, using the fact that
\[
	\Lambda \otimes_{\alg{m}{\infty}(X)} \alg{m}{\mathbf{n}}(X) \simeq \pi_{?*}^{\un}(\Lambda \otimes_{\algOr{m}{\infty}(X)} \algOr{m}{\mathbf{n}}(X)).
\]
\end{proof}

Combining with Proposition~\ref{prop:trace_decat},  noting that
\[
	\Lambda \otimes_{\alg{m}{\infty}(X)} \alg{m}{\mathbf{n}}(X) \simeq \Lambda \otimes_{\Lambda[\multgr{m}]} \Lambda[\multgr{m}] \otimes_{\alg{m}{\infty}(X)} \alg{m}{\mathbf{n}}(X),
\]
where the LHS is $\Frob_q^{-1}$-summable (by the explicit computation at Corollary~\ref{cor:computation_of_graded_quotients}), we obtain the following

\begin{cor}
\label{cor:quotient_Frob_trace_infty}
When $X$ is as in the second part of Proposition~\ref{prop:graded_euler_Zmn_quotients}. Then
\[
	\chi_{\Frob_q^{-1}}(\Lambda \otimes_{\algb{m}{\infty}(X)} \algb{m}{\mathbf{n}}(X)) = \chi_{\Frob_q^{-1}}^{\gr}(\Lambda \otimes_{\alg{m}{\infty}(X)} \alg{m}{\mathbf{n}}(X), 1) = \zeta_{X_0}(d|\mathbf{n}|)^{-1}. \teq\label{eq:densities_chi_frob}
\]
\end{cor}

\begin{rmk}
\label{rmk:quotient_Frob_trace_infty}
When $\mathbf{n} = (n, n, \dots, n)$, evaluating $t=1$ and using the Grothendieck-Lefschetz trace formula, Theorem~\ref{thm:Grothendieck_Lefschetz_dual}, we see that the left most term of~\eqref{eq:densities_chi_gr_frob} captures the arithmetic densities
\[
	\lim_{\mathbf{d} \to \infty} \frac{|\gConf{\mathbf{d}}{\mathbf{n}}(X_0)(\Fq)|}{|\gConf{\mathbf{d}}{\infty}(X_0)(\Fq)|}
\]
and the middle term becomes
\[
	\chi_{\Frob^{-1}_q}^{\gr}(\Lambda \otimes_{\alg{m}{\infty}(X)} \alg{m}{\mathbf{n}}(X), 1) = \chi_{\Frob_q^{-1}}(\Lambda \otimes_{\algb{m}{\infty}(X)} \algb{m}{\mathbf{n}}(X)) = \zeta(d|\mathbf{n}|) = \zeta(dmn)^{-1}.
\]

Thus, equality~\eqref{eq:densities_chi_frob} categorifies (and generalizes) arithmetic density of Theorem~\ref{thm:arithmetic_densities}.
\end{rmk}

\subsubsection{}
Arguing similarly as Proposition~\ref{prop:graded_euler_Zmn_quotients}, we get
\begin{prop}
\label{prop:Poinc_vir_Zmn_quotients}
When $X$ is a pullback of a scheme $X_0$ over $\Fq$, then we have
\[
	\frac{\PoincVir(\algb{m}{\mathbf{n}}(X))}{\PoincVir(\algb{m}{\infty}(X))} = \PoincVir(\Lambda \otimes_{\algb{m}{\infty}(X)} \algb{m}{\mathbf{n}}(X)),
\]
and hence, only depends on $|\mathbf{n}|$.
\end{prop}
\begin{proof}
We argue similarly as the Proposition above, but using Theorem~\ref{thm:equivalence_at_limit} and Corollary~\ref{cor:quotient_decat_poinc_vir} in place of Theorem~\ref{thm:dependence_on_mn_of_graded_quotients} and Corollary~\ref{cor:quotient_decat_graded_euler}.
\end{proof}

\begin{rmk}
Note that all the terms appearing in Propositions~\ref{prop:graded_euler_Zmn_quotients} and~\ref{prop:Poinc_vir_Zmn_quotients} are readily available in explicit forms using Propositions~\ref{prop:graded_euler_decat} and~\ref{prop:poinc_vir_decat}.
\end{rmk}

\subsection{$\Linfty$-algebras, the case of \Poincare{} series, and stable homological densities}
\label{subsec:L_infty_algebras_computation}
As we have seen above, the various equivalences of algebras given in Theorem~\ref{thm:dependence_on_mn_of_graded_quotients} and Theorem~\ref{thm:equivalence_at_limit} allow us to systematically recover information about densities of Euler-characteristics and virtual \Poincare{} polynomials by decategorification. This comes from the fact that these invariants are additive, i.e. they behave nicely with respect to fiber sequences. \Poincare{} polynomials are, unfortunately, not additive, and only exhibit good behaviors when $\coLie$-coalgebras involved are abelian. In the first part of this subsection, we will find the conditions on $X$ so that $\liealgOr{m}{\mathbf{n}}(X)$ is abelian. In the second part, using the same technique, we obtain an expression for the stable homological density $\Lambda \otimes_{\algb{m}{\infty}(X)} \algb{m}{\mathbf{n}}(X)$.

Since it is more convenient to manipulate $\Lie$-algebras than $\coLie$-coalgebras, in this subsection, we will exclusively work with the linear dual $\liealg{m}{\mathbf{n}}(X)^\vee$, which is now a $\Lie$-algebra. We can take the double-dual to get back to the original $\coLie$-coalgebra, since everything is finite dimensional.

We start with the following dual description of Corollary~\ref{cor:computation_of_liealg}. We ignore Tate twists, since it is not relevant for the purposes of taking \Poincare{} polynomials.

\begin{lem}
\label{cor:computation_of_liealg_dual}
We have
\[
	\oblv_{\Lie}(\liealg{m}{\mathbf{n}}(X)^\vee) \simeq C^*_c(X) \otimes \liealgOr{m}{\mathbf{n}}^\vee \simeq \left(\bigoplus_{k=1}^m C^*_c(X) \otimes \Lambda_{\unit_k}[2d-1]\right) \oplus (C^*_c(X)\otimes \Lambda_{\mathbf{n}}[-2 + 2d|\mathbf{n}|]).
\]
\end{lem}

Here, we use $C^*_c(X)$ to denote $C^*_c(X, \Lambda) \simeq C^*(X, \omega_X)^\vee$, where the equivalence is due to Verdier duality. Recall that $C^*_c(X)$ lives in cohomological degrees $[0, 2d]$ in general. When $X$ is smooth affine, $C^*_c(X)$ lives in cohomological degrees $[d, 2d]$.

\subsubsection{} The $\Lie$-algebra structure on $\liealg{m}{\mathbf{n}}(X)^{\vee}$ is obtained from the $E_\infty$-algebra structure on $C^*_c(X)$ (via cup-product) and the $\Lie$-algebra structure on $\liealgOr{m}{\mathbf{n}}^\vee$ (see also~\cite{gaitsgory_study_2017}*{Vol. II, Chapter 6, \S1.2} and~\cite{ho_free_2017}*{Example. 4.2.8}). By homotopy transfer (\cite{loday_algebraic_2012}*{\S10.3}), we obtain the structures of graded $\Linfty$-algebra structures on
\[
	\Ho^*(\liealgOr{m}{\mathbf{n}}^\vee) \simeq \bigoplus_{k=1}^m \Lambda_{\unit_k}[2d-1] \oplus \Lambda_\mathbf{n}[-2+2d|\mathbf{n}|]
\]
and
\[
	\Ho^*(\liealg{m}{\mathbf{n}}(X)^{\vee}) \simeq \Ho^*_c(X) \otimes \Ho^*(\liealgOr{m}{\mathbf{n}}^\vee) \simeq \bigoplus_{k=1}^m \Ho^*_c(X)[2d-1]_{\unit_k} \oplus \Ho^*_c(X)[-2 + 2d|\mathbf{n}|]_{\mathbf{n}}
\]
and a $\Cinfty$-algebra structure on $\Ho^*_c(X)$. The $\Linfty$-algebra structure on $\Ho^*(\liealg{m}{\mathbf{n}}(X)^{\vee})$ comes from the $\Cinfty$- (resp. $\Linfty$-)algebra structure on $\Ho^*_c(X)$ (resp. $\Ho^*(\liealgOr{m}{\mathbf{n}}^\vee)$).

\begin{lem}
Suppose the cup product of any $|\mathbf{n}|$ classes in $\Ho^*_c(X)$ vanishes. Then, the $\Linfty$-algebra structure on $\Ho^*(\liealg{m}{\mathbf{n}}(X)^\vee)$ is trivial.
\end{lem}

\begin{proof}
By graded-degree considerations, we see that the $|\mathbf{n}|$-ary operations are the only non-trivial operations the $\Linfty$-algebras $\Ho^*(\liealgOr{m}{\mathbf{n}}^\vee)$ and $\Ho^*(\liealg{m}{\mathbf{n}}(X)^\vee)$ can have.\footnote{In fact, the one on $\Ho^*(\liealgOr{m}{\mathbf{n}}^\vee)$ is necessarily non-trivial since we know that $\liealgOr{m}{\mathbf{n}}$ is non-trivial. Indeed, its Koszul dual $\algOr{m}{\mathbf{n}}$ is not free.}
%\[
%	l_{|\mathbf{n}|}: \bigotimes_{k=1}^m \Lambda_{\unit_k}[2d-1]^{\otimes \mathbf{n}_k} \to \Lambda_{\mathbf{n}}[2d|\mathbf{n}| - \mathbf{n}].
%\]
%\[
%	l_{|\mathbf{n}|}: \bigotimes_{k=1}^m (\Ho^*_c(X)[2d-1]_{\unit_k})^{\otimes \mathbf{n}_k} \to \Ho^*_c(X)[2d|\mathbf{n}|-|\mathbf{n}|]_{\mathbf{n}},
%\]
Using the explicit formula for homotopy transfer~\cite{loday_algebraic_2012}*{Thm. 10.3.5}, we see that the resulting $|\mathbf{n}|$-ary operation on $\Ho^*(\liealg{m}{\mathbf{n}}(X)^\vee)$ involves the $|\mathbf{n}|$-ary operation of $\Ho^*(\liealgOr{m}{\mathbf{n}}^\vee)$ (since that's the only non-trivial operation $\Ho^*(\liealgOr{m}{\mathbf{n}}^\vee)$ has) together with iterated $2$-ary operations of $\Ho^*_c(X)$, namely the cup products. But these vanish, due to the current assumptions on $\Ho^*_c(X)$ and we are done.
\end{proof}

Using Proposition~\ref{prop:quotient_decat_Poincare} and Theorem~\ref{thm:equivalence_at_limit}, we arrive at the following generalization of~\cite{farb_coincidences_2019}*{Thm. 1.2}
\begin{prop}
\label{prop:densities_Poincare_Zmn}
Let $X$ be an irreducible scheme of dimension $d$ such that the cup-product of any $|\mathbf{n}|$ classes in $\Ho^*_c(X)$ vanishes. Then
\[
	\frac{\Poinc(\algb{m}{\mathbf{n}}(X))}{\Poinc(\algb{m}{\infty}(X))} = \Poinc(\Lambda \otimes_{\algb{m}{\infty}(X)} \algb{m}{\mathbf{n}}(X)). \teq\label{eq:densities_poinc_nice_case}
\]
In particular, the quotient depends only on $|\mathbf{n}|$.
\end{prop}

\subsubsection{}
In general, we do not have equality~\eqref{eq:densities_poinc_nice_case}. The RHS of~\eqref{eq:densities_poinc_nice_case}, which should be thought of as the ``real'' density, is explicitly computable. The following result computes the categorified homological density at the limit.

\begin{prop}
\label{prop:stable_homological_density}
We have the following equivalences
\[
	\Lambda\otimes_{\algb{m}{\infty}(X)} \algb{m}{\mathbf{n}}(X) \simeq \oblv_{\gr}(\Lambda \otimes_{\alg{m}{\infty}(X)} \alg{m}{\mathbf{n}}(X)) \simeq \oblv_{\gr} (\Sym(C^*(X, \omega_X) \otimes \Lambda_{\mathbf{n}}[1-2d|\mathbf{n}|](-d|\mathbf{n}|))).
\]
\end{prop}
\begin{proof}
Consider the pushout in $\coLie(\Vect^{\multgrplus{m}})$
\[
\xymatrix{
	\liealg{m}{\infty}(X) \ar[d] \ar[r] & \liealg{m}{\mathbf{n}}(X) \ar[d] \\
	\bigoplus_{k=1}^m \Lambda_{\unit_k}[1] \ar[r] &  \mathfrak{q}
}
\]
which corresponds to the quotient
\[
	\Lambda[\multgr{m}] \otimes_{\alg{m}{\infty}(X)} \alg{m}{\mathbf{n}}(X) = \coChev^{\un} \mathfrak{q}.
\]

The miracle here is that $\mathfrak{q}$ is abelian. Indeed,
\[
	\oblv_{\Lie} (\mathfrak{q}^\vee) \simeq \bigoplus_{k=1}^m \Lambda_{\unit_k}[-1] \oplus (C^*_c(X) \otimes \Lambda_\mathbf{n}[-2+2d|\mathbf{n}|]),
\]
and by graded-degree considerations, the only possibly non-trivial $\Linfty$-operations on $\Ho^*(\mathfrak{q}^\vee)$ are the $|\mathbf{n}|$-ary operations, which have cohomological degree $2 - |\mathbf{n}|$
\[
	l_{|\mathbf{n}|}: \bigotimes_{k=1}^m \Lambda_{\unit_k}[-1]^{\otimes \mathbf{n}_k} \to \Ho^*_c(X)[2d|\mathbf{n}| - |\mathbf{n}| ]_{\mathbf{n}}.
\]
Now, the highest cohomological degree of the RHS is $2d + |\mathbf{n}|(1-2d)$, which is always less than the cohomological degree of the LHS, $|\mathbf{n}|$, since $|\mathbf{n}| \geq 2$:
\[
	|\mathbf{n}| > 2d + |\mathbf{n}|(1-2d) \Leftrightarrow 0 > 2d(1-|\mathbf{n}|).
\]
This forces the operation to be trivial.

Altogether, we have
\begin{align*}
	\Lambda\otimes_{\algb{m}{\infty}(X)} \algb{m}{\mathbf{n}}(X) 
	&\simeq \lbar{\Lambda[\multgr{m}]} \otimes_{\lbar{\alg{m}{\infty}(X)}} \lbar{\alg{m}{\mathbf{n}}(X)} \\
	&\simeq \lbar{\Lambda[\multgr{m}] \otimes_{\alg{m}{\infty}(X)} \alg{m}{\mathbf{n}}(X)} \\
	&\simeq \oblv_{\gr}(\coChev^{\un} (\Quot_\un \mathfrak{q})) \tag{Proposition~\ref{prop:stable_homology_split_coLie}}\\
	&\simeq \oblv_{\gr}(\coChev^{\un}(0 \sqcup_{\liealg{m}{\infty}(X)} \liealg{m}{\mathbf{n}}(X))) \tag{by~\eqref{eq:quot_un_coLie_q}} \\
	&\simeq \oblv_{\gr}(\Lambda \otimes_{\alg{m}{\infty}(X)} \alg{m}{\mathbf{n}}(X)) \\
	&\simeq \oblv_{\gr} (\Sym(C^*(X, \omega_X) \otimes \Lambda_{\mathbf{n}}[1-2d|\mathbf{n}|](-d|\mathbf{n}|))). \tag{Corollary~\ref{cor:computation_of_graded_quotients}}
\end{align*}
\end{proof}

\section*{Acknowledgments}
This paper owes an obvious intellectual debt to the illuminating treatments of factorization homology by J. Francis, D. Gaitsgory, and J. Lurie in~\cite{gaitsgory_weils_2014,gaitsgory_atiyah-bott_2015,francis_chiral_2011}. The author would like to thank B. Farb and J. Wolfson for bringing the question of explaining coincidences in homological densities to his attention. Moreover, the author thanks J. Wolfson for many helpful conversations on the subject, O. Randal-Williams for many comments which greatly help improve the exposition, and G. C. Drummond-Cole for many useful conversations on $\Linfty$-algebras. Finally, the author is grateful to the anonymous referee for carefully reading the manuscript and for providing numerous comments which greatly helped improve the clarity and precision of the exposition.

This work is supported by the Advanced Grant ``Arithmetic and Physics of Higgs moduli spaces'' No. 320593 of the European Research Council and the Lise Meitner fellowship ``Algebro-Geometric Applications of Factorization Homology,'' Austrian Science Fund (FWF): M 2751.

%\printbibliography
\bibliography{densities}
\end{document}